\documentclass[letterpaper,11pt]{amsart}
\usepackage{amsmath}
\usepackage{amssymb,amsfonts,scalerel,dsfont, mathtools,bm, comment,float}
\usepackage{bbm}
\usepackage{hyperref}
\hypersetup{
	colorlinks,
	linkcolor={blue},
	citecolor={blue},
	urlcolor={blue}
}

\usepackage{tikz}
\usetikzlibrary{positioning, arrows.meta, shapes.geometric,decorations.pathreplacing,fit, calc}

\usepackage{mathtools,amssymb,amsthm}
\usepackage{xcolor}

\usepackage{helvet}

\RequirePackage{mathrsfs} \let\mathcal\mathscr

\numberwithin{equation}{section}

\newtheorem{theorem}{Theorem}[section]
\newtheorem{lemma}[theorem]{Lemma}
\newtheorem{proposition}[theorem]{Proposition}

\theoremstyle{definition}
\newtheorem*{ack}{Acknowledgments}
\newtheorem*{oa}{Open Access}
\newtheorem*{fund}{Funding}
\newtheorem{remark}[theorem]{Remark}
\newtheorem{definition}[theorem]{Definition}
\newtheorem{example}{Comparison}

\newtheorem{heuristic}{Heuristic}

\usepackage{graphicx}
\usepackage{amsmath}
\usepackage{amsthm}
\usepackage{amssymb}
\usepackage{enumerate}
\usepackage{amstext}
\usepackage{amsfonts}
\usepackage{graphicx}

\numberwithin{equation}{section}

\newcommand{\Z}{\mathbb{Z}}

\newcommand{\PP}{\mathbb{P}}
\renewcommand{\AA}{\mathbb{A}}
\newcommand{\FF}{\mathbb{F}}
\newcommand{\ZZ}{\mathbb{Z}}

\newcommand{\NN}{\mathbb{N}}
\newcommand{\QQ}{\mathbb{Q}}
\newcommand{\RR}{\mathbb{R}}

\renewcommand{\hat}{\widehat}

\newcommand{\x}{\mathbf{x}}
\newcommand{\y}{\mathbf{y}}

\newcommand{\uu}{\mathbf{u}}
\newcommand{\bb}{\mathbf{b}}

\newcommand{\Y}{\underline{\mathbf{Y}}}

\renewcommand{\v}{\mathbf{v}}
\renewcommand{\d}{\mathbf{d}}
\newcommand{\e}{\mathbf{e}}

\newcommand{\z}{\mathbf{z}}

\newcommand{\starsum}{\sideset{}{^*}\sum}

\newcommand{\ux}{{\vec{\mathrm{x}}}}
\newcommand{\ua}{\vec{\mathrm{a}}}
\newcommand{\ub}{\vec{\mathrm{b}}}
\newcommand{\ur}{\vec{\mathrm{r}}}

\newcommand{\uz}{{\vec{\mathrm{z}}}}
\newcommand{\uw}{{\vec{\mathrm{w}}}}
\newcommand{\uv}{{\vec{\mathrm{v}}}}
\newcommand{\uF}{{\vec{\mathrm{F}} }}

\newcommand{\uc}{{\vec{\mathrm{c}}}}
\newcommand{\ud}{\vec{\mathrm{d}}}
\newcommand{\ualf}{\vec{\alpha}}

\newcommand{\uy}{{\vec{\mathrm{y}}}}

\def\scrF{{\mathcal F}}

\def\vecnull{{\text{\boldmath$0$}}}

\def\GL{\operatorname{GL}}

\def\diag{{\mathrm{diag}}}

\mathtoolsset{showonlyrefs=true}
\newcommand{\scrm}{\mathfrak{m}}

\DeclareMathOperator{\meas}{meas}

\DeclareMathOperator{\Spec}{Spec}

\renewcommand{\hat}{\widehat}
\renewcommand{\bar}{\overline}

\renewcommand{\mod}{\:\text{mod}\:}

\newcommand{\ve}{\varepsilon}
\newcommand{\F}{{\underline{F}}}

\newcommand{\N}{\mathbb{N}}
\newcommand{\R}{\mathbb{R}}

\newcommand{\na}{\ua}

\newcommand{\nn}{\vec{\mathrm{n}}}

\newcommand{\nF}{\uF}

\begin{document}




\title[Two-dimensional delta symbol method]{A  two-dimensional delta symbol method and its application to pairs of quadratic forms}

 \author[J. Li]{Junxian Li}
 \address{Department of Mathematics\\
 University of California, Davis,
 1 Shields Avenue, Davis, CA 95616, U.S.A
 }
 \urladdr{https://orcid.org/0000-0003-0740-7404}
 \email{junxian@math.ucdavis.edu}

 \author[S. L. Rydin Myerson]{Simon L. Rydin Myerson}
 \address{Department of Mathematical Sciences, Chalmers University of Technology and the University of Gothenburg, 412 96 Gothenburg, Sweden}
 \urladdr{https://orcid.org/0000-0002-1486-6054}
 \email{myerson@chalmers.se}

 \author[P. Vishe]{Pankaj Vishe}
 \address{Department of Mathematical Sciences, Durham University, Durham, DH1 3LE, United Kingdom}
 \urladdr{https://orcid.org/0000-0002-7355-0615}
 \email{pankaj.vishe@durham.ac.uk}

\begin{abstract}
	We present a  two-dimensional delta symbol method that facilitates a version of the Kloosterman refinement of the circle method, addressing a question posed by Heath-Brown.  As an application, we establish the asymptotic formula for the number of integral points on a non-singular intersection of two integral quadratic forms with at least $10$ variables. Assuming the generalized Lindel\"of Hypothesis, we reduce the number of variables to $9$ by performing a double Kloosterman refinement. A heuristic argument suggests our  two-dimensional delta symbol will typically outperform known expressions of this type by an increasing margin as the number of variables grows.
\end{abstract}

\maketitle
 \setcounter{tocdepth}{1}
\tableofcontents

\section{Overview: quadratic forms and delta symbol methods}\label{sec:delta}
Consider a system of $R$ homogeneous degree $d$ integral forms in $s$ variables. We aim to derive an asymptotic formula for the number of integral zeroes of this system with size at most $P$, as $P$ approaches infinity. A straightforward heuristic suggests that there are roughly $P^{s-dR}$ integral solutions as long as $s>dR$. Indeed, such asymptotic formulae can be obtained using the circle method if $s$ is sufficiently large in terms of $d$ and $R$ and the forms are jointly non-singular.
Particularly sharp results are obtained when $R=1$ by an idea of Kloosterman, who took advantage of cancellations in certain complete exponential sums which arise from a
partition of $[0,1]$ into regions of scale $1/q$ centered at rationals of the form $a/q$.
 When $R=2$, a literal partition of $[0,1]^2$ of this kind would be impossible, since the distance between nearest rational neighbours of type $\ua/q$  not only depends on their respective denominators but also on the heights of affine lines that they belong to.
Our main technical result, Theorem~\ref{thm:delta}, in a certain sense comes as close as possible to such a partition, enabling a Kloosterman refinement of the circle method in dimension two. As an illustration of the strength of the method, we use it to derive an asymptotic formula for the number of integral points on non-singular intersections of two quadrics.

 \begin{theorem}\label{thm:application}
 	Let $F_1, F_2$ be two quadratic forms with integral coefficients in $s$ variables. 
    Suppose the projective variety defined by $F_1(\x)=F_2(\x)=0$ is non-singular of codimension 2. Let $w\in C_c^\infty (\RR^s)$. If $s$ is even and $s\geq10$, then for any $\Delta< 1/6$, we have 
 	\begin{align}\label{asymtwoquadrics}
 \operatornamewithlimits{\sum}_{\substack{\x\in \ZZ^s\\ F_1(\x)=F_2(\x)=0}}w\Big(\frac{\x}{P}\Big)=\mathfrak S \mathfrak I P^{s-4}+O(P^{s-4-\Delta}),
 	\end{align} 
 	where the singular series $\mathfrak S$ defined in \eqref{eq:SS} depends on $F_1, F_2$ and the singular integral $\mathfrak I$ defined in \eqref{eq:SI} depends on $F_1, F_2$ and $w$. The implicit constant in the error term depends on \(\Delta,F_1,F_2\) and \(w\).
    
    If $s$ is odd and $s\geq 11$, then \eqref{asymtwoquadrics} holds with $\Delta<1/15$. 
    Moreover, under the generalized Lindel\"of
    Hypothesis (GLH) for Dirichlet $L$-functions,~\eqref{asymtwoquadrics} also holds for $s=9$ with $\Delta<1/15$.
 \end{theorem}

As explained on the next page, the asymptotic formula \eqref{asymtwoquadrics} verifies the Manin--Peyre conjecture for nonsingular complete intersections of two quadrics with dimension at least $7$ (or $6$ under GLH). This improves an earlier result of Munshi~\cite{M}, which requires $s\geq 11$ with  $\Delta <1/32$.
Our condition $s\geq 10$
 matches that in the work of Heath-Brown--Pierce~\cite{HBP}, which handles the ``split" case when $F_i(\mathbf x,\mathbf y)=G_i(\mathbf x)+H_i(\mathbf y)$ where $G_i,H_i$ are quadratic forms in at least $5$ variables for $i=1,2$ with $\Delta<1/32$.
The conditional part of our result $s\geq 9$ matches the analogous result in the function field setting obtained by Vishe~\cite{V19}, where GLH follows from the generalized Riemann Hypothesis over finite fields. 

Before going into details of the method, we briefly discuss what is known and conjectured for smaller $s$. We phrase our discussion in terms of integer solutions of homogeneous equations. (See for example~\cite{JL} for more on the terminology of smooth proper models used in parts of the literature.)

Let $K$ be a number field and let $\mathcal O_K$ denotes the ring of integers in $K$. 
For nontrivial zeroes in $\mathcal O_K^s$ of two quadratic forms in $s$ variables over $K$, the Hasse Principle holds provided that $s\geq 8$ and the quadratic forms cut out a smooth projective variety by a result of Heath-Brown~\cite{HBHasse}, which  improves upon an earlier result of Colliot-Th\'{e}l\`ene--Sansuc--Swinnerton-Dyer~\cite{CSS}. In a recent work, Molyakov~\cite{Molyakov} proves the Hasse Principle for $\mathcal O_K$-points on the smooth part of the zero locus of two quadratic forms over $K$, provided $s\geq 8$ and the quadratic forms define a non-conical, geometrically irreducible projective variety. If we reduce the number of variables to $5$, which includes the cases of del Pezzo surfaces of degree 4 and Ch\^atelet surfaces, then the smooth Hasse principle can fail due to Brauer--Manin obstruction~\cite[Example 15.5]{CSS}. If the pair of quadratic forms define a smooth codimension $2$ variety and $s\geq 6$, then there is no Brauer--Manin obstruction and the Picard rank is $1$. Thus the Manin--Peyre conjecture predicts that \eqref{asymtwoquadrics} should hold in that case, at least with a weaker error term and after removing those solutions lying in a thin set~\cite{LST}.
The usual statement of the conjecture has the indicator function of a box in place of the smooth weight $w$; this can be recovered by allowing $w$ to approach the indicator function sufficiently slowly from above and from below. One could even preserve the power-saving error term if one checks that our implicit constants depends at most polynomially on bounds for the derivatives of $w$, but we do not pursue this here.

As mentioned before, the proof of Theorem \ref{thm:application} is based on a new version of the two-dimensional form of the circle method, which we present in Theorem~\ref{thm:delta} below. In order to describe this result and its context we introduce the $R$-dimensional delta symbol
\begin{equation}
    \label{eq:delta_def}
\delta_{\nn} =\begin{cases} 1&\text{if }\nn=\vec{0},\\0&\text{if }\nn\in\ZZ^R\setminus\{\vec{0}\}.\end{cases}
\end{equation}
When \(R=1\), this is written as \(\delta_n\) to detect if an integer $n$ is zero. 

\subsection{Variations on the circle method: a history}

The classical circle method from Hardy, Littlewood, Ramanujan and Vinogradov writes $\delta_n$ as $\int_{0}^1 e(\alpha n)d\alpha$. One then analyzes the quantity in question by breaking the unit circle into major and minor arcs. The major arc contributions are estimated asymptotically leading to the expected main terms, while the minor arc contributions are only estimated with an upper bound which should be in the error terms.  This type of argument allows one to obtain the \emph{analytic Hasse Principle} (for the exact meaning of this term see Arala~\cite{Arala}), when $s$ is sufficiently large. Building on the work of Davenport~\cite{Davenport}, Birch~\cite{Birch} established the analytic Hasse Principle for a smooth complete intersection defined by $R$ forms of degree $d$ in $s$ variables over $\QQ$ when $s>(d-1)2^{d-1}R(R+1)+R-1$. Improvements of Birch's result have been made in many cases. In particular, the quadratic dependency on $R$ has been reduced to linear in the work of Rydin Myerson~\cite{Myerson} by estimating the minor arc contributions using repulsions in exponential sums.  

Another type of refinement of the circle method, aimed at reducing the number of required variables, avoids the traditional splitting of the unit circle into major and minor arcs and instead treats the contributions of all arcs asymptotically. This imitates the method of Kloosterman~\cite{Kloosterman} who used
the Farey dissection of the unit circle to 
obtain the analytic Hasse Principle for representations of integers by positive definite quaternary quadratic forms. This bypasses the barrier encountered in the Hardy--Littlewood type circle method which requires at least five variables (see also Birch~\cite{Birch} for $R=1, d=2$). 

The advantage of Kloosterman's method is to make use of cancellations between complete exponential sums when averaged over different rationals with the same denominator. One can obtain a version of this Kloosterman refinement without appealing to the Farey dissection by using the delta symbol method, originated from the work of Duke-Friedlander-Iwaniec~\cite{DFI} and further developed by Heath-Brown~\cite{Heath-Brown96}. 
The delta symbol method essentially 
 provides a smooth partition of the unit circle (see Marmon-Vishe~\cite[Proposition 1.2]{Marmon_Vishe}) of the form 
 \begin{equation}
     \label{Marmon_Vishe}
 			\delta_{n}
			=
			\sum_{1\leq q\leq Q}
			\underset{\hphantom{a\bmod{q}}}{\starsum_{\ a\bmod{q}}}			
			\int_{-Q^\delta/qQ}^{Q^\delta/qQ}
			p_{q}(w)
			e((a/q+w)n)\,dw
			+O_{\delta,N}(Q^{-N}),
 \end{equation}
 for all \(n\in\ZZ\) and all \(\delta,N,Q>0\).
 This not only allows one to carry out the Kloosterman refinement more easily, by switching the sum and the integral to make use of averages over the $a$-sum, but also allows one to take advantage of the average over the \(q\)-sum. That is, one can use cancellations in averages over rationals with different denominators, performing a \emph{double} Kloosterman refinement, a key technique in Heath-Brown~\cite{Heath-Brown96} for studying integral solutions to a single quadratic form in at least $3$ variables. The delta symbol method has seen many other applications in recent years, such as shifted convolution problems and subconvexity estimates for $L$-functions (see e.g.~\cite{Munshi0, Munshi1, Munshi2, Munshi3}). There are also delta symbol methods over number fields by Browning--Vishe~\cite{BV} and over central simple division algebras over number fields (for example, quaternions over $\QQ$) by Arala--Getz--Hou--Hsu--Li--Wang~\cite{AGHHLW}.
 
\subsection{The higher-dimensional problem}\label{sec:higher}

It is natural to ask whether an analogue of the Kloosterman refinement of the circle method can be obtained in higher dimensions. This problem was posed by Heath-Brown~\cite{Heath-Brown96} for dimension two; we return to his question in Section~\ref{sec:heuristic}.

There have been various attempts of a  two-dimensional Kloosterman refinement over $\QQ$, including work of Browning--Munshi~\cite{BM}, Munshi~\cite{M}, Heath-Brown--Pierce~\cite{HBP} and Arala~\cite{Arala} for pairs of quadratic forms and Northey--Vishe~\cite{NV} for pairs of cubic forms.
The work of Pierce--Schindler--Wood~\cite{PSW} and Browning--Pierce--Schindler~\cite{BPS} even extends such a process to higher dimensions. 
The strategy in Heath-Brown--Pierce~\cite{HBP}, Pierce--Schindler--Wood~\cite{PSW} and Northey--Vishe~\cite{NV} starts with the setup in the classical circle method and uses  two-dimensional Dirichlet approximation with the same denominator to carry out Kloosterman refinement on the minor arc contributions. However, this is only possible when the underlying exponential sum is an absolute square since the arcs created by the Dirichlet approximation overlap. 

Alternatively, one may hope to use the one dimensional delta symbol to carry out Kloosterman refinement in higher dimensions. 
Browning--Munshi~\cite{BM} and Arala~\cite{Arala} considered the special case when one equation contains a binary quadratic form independent of the rest of the variables. The structure of the binary quadratic forms allows the detection of one equation with convolutions of Dirichlet characters and the remaining equation can then be detected using the one dimensional delta symbol. More generally, Munshi~\cite{M} applied a nested delta symbol to study smooth pairs of quadratic forms. However, the moduli in these results are larger than what one expects from a  two-dimensional Dirichlet approximation, thus making it less efficient in the Poisson summation step; see Comparison \ref{eg:GRH} for a more detailed discussion. 

Progress has been made over function fields as Vishe~\cite{V19} established a  two-dimensional Farey dissection of the unit square over $\FF_q(t)$, thus allowing a Kloosterman refinement in dimension two with optimal sizes of the denominators at the centers of the arcs in the dissection. This was applied to establish the analytic Hasse Principle for zero locus of any nonsingular pair of quadratic forms in at least $9$ variables over $\mathbb F_q(t)$ when $q$ is odd. Vishe's work has been generalized by Glas~\cite{Glas} to handle the nonsingular intersection of a cubic and a quadratic form over $\mathbb F_q(t)$. However, the non-Archimedean nature of the norm in positive characteristics was crucially used in~\cite{V19}, thus making it difficult to generalize the method there to the number field setting.

\subsection{A  two-dimensional delta symbol method}

In this paper, we develop a  two-dimensional version of the delta symbol method which provides an alternative approach to proving the analogous result of~\cite[Theorem 1.1]{V19} that facilitates a (double) Kloosterman refinement of the circle method in dimension two over $\mathbb Q$. We then apply it to study non-singular intersections of two quadratic forms over $\QQ$ in few variables. As the heuristic in Section~\ref{sec:heuristic} suggests, our version of the  two-dimensional delta symbol would typically outperform previous results due to the optimal size of the denominators in the center of the arcs in the smooth decomposition of the unit square. 

Before we state the main results, we need to set some notation. Throughout, we will write \(|\,\cdot\,|\) for the standard Euclidean norm on $\RR^n$. Let \(\delta_{\nn}\) be as in \eqref{eq:delta_def}.
  We use \(A \ll_{a,\dotsc,z} B\) or \(A = O_{a,\dotsc,z}(B)\) to mean that \(|A|<C|B|\) for some implicit constant \(C\) depending only on the parameters \(a,\dotsc,z\). (In later sections we slightly relax this convention, see  Section~\ref{sec:notation} below.)  We will use $A\asymp B $ to denote $B\ll A \ll B$. 
We write $e(x):=\exp(2\pi i x)$ and $e_q(x):=\exp(2\pi i x/q)$. 
When we write a sum with an asterisk, as in \(\sum_{\ua\mod q}^\ast\), it indicates that \(\gcd(\ua,q)=1\). Given \(q\in \N,\ua\in \Z^2\) with \(\gcd(\ua,q)=1\), a central figure in our analysis will be the lattice  \(\Lambda(\ua,q)\) defined as 
\begin{equation}\label{eq:Lambdadef}
\Lambda(\ua,q):=\{k\ua+q\uy: k\in\Z,\uy\in\Z^2\}.
\end{equation}
Throughout this work, we fix smooth, non-negative functions $\omega, \omega_0$ on $\R$, with the following properties. The function $\omega$ has support on $(1/2, 1)$ and $\int \omega (x)dx=1$ holds;
see~\cite[section~3]{Heath-Brown96} for an explicit example.  Meanwhile $\omega_0$ is even, satisfies $\omega_0(0)=1$ and $\int \omega_0(x)dx=1$, has support on $(-1/2,1/2)$, and $\omega_0(|x|^{1/2})$ is a smooth function. For instance, one function with all these properties is the triple-bump function given by
$$ f(x)=\begin{cases}e^{1-1/(1-(4x)^4)}&\textrm{ if }|x|<1/4,\\
(1-\beta)
\omega(2|x|)&\textrm{ otherwise,}
\end{cases}$$
with $\beta =
\int_{-1/4}^{1/4} e^{1-1/(1-(4y)^4)}\,dy.$ This 
 is non-negative, since  $\beta \approx 0.4<1$. Therefore, this is an admissible choice for $\omega_0$.

With these notations, we summarize our version of the  two-dimensional delta symbol method in the following theorem, proved in Section~\ref{sec:proof-delta}.

\begin{theorem}\label{thm:delta}
		Let $\nn \in \mathbb Z^2$ and let $Q\geq 1$ be a large parameter.
	For each \(q\in \N,\ua\in \Z^2\) with \(\gcd(\ua,q)=1\),
	 there exists a function $p_{\Lambda(\ua,q)}$ on $\RR^2$, defined in \eqref{eq:p2_p_2} below, which depends only on the lattice \(\Lambda(\ua,q)\) and the choice of functions \(\omega\) and \(\omega_0\) above, such that
	\begin{equation}\label{eq:delta}
			\delta_{\nn}
			=
			\sum_{1\leq q\leq Q}
			\underset{\hphantom{\ua\bmod{q}}}{\starsum_{\ua\bmod{q}}}			
			\int_{\R^2}
			p_{\Lambda(\ua,q)}(\uw)
			e((\ua/q+\uw)\cdot\nn)\,d\uw
	+O_{N}(Q^{-N}),
	\end{equation}
	for any $N>0$.
Here the function
\(p_{\Lambda(\ua,q)}(\uw)\)  
satisfies 
\begin{equation}
    	\label{eq:p1bound1}
	p_{\Lambda(\ua,q)}(\uw)
	= 1+O_{N, \delta}(Q^{-N})
	\quad \text{ if } q< Q^{1/2-\delta},
	|\uw| < q^{-1} Q^{-1-\delta},
\end{equation}
for any $\delta>0$. More explicitly, we have the following decomposition for $p_{\Lambda(\ua,q)}:$ $$
p_{\Lambda(\ua,q)}(\uw)
=
 p_{1, q}(\uw)
 +	\sum_{\substack{\ur\in \Lambda(\ua,q)\\ k=\gcd(
 		\ur)/\gcd(\ur, q)}}\omega\big(\frac{\ur}{Q^{1/2}}\big) p_{2,\ur,k,q}(\uw),
 $$
where the functions $p_{1, q}$ and $p_{2,\ur,k,q}$, defined in Lemma~\ref{lem:def_of_p_i}, satisfy
	\begin{align}\label{eq:p1bound2}
	p_{1, q}(\uw)&\ll _N \frac{Q}{q(1+|\uw|Q^{3/2})}(1+|\uw|qQ^{1/2})^{-N},\\
 \label{eq:p1bound3}
p_{2,\ur,k,q}(\uw)
&\ll_{N}
(1+|\uw|kqQ^{1/2}+|\uw\cdot \ur^\perp|Q)^{-N}.
	\end{align}
Moreover, we may interchange the sums over $\ua\bmod q$ and $\ur\in\Lambda(\ua,q)$ to write
	\begin{align} \label{eq:delta0} 
			\delta_{\nn}
		&=\sum_{1\leq q\leq Q}\,\,\starsum_{\ua \bmod q}\int_{\mathbb{R}^2}p_{1,q}(\uw)e((\ua/q+\uw)\cdot \nn)d\uw\\
\MoveEqLeft[4]
+\sum_{\substack{d,k\in \NN\\ \uc \in \ZZ^2 \mathrm{ primitive }\\ \ur=dk\uc}}\omega\big(\frac{\ur}{Q^{1/2}}\big)\sum_{\substack{1\leq q\leq Q/k\\d\mid q\\ \gcd(q/d, k)=1}}\,\,\starsum_{\substack{\ua \bmod q\\q\mid d\uc \cdot\ua^\perp}}\int_{\R^2}p_{2,\ur,k,q}(\uw)e((\ua/q+\uw)\cdot \nn)d\uw.
	\end{align}
\end{theorem}  

\begin{remark}\label{rem:decay}
	With an additional application of the geometry of numbers, one can show that the bound in \eqref{eq:p1bound2} is also satisfied by $p_{\Lambda(\ua, q)}$. However, we will not delve into this detail, as 
    we use a slightly rephrased version of \eqref{eq:delta0} stated in Proposition~\ref{prop:delta2} below, together with Lemma~\ref{lem:nonstationary} on the derivatives of  $p_{2,\ur,k,q} $, in the application to Theorem~\ref{thm:application}.  
\end{remark}

\begin{remark}
    \label{rem:HBQ}
    One might expect the functions $p_{\Lambda(\ua, q)}$ to be independent of $\ua$.  In fact, the first sum on the right hand side of \eqref{eq:delta0} is in this {\em ideal} form. However, due to geometry of numbers in $\RR^2$, this gives rise to accumulation around rationals which lie on {\em affine lines} of smaller height, which needs to be effectively canceled by the second sum on the right hand side of \eqref{eq:delta0}. See 
\cite[comments after (1.17)]{HBP} for another perspective on this difficulty.

\end{remark}

Let us briefly discuss the implications of Theorem~\ref{thm:delta}. 
In light of the expression
\begin{equation*}
\delta_{\nn}=\int_{[0,1]^2}e(\uw\cdot\nn)d\uw,
\end{equation*} 
we see that \eqref{eq:delta} can be viewed as a smooth partition of unity on $[0,1]^2$, with smooth functions $p_{\Lambda(\ua, q)}$ placed around reduced fractions $\ua/q$ for $q\leq Q$. This should be compared to \eqref{Marmon_Vishe} in the one-dimensional setting. The function $p_{\Lambda(\ua, q)}(\ualf-\ua/q)$ is roughly supported in the ball $\{\ua/q+\uw:|\uw|\ll  q^{-1}Q^{-1/2+\delta}\}$, where we will choose $\delta $ to be very small. Note that this support is as small as possible, in the sense that if we reduced the radius by $Q^\delta$, these balls would no longer cover $[0,1]^2$.

Using properties \eqref{eq:p1bound1} and \eqref{eq:p1bound3}, one may obtain an asymptotic formula for the number of integral points on the non-singular intersection of two quadratic forms in at least $13$ variables. To improve this result, we observe that the dependence of the function $p_{\Lambda(\na, q)}(\uw)$ on $(\na, q)$ is only through the lattice $\Lambda(\ua,q)$. In particular, as shown in Lemma~\ref{lem:lattice} below, for any $\gcd(\ua,q)=1$ and $(\lambda, q)=1$, we have $\Lambda(\ua, q)=\Lambda(\lambda\ua, q)$. Thus we can write 
\begin{equation*}
    \delta_{\nn}
			=
			\sum_{1\leq q\leq Q}
			\underset{\hphantom{\ua\bmod{q}}}{\starsum_{\ua\bmod{q}}}\frac{1}{\phi(q)}		
			\int_{\R^2}
			p_{\Lambda(\ua,q)}(\uw)\starsum_{\lambda\bmod q}	
			e((\lambda\ua/q+\uw)\cdot\nn)\,d\uw
			+O_{N}(Q^{-N}).
\end{equation*}
This is the key property that allows us to obtain  extra cancellations by averaging exponential sums on a set of the form
$\{\frac{\lambda\ua}{q}+\uw:\gcd(\lambda,q)=1\}$ for fixed $\uw$, thereby
 carrying out a version of Kloosterman refinement in dimension two. The function $p_{\Lambda(\ua,q)}$  is essentially supported on $|\uw|\asymp \frac{1}{qQ^{1/2}}$ and thus we can choose $Q\asymp \max\{|\nn|^{2/3}\}$ for applications. This size of $Q$ is smaller than previous results and would typically be more advantageous when applying dual summation formulae. It is worth noting that Theorem \ref{thm:delta} has applications beyond Diophantine problems. In particular, a recent work of Leung--Pandey~\cite{LP} employs our Theorem \ref{thm:delta} to obtain power saving asymptotic formula for the divisor function along binary quartic forms.

\subsection{Outline of this paper}  We begin by setting some further notation which will be used throughout this paper in Section~\ref{sec:notation}.
With the key duality Lemma~\ref{lem:duality} and analytic properties of the $p$-functions in Lemma~\ref{lem:def_of_p_i} established in Section~\ref{sec:setup}, we prove Theorem~\ref{thm:delta} in Section~\ref{sec:delta proof} using properties of the lattice $\Lambda(\ua, q)$. We provide a heuristic comparison of Theorem \ref{thm:delta} with existing $\delta$-methods in Section~\ref{sec:heuristic}.

\begin{figure}[h]
    \centering
\begin{tikzpicture}[scale=0.85, every node/.style={transform shape},
	plain/.style={rectangle, draw=black!80, fill=white, rounded corners, minimum width=2cm, minimum height=0.5cm, align=center},
	arrow/.style={-latex, thick},
	every path/.style={rounded corners=8pt}
	]
	
	\node[plain] (thm1) at (7,0) {Theorem \ref{thm:application}};
	\node[plain] (thm12) at (0,0) {\S\ref{sec:delta proof}/Theorem \ref{thm:delta}};

	\node[plain] (lem52) at (3.5,0) {Lemma \ref{lem:N1P}\\\(\text{circle method}\)};
	\node[plain] (lem53) at (7,-2) {Lemma \ref{lem:Ni-bounds}, \(\text{$N_{i}{\ll} 
    \dots$}\)};
	\node[plain] (lem112) at (7,-8) {\S\ref{sec:Qdelta-and-Sigma}/Lemma \ref{N2splitbounds}\\$N_{2,i}\ll\dots$\\Definition of $\mathcal Q_\delta$, $ \Sigma$};
    
	\node[plain] (s2) at (-3.5,0) {\S\ref{sec:notation}\\notation, $h$};
	\node[plain] (s3) at (-3.5,-2) {\S\ref{sec:setup}\\duality, $p_i$};
	\node[plain] (s8) at (3.5,-2) {\S\ref{sec:major arc}, $N_0$\\\(\text{major arcs}\)};
	\node[plain] (s67) at (0,-2) {\S\ref{sec:exp integral}, \S\ref{sec:exp sum}\\\(\text{exponential}\)\\\(\text{sums and}\)\\\(\text{integrals}\)};
	\node[plain] (s10) at (7,-4) {\S\ref{sec:N1}\\\(\text{``Dirichlet"}\)\\minor arcs $N_1$};
	\node[plain] (s132) at (3.5,-4) {\S\ref{n2prelim}\\\(\text{averages}\)};
	\node[plain] (s9) at (7,-6) {\S\ref{sec:counting}\\\(\text{counting}\)};
	\node[plain] (s13) at (3.5,-6) {\S\ref{sec:minor N21}\\\(\text{``nonstandard''}\)\\arcs $N_2$: $\uu\neq \vecnull$};
	\node[plain] (s12) at (0,-8) {\S\ref{sec:V<1}\\\(\text{``nonstandard''}\)\\arcs $N_2$: $\uu=\vecnull$};

	\draw[arrow] (thm12) -- (lem52);
    \draw[arrow] (lem52) -- (thm1);
    \draw[arrow] (s8) -- (lem53);
    \draw[arrow] (s10) -- (lem53);
    \draw[arrow] (lem53) -- (thm1);
    \draw[arrow] (s13) -- ($(lem112.north west)+(1.5pt,-1pt)$);
    \draw[arrow] ($(lem112.north east)+(-2pt,-1pt)$) [sharp corners] -| ($(lem53.south east)+(-5pt,0pt)$);

    \draw[arrow] (s67) -- (s8);
    \draw[arrow] (s67) -- (s10);
    \draw[arrow] (s67) -- ($(s132.north west)+(1.5pt,-1pt)$);
    \draw[arrow] (s67) -- (s12);

    \draw[arrow] ($(s9.north west)+(2pt,-1pt)$) -- ($(s132.south east)+(-1.5pt,1pt)$);
    \draw[arrow] (s9) -- (s10);
    \draw[arrow] (s9) -- (s13);
    
	\draw[arrow] (s2) -- (s3);
    \draw[arrow] (s2) -- (thm12);
    \draw[arrow] (s3) -- ($(thm12.south west)+(2pt,1pt)$);
    \draw[arrow] ($(s12.east)$) -- (lem112);
    \draw[arrow] (s132) -- (s13);
    \draw[arrow] (s3) -- (s67) node[midway,below,align=center] { $\bigg|$\\\\Lemma~\ref{lem:nonstationary}};

\end{tikzpicture}
 \caption{Structure of the paper}
    \label{fig:summary-graph}
\end{figure}

The rest of the paper is dedicated to proving Theorem~\ref{thm:application}. In Section~\ref{sec:the-Ni}, we reduce the proof of Theorem~\ref{thm:application} to Lemma~\ref{lem:Ni-bounds} after an application of Proposition~\ref{prop:delta2}, which is essentially a restatement of Theorem~\ref{thm:delta}. We also recall some known geometric properties of a non-singular complete intersection variety defined by two quadrics which provide guidelines for our estimations for the minor arc contributions. Section~\ref{sec:exp integral} is dedicated to the exponential integral bounds. 
We begin by re-interpreting known bounds for quadratic exponential integrals to our setting and finish with a key result Lemma~\ref{p1derivative}, which plays a critical role in carrying out a \emph{double} Kloosterman refinement in the case $s=9$. Section~\ref{sec:exp sum} is dedicated to the exponential sum estimates, most of which are either known or are direct analogues of their known function field counterparts. An asymptotic formula for the major arcs contribution is obtained in Section~\ref{sec:major arc}. After the preparations in Section~\ref{sec:counting}, we prove sufficient bounds for the minor arcs contributions coming from the $p_1$ and $p_2$ functions in Section~\ref{sec:N1} and Sections~\ref{sec:preparing-N2}-\ref{sec:minor N21}. See Figure \ref{fig:summary-graph} for the connections between these sections. While the final optimization in the extreme cases follows closely that in~\cite{V19}, the growth of our functions, for instance in \eqref{eq:p1bound2}, requires special care in non-extreme cases (small/medium $q$ and $\uw$).

\section{Basic notation and \texorpdfstring{$h$}{h}-function}\label{sec:notation}
 We first begin by setting notation that will be throughout. We also draw the reader's attention to Section~\ref{sec:quad-forms-basic-notation}, where much additional terminology is introduced which is used from that point on, and to Section~\ref{sec:Qdelta-and-Sigma}, where the expressions $\mathcal Q_\delta$ and $\Sigma(\cdots)$ are defined.
 
 The notation  \(|\,\cdot\,|\), $\delta_{\nn}$, \(A\ll B\), $A\asymp B$, $e(x)$, $e_q(x)$, \(\sum_{\ua\mod q}^\ast\) and \(\Lambda(\ua,q)\) will always be as defined before Theorem~\ref{thm:delta}.
 We use $\ua, \ub, \uc, \ur $ to denote vectors in $\mathbb Z^2$ and $\mathbf{x}, \mathbf{u}, \dots $ to denote vectors in $\mathbb Z^s$, where $s$ will denote the number of variables needed to define the quadratic forms appearing in the statement of Theorem~\ref{thm:application}.
Set $\mathds 1_{S}=1$ if $ S$ is true and 
$0$ otherwise. In order to reserve the symbol $\cdot$ for the scalar product of two vectors, we use $\times$ where necessary to denote multiplication in a field. 

Recall that we fix smooth functions $\omega$ and $\omega_0$ on $\RR$ as described before Theorem \ref{thm:delta}.
 Additionally, from Section~\ref{sec:the-Ni} onward, we fix a weight \(w\in C_c^\infty(\mathbb R^s)\), fix a pair of quadratic forms \(F_1,F_2\) in \(s\) variables, and use \(\ve\) to represent an arbitrarily small positive constant whose value may vary from line to line. All implicit constants in \(\ll, O(\,\cdot\,)\) and \(\asymp\) notation are then allowed to depend on \(s,w,F_1,F_2\) and \(\ve\).

We define
 \begin{equation}
 	\label{eq:cdef}
 	c = (\int_{\R^2} \omega(|\ux|)\,d\ux)^{-1}=(2\pi\int_\R r\omega(r)\,dr)^{-1}.
 \end{equation}
Throughout this paper \(c\) always denotes this particular constant in \eqref{eq:cdef}.

 For each 2-vector \(\ux=(x,y)\) we write \(\ux^\perp = (y,-x)\). We write \(\partial_{\vec{\xi}} = \frac{\xi_1}{|\vec\xi|}\frac{\partial}{\partial w_1}+ \frac{\xi_2}{|\vec\xi|}\frac{\partial}{\partial w_2}\) to denote the normalized directional derivative (with respect to $\uw$) along \(\vec{\xi}\).

 Let $\Lambda\subset \RR^2$ be a lattice, where  we  will   typically take \(\Lambda=\Lambda(\ua,q)\). We shall study the shortest non-zero vector of the lattice \(M\Lambda=\{M\uv:\uv\in \Lambda\}\). More precisely, if \(M\) is an \(m \times 2\) real matrix with full rank, we define
 \begin{equation*}
 	\mu_{M}:=\mu_M(\Lambda)=\min\{|M\ux|:\ux\in \Lambda\setminus\{\vec{0}\}\}.
 \end{equation*}
 In other words \(\mu_M(\Lambda)\) is the Euclidean norm of the shortest nonzero vector of the lattice \(M\Lambda\).

Given a function \(\omega_1\in C^\infty (\RR)\) with support on $(1/2,1)$, we define a function
\begin{equation}
	\label{eq:hyzdef}
	h_{\omega_1}(y,z)
	=
	\sum_{j\in\N}
	\frac{1}{yj}\left(\omega_1(yj)-\omega_1\left(\frac{|z|}{yj}\right)\right).
\end{equation}
We will usually take $\omega_1=\omega$, and so we write $h(y,z)$ as an abbreviation of $h_{\omega}(y,z)$. This function
appears in the one dimensional version of the delta symbol method in~\cite{Heath-Brown96} and will appear in the definition of $p_{\Lambda(\ua, q)}$. 
However we use $\omega_1(x)=x\omega(x)$ in the proof of Lemma~\ref{lem:1}, where we write $h_{2}(y,z)=h_{x\omega(x)}(y,z)$. 

We need a variant of~\cite[Lemmas 4 and 9]{Heath-Brown96} on the properties of the $h$-function. 

\begin{lemma}\label{lem:HB}
	Let $\omega_1$ be a
    smooth function supported on $(1/2, 1)$. Recall the function $h_{\omega_1}(y,z)$ from \eqref{eq:hyzdef}.
    \begin{enumerate}
        \item 
    The value of $h_{\omega_1}(y,z)$ is non-zero only if  $y\leq \max (1, 2|z|)$. Moreover $h_{\omega_1}(y,z)\ll_{\omega_1} 1/y$ for all $z$.
    \item
    Suppose that \(A,B, \delta>0\) and that \(f:\R\to \R\) satisfies \(|f^{(k)}(z)|\leq C_k A^k\) for all \(z\in \R\) and some sequence of positive real numbers \(\boldsymbol{C}=(C_k)_{k\geq 0}\). 
    Then for \(0<y\leq Q^{-\delta}\min\{B/A,1\}\) we have
	\[
	\int_{\R} f(z)h_{\omega_1}(y,Bz)\,dz
	= B^{-1}f(0)\int_{\R}\omega_1(z)\,dz
	+O_{N,\delta,\boldsymbol{C},\omega_1}(B^{-1}Q^{-N}).
	\]
    \end{enumerate}
\end{lemma}
\begin{proof}
    We use the results from Heath-Brown~\cite[Section~4]{Heath-Brown96} with $\omega_1$ in place of the function $\omega$ appearing there. This deserves a word of justification, as Heath-Brown states his results for a specific function $\omega$ which he constructs explicitly. However, it only matters that this function is fixed, nonnegative, infinitely differentiable (hence in particular with bounded derivatives of all orders), $L^1$-normalized, and supported on $(1/2,1)$. One can inspect the proofs and see that only bounds for some finitely many derivatives of $\omega$ enter into the error terms; one can also consult the comments prior to \cite[Lemma 6]{Heath-Brown96} where it is mentioned that any other choice for $\omega$ would work as well. Furthermore, the conclusions of our lemma are invariant if we multiply $\omega_1$ by a constant. Thus for $L^1$-normalization, we can assume that \(\int_{\R}\omega_1(z)\,dz=1\), as required. The other properties already hold, and the implicit constants in the lemma in the lemma are permitted (as they must be) to depend on the choice of function $\omega_1$.

    The first part of the lemma follows at once from~\cite[Lemma 4]{Heath-Brown96}. The rest of the proof concerns the second part of the lemma.
    
	Making a change of variable \(z'=z/B\), it is enough to consider the case when \(B=1\).
	Note that for  any \(R\in\N\),~\cite[Lemma 5]{Heath-Brown96} gives
	\[
	\int_{|z|>yQ^{\delta/2}}
	f(z)h_{\omega_1}(y,z)\,dz
	\ll_R
	C_0
	\int_{|z|>yQ^{\delta/2}}
	y^{-1}(y^R+(y/|z|)^R)\,dz,
	\]
	which is \(O_{N,\delta,\boldsymbol{C}}(Q^{-N})\) if we choose \(R\) sufficiently large when $y\leq Q^{-\delta}$. It remains to show
	\begin{equation}\label{eq:HB}
\int_{-yQ^{\delta/2}}^{yQ^{\delta/2}}
	f(z)h_{\omega_1}(y,z)\,dz
	=	f(0)	+O_{N, \delta, \boldsymbol{C}} ( Q^{-N}).
	\end{equation}
	Using the Taylor expansion 
	\(
	f(t) = f(0)+\sum_{n=1}^{2M} f_n t^n+O_{\boldsymbol{C}}  (
	(At)^{2M+1}
	),
	\)
	we obtain
	\begin{align*}
&\int_{-yQ^{\delta/2}}^{yQ^{\delta/2}}
	f(z)h_{\omega_1}(y,z)\,dz
	\\=&
	\sum_{n=0}^{2M}
	f_n\int_{-yQ^{\delta/2}}^{yQ^{\delta/2}}
	z^n
	h_{\omega_1}(y,z)\,dz
	+O_{ \boldsymbol{C}}  \left( f_{2M+1}
	\int_{-yQ^{\delta/2}}^{yQ^{\delta/2}}
	|z|^{2M+1}
	y^{-1}\,dz\right)
	\\
	=
	{}&
	f(0)+O_{ R,\boldsymbol{C}} \bigg(
	\sum_{n=0}^{2M}
	(AyQ^{\delta/2})^n
	\big(
	y^RQ^{\delta/2}+
	Q^{-R\delta/2}
	\big)
	+\frac{A^{2M+1}
	(yQ^{\delta/2})^{2M+2}}{y}
	\bigg),
	\end{align*}
	by~\cite[Lemmas 6, 8]{Heath-Brown96} and the assumption $f_n\ll_n A^n C_n$. 
	By the condition on the size of $y$, this becomes
	\begin{align*}
	f(0)+	O_{ R,\boldsymbol{C}} \bigg(
	\sum_{n=0}^{2M}
	(Q^{-\delta/2})^n
	(Q^{-R\delta+\delta/2}+
	Q^{-R\delta/2})
	+
	(Q^{-\delta/2})^{2M+1}
	Q^{\delta/2}
	\bigg),
	\end{align*}
	and \eqref{eq:HB} follows.
\end{proof}

\section{Two-dimensional delta symbol: duality and \texorpdfstring{$p$}{p}-functions}\label{sec:setup}
In this section, we make some preparations to prove Theorem~\ref{thm:delta}.
We begin by elaborating the first and key step that facilitates the two-dimensional delta symbol in Lemma~\ref{lem:duality} below. This can be seen as a higher dimensional analogue of the equality used in the one dimensional delta symbol method of Duke--Friedlander--Iwaniec~\cite[(2.3)]{DFI}.

\subsection{Detecting the delta symbol by duality of divisors}

We observe that if $\nn\not=\vec{0}$, then there is a unique primitive vector $\uc$ such that $\nn=\lambda\uc^\perp$ for some $\lambda>0$. By symmetry, the sets 
$\{d:d\mid \lambda\}$ and $\{\lambda/d: d\mid \lambda\}$ are the same as long as $\lambda>0$. The vector $\uc$ can be determined using the condition $\uc\cdot \nn=0$, which can be detected with the one dimensional delta symbol. It remains to determine the size of $d, \uc$ that we shall use. 
  The  two-dimensional Dirichlet approximation states that, given any real numbers $\alpha_1, \alpha_2$ and a natural number $Q$, there exist integers $a_1, a_2$ and $1\leq q\leq Q$ such that 
  \begin{equation}\label{eq:Dirichletsupport}
         \Big|\alpha_i-\frac{a_i}{q}\Big|\leq \frac{1}{q Q^{1/2}}, \quad i=1,2.
  \end{equation}
  We would like to choose $Q^{1+1/2}\asymp |\nn|$ so that the error terms of the approximation do not oscillate when $q\asymp Q$. 
  The size of the $d, \uc$ parameters will be chosen so that the support of the resulting functions roughly matches \eqref{eq:Dirichletsupport}.   

\begin{lemma}\label{lem:duality} 
	Let $\omega(x),$ $ h(y,z),$ and $c$ be as in Section~\ref{sec:notation}.
	Let $Q\geq1$ be a parameter and let $ \nn\in \mathbb Z^2$ such that $|\nn|\leq \frac 12Q^{3/2}$. Then for any $N>0$, we have 
	\begin{align}\label{eq:duality}
	\delta_{\nn}&=\frac{c}{Q^3}\sum_{\ur\in\ZZ^2}\omega\Big(\frac{|\ur|}{Q^{1/2}}\Big)\sum_{q\leq Q}\,\,\starsum_{\substack{\ua \bmod q\\ q\mid \ur \cdot \ua^\perp}}e_q(\ua \cdot \nn)h\Big(\frac{\gcd(\ur)q}{\gcd(q,\ur)Q}, \frac{\ur \cdot \nn}{Q^2}\Big)\\& \quad -\frac{2c}{Q}\sum_{d'\in \NN}\mathds 1_{d'\mid \nn}\omega\Big(\frac{|\nn|}{d'Q^{1/2}}\Big)+O_N(Q^{-N}).\nonumber
	\end{align}
\end{lemma}
\begin{proof}
	If $\nn\not=\vec{0}$, then there exist \emph{exactly two} primitive vectors $\uc$ such that  $\nn=\lambda\uc^\perp$ for some $\lambda \neq 0$. Therefore, by writing $\nn =dd'\uc^\perp$, whenever $\nn\not=\vec{0}$,  we have 
	\begin{align}\label{eq:first identity}&\sum_{\uc}\delta_{\uc\cdot \nn}\sum_{d\mid \nn}\omega\Big(\frac{d|\uc|}{Q_1}\Big)-2\sum_{d'\mid \nn}\omega\Big(\frac{|\nn|}{d'Q_1}\Big)\\=&\sum_{\uc}\sum_{d\mid \nn}\delta_{\uc\cdot \frac{\nn}{d}}\omega\Big(\frac{d|\uc|}{Q_1}\Big)-2\sum_{d'\mid \nn}\omega\Big(\frac{|\nn|}{d'Q_1}\Big)=0,\nonumber
	\end{align}
	where $Q_1$ is some parameter to be chosen later and the sum over $\uc$ is over primitive integer vectors.

    We now apply the one dimensional delta symbol~\cite[Theorem 1]{Heath-Brown96} with parameter $Q_2$ to detect $\uc\cdot \frac{\nn}{d}=0$. We may allow $Q_2$ to depend on $d$  and $\uc$;  we can write this as $Q_2=Q_2(d\uc)$ without loss of generality, since $d=\gcd(d\uc)$.
    This gives us 
	\begin{align}\label{cnd}
	\delta_{\uc \cdot \frac{\nn}{d}}=  \frac{c_{Q_2}}{Q_2^2}\sum_{q}\ \starsum_{a\bmod q}e_q(a\uc\cdot \frac{\nn}{d})h\Big(\frac{q}{Q_2},\frac{\uc\cdot\nn}{dQ_2^2} \Big),
	\end{align}
	where $h$ is defined below \eqref{eq:hyzdef} and 
	$c_{Q_2}=1+O_N(Q_2^{-N})$ for any $N$. Using the support of $h(y,z)$ as recorded in Lemma~\ref{lem:HB}, together with the assumption on the size of $\nn$, we can only have non-zero contributions from terms satisfying
    \begin{equation}
        \label{eq:support-h}
    q\leq \max\{ Q_2, {Q_1 d^{-2} Q^{3/2}}{Q_2^{-1}}\}.
    \end{equation}
    Since eventually we would seek a bound of the type $qd\leq Q$, it is natural for us to choose $Q_2=Q/d$ so that \eqref{eq:support-h} becomes
       \begin{equation*}
    q\leq \max\{ Q/d, Q_1 Q^{1/2}/d\}=Q/d,
    \end{equation*}
    upon further choosing $Q_1=Q^{1/2}$.
   With the choice $Q_2=Q/d$ and $Q_1=Q^{1/2}$, we see that \eqref{cnd} becomes
    \begin{align*}
	\delta_{\uc \cdot \frac{\nn}{d}}=  \frac{d^2}{Q^2}\sum_{q\leq Q/d}\ \starsum_{a\bmod q}e_q(a\uc\cdot \frac{\nn}{d})h\Big(\frac{qd}{Q},\frac{d\uc\cdot\nn}{Q^2} \Big)+O_N( Q^4 Q_2^{-N}),
	\end{align*} for any $N>1$. As $Q^{1/2}\leq Q/Q_1\leq Q/d=Q_2$, the error term above is $O_N(Q^{4-N/2})$. Therefore, introducing the sum over $\uc$ and $d$ yields
 \begin{align*}
	 &\sum_{\uc}\sum_{d\mid \nn}\delta_{\uc \cdot \frac{\nn}{d}}\omega\Big(\frac{d|\uc|}{Q^{1/2}}\Big)\\=&  \sum_{\uc}\sum_{d\mid \nn}\omega\Big(\frac{d|\uc|}{Q^{1/2}}\Big)\frac{d^2}{Q^2}\sum_{q\leq Q/d}\ \starsum_{a\bmod q}e_q(a\uc\cdot \frac{\nn}{d})h\Big(\frac{qd}{Q},\frac{d\uc\cdot\nn}{Q^2} \Big)+O_N(Q^{5-N/2}),
	\end{align*}
	for any $N>1$. After using the additive characters to replace the condition $d\mid \nn$, we obtain for any primitive $\uc$,
	\begin{align*}
	&\sum_{d\mid \nn}\starsum_{a\bmod q}e_q(a\uc\cdot \frac{\nn}{d})=\sum_{d}\frac{1}{d^2}\starsum_{\substack{a\bmod q\\ (a, q)=1}}\sum_{\ub\bmod d}e_{qd}(a\uc\cdot \nn+q\ub\cdot\nn)
	\\=&\sum_{d}\frac{1}{d^2}\sum_{\substack{\ua \bmod {qd}\\ q\mid \uc \cdot \ua^\perp\\ (\ua, q)=1}}e_{qd}(\ua \cdot \nn) =\sum_{d}
    \frac{1}{d^2}
\sum_{\substack{d_1,d_2\\d=d_1d_2\\ (d_1, q)=1}}
\sum_{\substack{\ua \bmod qd_2\\ q\mid \uc\cdot d_1\ua ^\perp\\ (\ua, qd_2)=1}}e_{qd_2}(\ua\cdot \nn),
	\end{align*}
	upon writing $d_1=\gcd(\ua, d)$ and then replacing $\ua$ by $d_1\ua$.
	We then re-name $qd_2$ as $q$ and $\ur =d\uc$ so that $d=\gcd(\ur)$, $Q_2=Q_2(\ur)$, and $d_2=\gcd(q, d)=\gcd(q,\ur)$,  to get 
	\begin{align*}
	&\sum_{\uc}\sum_{d\mid \nn}\delta_{\uc \cdot \frac{\nn}{d}}\omega\Big(\frac{d|\uc|}{Q^{1/2}}\Big)\\=&\frac{1}{Q^2}\sum_{\ur=d \uc}\omega\Big(\frac{|\ur|}{Q^{1/2}}\Big)\sum_{q\leq Q/d_1}\;\starsum_{\substack{\ua \bmod q\\ q\mid \ur\cdot \ua^\perp}}e_{q}(\ua\cdot \nn)h\Big(\frac{qd_1}{Q}, \frac{\ur\cdot \nn}{Q^2}\Big)+O_N(Q^{5-N/2})\\
    =&\frac{1}{Q^2}\sum_{\ur=d \uc}\omega\Big(\frac{|\ur|}{Q^{1/2}}\Big)\sum_{\substack{q\leq Q}}\;\starsum_{\substack{\ua \bmod q\\ q\mid \ur\cdot \ua^\perp}}e_{q}(\ua\cdot \nn)h\Big(\frac{\gcd(\ur)q}{\gcd(q,\ur)Q}, \frac{\ur\cdot \nn}{Q^2}\Big)+O_N(Q^{5-N/2}).
	\end{align*}
    Here in the last equality, we have used $d_1=d/d_2=\gcd(\ur)/\gcd(q,\ur)$ and relaxed the upper bound for $q$ using the fact that the terms with $q>Q/d_1$ are all zero from Lemma \ref{lem:HB}. 
    To summarize, we have
	\begin{align}
    \label{eq:second-identity}
	&\quad \sum_{\uc}\sum_{d\mid \nn}\delta_{\uc \cdot \frac{\nn}{d}}\omega\Big(\frac{d|\uc|}{Q^{1/2}}\Big)\\=&\frac{1}{Q^2}\sum_{\ur=d \uc}\omega\Big(\frac{|\ur|}{Q^{1/2}}\Big)\sum_{\substack{q\leq Q}}\;\starsum_{\substack{\ua \bmod q\\ q\mid \ur\cdot \ua^\perp}}e_{q}(\ua\cdot \nn)h\Big(\frac{\gcd(\ur)q}{\gcd(q, \ur)Q}, \frac{\ur\cdot \nn}{Q^2}\Big)+O_N(Q^{-N}),\nonumber
	\end{align} 
    for any $N>1$.

    For $\nn\not=\vec{0}$, we see that \eqref{eq:duality} follows from \eqref{eq:first identity} and \eqref{eq:second-identity}. 
    When $\nn=\vec{0}$, the left-hand side of \eqref{eq:first identity} becomes 
	\begin{align*}
	\sum_{\uc}\sum_{d\mid \nn}\delta_{\uc \cdot \frac{\nn}{d}}\omega\Big(\frac{d|\uc|}{Q^{1/2}}\Big)=\sum_{d,\uc}\omega\Big(\frac{d|\uc|}{Q^{1/2}}\Big)=c^{-1}Q+O_N(Q^{-N}), 
	\end{align*}
	for any $N>0$ with $c$ is defined in \eqref{eq:cdef}, and so by \eqref{eq:second-identity} we see that \eqref{eq:duality} holds in this case as well. 
\end{proof}

\subsection{Defining the \(p\)-functions}
Next, we use Fourier inversion to connect the expression on the right hand side of \eqref{eq:duality} to the $p$-functions appearing in Theorem~\ref{thm:delta}. The first sum in \eqref{eq:duality} can be interpreted as placing rectangular arcs around rationals $\ua/q$ with $q\leq Q$ and $\gcd(\ua,q)=1$ \emph{along the line} $q\mid  \ur\cdot \ua^\perp$ for every $|\ur|\asymp Q^{1/2}$, which results in the term $p_{2,\ur,k,q}$ where $k=\gcd(\ur)/\gcd(\ur, q)$. The second sum in \eqref{eq:duality} corresponds to placing symmetric square arcs
around each rational $\ua/q$ with $q\leq Q$ and $\gcd(\ua,q)=1$, which gives rise to the term $p_{1,q}$. 
 
\begin{lemma}\label{lem:def_of_p_i}
	Let $\omega_0, \omega$ be as in Section~\ref{sec:notation}. Let $\nn\in \mathbb Z^2$ and let $Q\geq 1$ be a large parameter. 
	Then we have
		\begin{equation}\begin{split}\label{eq:p1p2}
		\delta_{\nn}
	=& 
	\sum_{1\leq q\leq Q}
	\underset{\hphantom{\ua\bmod{q}}}{\starsum_{\ua\bmod{q}}}
	e_q(\ua\cdot \nn)
	\int_{\R^2}
	\bigg(p_{1,q}(\uw)\\ & \quad +
	\sum_{\substack{q\mid \ur\cdot \ua^\perp\\ k=\gcd(\ur)/\gcd(\ur, q)}}
	\omega\bigg(
	\frac{ |\ur| }{ Q^{1/2}}
	\bigg)
	p_{2,\ur,k,q}(\uw)
	\bigg)
	e(\uw\cdot\nn)\,d\uw
 +O_N(Q^{-N})
 \end{split}
	\end{equation}
	for any $N>0$	where 
		\begin{align}
		\label{defp2}\quad
			p_{1,q}(\uw)
			&=
			-\frac{2c}{Q}
			\int_{\R^2} 
			\omega_0\left(\frac{|\ux|}{Q^{3/2}}\right)
			\sum_{j\in\NN}
			\frac{1}{q^2j^2}
			\omega
			\left(\frac{|\ux|}{jqQ^{1/2}}
			\right)
			e(-\uw\cdot \ux)
			\,d\ux,\\
				\label{defp1}\quad
				p_{2,\ur,k,q}(\uw)
				&=\frac{c}{Q^3}
				\int_{\R^2} 
				\omega_0\left(\frac{|\ux|}{Q^{3/2}}\right)
				h\left(
				\frac{kq}{Q},
				\frac{\ur\cdot \ux }{ Q^2}
				\right)
				e(-\uw\cdot \ux)
				\,d\ux.
	\end{align}
	Here $h(y,z)$ is defined by \eqref{eq:hyzdef}. 
\end{lemma}

\begin{proof}
	We will introduce a smooth weight on $\nn$ in \eqref{eq:duality} by observing
	\[
	\delta_{\nn}
	= 
	\omega_0\left(\frac{|\nn|}{Q^{3/2}}\right)
	\delta_{\nn}.
	\]
After using Fourier inversion in the $\nn$ variable, we see that the first term on the right side in \eqref{eq:duality} becomes 
\begin{align}\nonumber
\MoveEqLeft\frac{c}{Q^3}
\sum_{\substack{ 1\leq q\leq Q}}
\;\starsum_{\ua \bmod q }
e_{q}(\ua\cdot \nn)
\sum_{\substack{ q\mid \ur\cdot \ua^\perp}}
\omega\left(
\frac{ |\ur| }{ Q^{1/2}}
\right)
h\left(
\frac{\gcd(\ur)q}{\gcd(\ur,q)Q},
\frac{\ur\cdot \nn }{ Q^2}
\right)\omega_0\left(\frac{|\nn|}{Q^{3/2}}\right)\\\nonumber
&=\frac{c}{Q^3}\sum_{\substack{ 1\leq q\leq Q}}
\;\starsum_{\ua \bmod q }
e_{q}(\ua\cdot \nn)
\sum_{\substack{q\mid \ur\cdot \ua^\perp\\\mathclap{k=\frac{\gcd(\ur)}{\gcd(\ur, q)}}}}
\omega\left(
\frac{ |\ur| }{ Q^{1/2}}
\right)
h\left(
\frac{kq}{Q},
\frac{\ur\cdot \nn }{ Q^2}
\right)\omega_0\left(\frac{|\nn|}{Q^{3/2}}\right)\\
&=
\sum_{1\leq q \leq Q}
\underset{\hphantom{\ua\bmod{q}}}{\starsum_{\ua\bmod{q}}}
e_q(\ua\cdot \nn)
\int_{\R^2}
\sum_{\substack{q\mid \ur\cdot \ua^\perp\\ \mathclap{k=\frac{\gcd(\ur)}{\gcd(\ur, q)}}}}
\omega\bigg(
\frac{ |\ur| }{ Q^{1/2}}
\bigg)
p_{2,\ur,k,q}(\uw)
e(\uw\cdot\nn)\,d\uw,
\label{eq:delta3}
\end{align}
where $p_{{2, \ur, k, q}}$ is defined in \eqref{defp1}.
Similarly, the second term on the right side in \eqref{eq:duality} becomes 
\begin{align}\nonumber
	&\frac{2c}{Q}
	\omega_0\left(\frac{|\nn|}{Q^{3/2}}\right)
	\sum_{d'\in\N}
\mathds{1}_{d'\mid \nn}\,
	\omega
	\left(\frac{|\nn|}{d'Q^{1/2}}
	\right)
	\\=&\frac{2c}{Q}\nonumber
	\omega_0\left(\frac{|\nn|}{Q^{3/2}}\right)
	\sum_{1\leq q\leq Q}
	\frac{1}{q^2}
	{\sum_{\ua\,(q)}}^*
	e_q(\ua\cdot\nn)
	\sum_{j\in\NN}
	\frac{1}{j^2}
	\omega
	\left(\frac{|\nn|}{jqQ^{1/2}}
	\right)
	\nonumber
	\\
	=&
	-\sum_{1\leq q\leq Q}
	\underset{\hphantom{\ua\bmod{q}}}{\starsum_{\ua\bmod{q}}}
	e_q(\ua\cdot \nn)
	\int_{\R^2}
	p_{1,q}(\uw)
	e(\uw\cdot\nn)\,d\uw,
	\label{eq:delta2}
\end{align}
where $p_{1,q}$ is defined as in \eqref{defp2}.
 The condition $q\leq Q$ is introduced due to support conditions on $\omega_0$ and $\omega$.
And the result follows from Lemma~\ref{lem:duality}, \eqref{eq:delta3} and \eqref{eq:delta2}.
\end{proof}

\subsection{The function \(p_{1, q}\)}
We give an estimate of of $p_{1,q}(\uw)$ defined in \eqref{defp2} in the following lemma. 
\begin{lemma}\label{lem:p2bound} Let $1\leq q\leq Q$ be an integer, $\uw\in\RR^2$ and $p_{1,q}(\uw)$ be as defined in \eqref{defp2}. Then for any $N>0 $,
	\[
	p_{1,q}(\uw)
	\ll_N \frac{Q}{q(1+|\uw|Q^{3/2})}(1+|\uw|qQ^{1/2})^{-N}.
	\]
\end{lemma}

\begin{proof}
	Making a change of variables, we write
	\begin{align*}
	p_{1,q}(\uw)
	&=
	-
	2c
	\sum_{j\in\NN}\int_{\R^2} 
	\omega_0\left(\frac{jq|\ux|}{Q}\right)
	\omega
	\left(|\ux|
	\right)
	e(-qjQ^{1/2}\uw\cdot \ux)
	\,d\ux.
	\end{align*}
	Since $\omega$ is supported in $(1/2,1)$ and $\omega_0$ is supported in $(-1/2,1/2)$, the sum over $j$ becomes $1\leq j\ll Q/q$. We therefore reach a trivial bound
	\begin{align}\label{p2bound1}
	|p_{1,q}(\uw)|\ll Q/q.
	\end{align}
	Alternatively, upon repeated integration by parts, for a fixed $j\ll Q/q$, we obtain
	\begin{align*}
	\left|\int_{\R^2} 
	\omega_0\left(\frac{jq|\ux|}{Q}\right)
	\omega
	\left(|\ux|
	\right)
	e(-qjQ^{1/2}\uw\cdot \ux)
	\,d\ux\right|\ll_N (1+qjQ^{1/2}|\uw|)^{-N}, 
	\end{align*}
	which leads to an alternate bound
	\begin{equation}\label{p2bound2}
	|p_{1,q}(\uw)|\ll_N \sum_{j\in \NN} (1+qQ^{1/2}j|\uw|)^{-N}\ll_N (qQ^{1/2}|\uw|)^{-1}(1+qQ^{1/2}|\uw|)^{-N}.
	\end{equation}
	The lemma then follows by combining \eqref{p2bound1} and \eqref{p2bound2}.
\end{proof}
\subsection{The function $p_{2, \ur, k, q}$}
We first give various estimates for the function $p_{2, \ur, k,q}$ defined in \eqref{defp1}. In particular, these estimates show that $p_{2, \ur, k, q}$ is supported in a rectangle around the origin with side length $O(1/Q^{3/2})$ in the direction of $\ur^\perp$ and side length $O(1/kqQ^{1/2})$ in the direction of $\ur$ in Lemma \ref{lem:nonstationary}. We then provide an asymptotic evaluation of $p_{r, \ur, k, q}$ for certain range of $\uw$ in Lemma~\ref{lem:stationary}. 
 
\begin{lemma}\label{lem:nonstationary}
Write \(\partial_{\vec{\xi}} = \frac{\xi_1}{|\vec \xi|}\frac{\partial}{\partial w_1}+ \frac{\xi_2}{|\vec \xi|}\frac{\partial}{\partial w_2}\) for the normalized directional derivative (with respect to \(\uw \)). For $Q^{1/2}/2\leq |\ur|\leq Q^{1/2}$,  we have the following for \(a,b,j\geq 0\):
	\begin{align*}
		&\quad q^j\frac{\partial^j}{\partial q^j}
		\partial_{\ur }^a
		\partial_{\ur ^\perp}^b
		 p_{2,\ur,k,q}(\uw)
		\\&\ll_N
		\mathds 1_{\frac{kq}{Q}< 1}
		\big(kqQ^{1/2}\big)^a Q^{3b/2} 
		(1+|\uw|kqQ^{1/2}+|\uw\cdot \ur^\perp|Q)^{-N}.
	\end{align*}
\end{lemma}

\begin{proof}
	If \(\frac{kq}{Q}\geq 1\), the result follows from the fact that
	\(
	h(y,z)=0\) when \(y\geq 1\) and \(|z|\leq y/2\) (see~\cite[Lemma~4]{Heath-Brown96}) together with the support of $\omega$ and $\omega_{0}$.  
	
	Next we consider \(\frac{kq}{Q}< 1\).
We make a change of variables $\ux=z_1\ur+z_2\ur^\perp$ to write 
\begin{align}
p_{2,\ur,k,q}(\uw)&=\frac{c}{Q^3}
\int_{\R^2} 
\omega_0\left(\frac{|\ux|}{Q^{3/2}}\right)
h\left(
\frac{kq}{Q},
\frac{\ur\cdot \ux }{ Q^2}
\right)
e(-\uw\cdot \ux)
\,d\ux
\nonumber
\\
&=\frac{c|\ur|^2}{Q^3}
\int_{\R^2}
\omega_0\left(\frac{|\ur||\uz|}{Q^{3/2}}\right)
h\left(
\frac{kq}{Q},
\frac{z_1|\ur|^2 }{ Q^2}
\right)
e(-\uw\cdot (z_1\ur+z_2\ur^\perp))
\,d\uz.
\label{eq:exposing_h}
\end{align}		
Since $|\ur|\asymp Q^{1/2}$ and $z_1\ur+z_2\ur^\perp=|\ur|(z_1\ur/|\ur|+z_2\ur^\perp/|\ur|)$, we obtain
\begin{align*}
&\quad |
q^j\frac{\partial^j}{\partial q^j}
\partial_{\ur }^a
\partial_{\ur ^\perp}^b
p_{2,\ur,k,q}(\uw)|\\&\ll Q^{-2} \int_{|\uz|< Q^{3/2}/|\ur|}
(Q^{1/2} z_1)^a
(Q^{1/2} z_2)^b
\left|
q^j\frac{\partial^j}{\partial q^j}
h\left(
\frac{kq}{Q},
\frac{z_1 |\ur|^2}{ Q^2}
\right)\right|dz_1dz_2
\\
&\ll  \int_{|\uz|< 1}
(Q^{3/2} z_1)^a
(Q^{3/2} z_2)^b
\left|
q^j\frac{\partial^j}{\partial q^j}
h\left(
\frac{kq}{Q},z_1
\right)\right|dz_1dz_2
\\&\ll
Q^{3b/2}
\int_{|\uz|< 1}
(Q^{3/2} z_1)^a
\left|
q^j\frac{\partial^j}{\partial q^j}
h\left(
\frac{kq}{Q},z_1
\right)\right|dz_1
\\&\ll 
\left(kqQ^{1/2}\right)^a Q^{3b/2} 
,
\end{align*}
where in the last inequality we invoked ~\cite[Lemma~5]{Heath-Brown96}, which gives for each \(R,n\in\Z_{\geq 0}\),
\begin{align}\nonumber
	\frac{\partial^m}{\partial y^m}	\frac{\partial^n}{\partial z^n}h(y,z)
	&\ll_{n,m,R}y^{-1-m-n}\Big(y^R\mathds 1_{n=0}+
	\min\{1,y^R|z|^{-R}\}\Big)
	\\
	&\ll
	y^{-1-m-n}
	\min\{1,y^R|z|^{-R}\},
	\quad \quad \text{if }|z|\leq 1, y<1.
	\label{HBhbound}
\end{align}

It is therefore enough to check the decay properties of $p_{2,\ur,k,q}$, that is, for any $A\gg 1$ such that either $kqQ^{1/2}|\uw|\gg A$ or $Q|\uw\cdot\ur^\perp|\gg A$ we have
\begin{equation*}
|
q^j\frac{\partial^j}{\partial q^j}
\partial_{\ur }^a
\partial_{\ur ^\perp}^b
p_{2,\ur,k,q}(\uw)|\ll_N
\left(kqQ^{1/2}\right)^a Q^{3b/2} 
A^{-N}.
\end{equation*}

The condition $kqQ^{1/2}|\uw|\gg A$ holds if and only if  $\max\{|\uw\cdot\ur|, |\uw\cdot\ur^\perp|\} \gg \frac{A}{kq} $. Since $ \frac{1}{kq} \geq Q^{-1}$, it is enough to consider the two cases
\begin{equation*}
|\uw\cdot\ur|\gg \frac{A}{kq}\,\,\,\,\,\,\textrm{or} \,\,\,\,\,\,|\uw\cdot\ur^\perp|\gg A/Q.
\end{equation*}
First suppose that $|\uw\cdot\ur|\gg \frac{A}{kq}$. 
After rewriting $p_{2,\ur,k,q}$ in \eqref{eq:exposing_h}  and applying integration by parts in the \(z_1\) variable  we obtain 
	\begin{align*}		
		&\quad \frac{c|\ur|^2}{Q^3}
		q^j\frac{\partial^j}{\partial q^j}
		\partial_{\ur }^a
		\partial_{\ur ^\perp}^b
		\int_{\R^2}
\omega_0\left(\frac{|\ur||\uz|}{Q^{3/2}}\right)
h\left(
\frac{kq}{Q},
\frac{z_1|\ur|^2 }{ Q^2}
\right)
e(-\uw\cdot (z_1\ur+z_2\ur^\perp))
\,d\uz
		\\
	&=\pm\frac{c|\ur|^2}{Q^3}
	\int_{\R^2}
	\frac{e(-\uw\cdot (z_1\ur+z_2\ur^\perp))
	}{(\uw\cdot\ur)^N}
	\frac{\partial^N}{\partial z_1^N}
	\Big(
	z_1^a z_2^b|\ur|^{a+b}
	\omega_0\big(\frac{|\ur||\uz|}{Q^{3/2}}\big)	\\ & \quad \quad \quad \quad \quad \quad\quad \quad \quad  \quad \quad \quad\quad \quad \quad\quad \quad \quad\quad \quad 	\times q^j\frac{\partial^j}{\partial q^j}
	h\left(
	\frac{kq}{Q},
	\frac{z_1|\ur|^2 }{ Q^2}
	\right)\Big)
	\,d\uz
	\\
		&\ll 
Q^{(a+b)/2-2}
	\int_{\R^2}
	\left(
	\frac{kq}{A}
	\right)^N\Big|
	\frac{\partial^N}{\partial z_1^N}
	\big[
	z_1^a z_2^b
	\omega_0\left(\frac{|\ur||\uz|}{Q^{3/2}}\right)	q^j\frac{\partial^j}{\partial q^j}
	h\left(
	\frac{kq}{Q},
	\frac{z_1|\ur|^2 }{ Q^2}
	\right)\big]\Big|
	\,d\uz.
	\end{align*}
We  invoke \eqref{HBhbound} again to find that above is
	\begin{align*}
		&
		\ll_{N,R}
	Q^{(a+3b)/2-2}
	\sum_{\substack{ n_1,n_2,n_3\geq 0 \\ n_1\leq a \\N=n_1+n_2+n_3 }}
	\int_{|\uz|\ll Q}
	\Big(
	\frac{kq}{A}
	\Big)^N
		z_1^{a-n_1}
	Q^{-n_2}
	\left(
	kq
	\right)^{-n_3}\\ & \quad \quad \quad \quad \quad \quad\quad \quad \quad  \quad \quad \quad\quad \quad \quad\quad \quad \quad\quad \quad \quad	\times \big(\frac{Q}{kq}\big)
	\min
	\Big\{1,
	\frac{kq}{ z_1}
	\Big\}^R	
	\,d\uz\\
	&\ll_{N,R} Q^{(a+3b)/2}A^{-N}\sum_{\substack{0\leq n_{1}\leq a\\ 0\leq n_{2}\leq N}}\int_{|\uz|\ll 1}(kq)^{n_{1}+n_{2}}z_{1}^{a-n_{1}}Q^{a-n_{1}-n_{2}}\\ & \quad \quad \quad \quad \quad \quad\quad \quad \quad  \quad \quad \quad\quad \quad \quad\quad \quad \quad\quad \quad \quad	\times \big(\frac{Q}{kq}\big)\min\Big\{1, \frac{kq}{Qz_{1}}\Big\}^{R}d\uz\\
&\ll_{N,R} 
\left(kqQ^{1/2}\right)^a Q^{3b/2}  A^{-N}
	\int_{\R^2}
	\omega_0(|\uz|)
	\big(
	\frac{Q}{kq}
	\big)
	\min
	\Big\{1,
	\frac{kq}{ Qz_1}
	\Big\}^R	
	\,d\uz,
	\end{align*}
	where in the last line we use the fact that the term with \(n_{1}=n_2=0\) dominates the sum,
	\(\frac{kq}{Q}<1\) and that \(R\) can be arbitrarily large. The result follows in the case $|\uw \cdot \ur|\gg \frac{A}{kq}$, since the integral over $\uz$ is $O(1)$. 
	
	The case \(
	|\uw\cdot \ur^\perp|
	\gg A/Q\) is similar. Integration by parts in $z_2$ gives 
	\begin{align*}
	&\pm\frac{c|\ur|^2}{Q^3}
	q^j\frac{\partial^j}{\partial q^j}
	\partial_{\ur }^a
	\partial_{\ur ^\perp}^b	\int_{\R^2}
		\omega_0\left(\frac{|\ur||\uz|}{Q^{3/2}}\right)
		h\left(
		\frac{kq}{Q},
		\frac{z_1|\ur|^2 }{ Q^2}
		\right)
		e(-\uw\cdot (z_1\ur+z_2\ur^\perp))
		\,d\uz
		\\
		&=\pm Q^{(a+b)/2-2}
		\int_{\R^2}
		\frac{e(-\uw\cdot (z_1\ur+z_2\ur^\perp))
		}{(\uw\cdot\ur^\perp)^N}
		\frac{\partial^N}{\partial z_2^N}
		\Big(
		z_1^a z_2^b
		\omega_0\big(\frac{|\ur||\uz|}{Q^{3/2}}\big)
		\\ & \quad \quad \quad \quad \quad \quad\quad \quad \quad  \quad \quad \quad\quad \quad \quad\quad \quad \quad\quad 	\times q^j\frac{\partial^j}{\partial q^j}
		h\left(
		\frac{kq}{Q},
		\frac{z_1|\ur|^2 }{ Q^2}
		\right)\Big)
		\,d\uz\\
		&\ll_{N,R}Q^{(a+b)/2-2}A^{-N}\sum_{\substack{0\leq n_{1}\leq b, n_{2}\geq 0\\ N=n_{1}+n_{2}}}\int_{|\uz|\ll Q}Q^{N}z_{1}^{a}z_{2}^{b-n_{1}}Q^{-n_{2}}\\ & \quad \quad \quad \quad \quad \quad\quad \quad \quad  \quad \quad \quad\quad \quad \quad\quad \quad \quad\quad \quad \times\big(\frac{Q}{kq}\big)\min\Big\{1, \frac{kq}{z_{1}}\Big\}^{R}d\uz,
	\end{align*}
	and we can complete the argument in the same way as before.
\end{proof}

To further understand the function $p_{2,\ur,k,q}$, we use Lemma~\ref{lem:HB} to obtain an asymptotic formula for $p_{2,\ur,k,q}(\uw)$ for certain ranges of $\uw$. 
\begin{lemma}\label{lem:stationary}
	Let $\ur\in \ZZ^2, k, q\in \NN$ and $\delta>0$. Suppose \(|\ur|\asymp Q^{1/2}\), $kq\leq Q$, and $kq ( Q|\uw\cdot \ur|+1)\leq Q^{1-\delta}$, then we have 
	\begin{align*}
	p_{2,\ur,k,q}(\uw)=\frac{cQ^{1/2}}{|\ur|}\hat{\omega}_0\Big(\frac{Q^{3/2}}{|\ur|} \uw\cdot\ur^\perp\Big)+O_{N,\delta}(Q^{-N}).
	\end{align*}
\end{lemma}
\begin{proof}
	We have from \eqref{eq:exposing_h} that
	\begin{align*}
	p_{2,\ur,k,q}(\uw)=\frac{c|\ur|^2}{Q^3}\int_{\R}\int_{\R}\omega_0\Big(\frac{|\ur||\uz|}{Q^{3/2}}\Big)h\Big(\frac{k q}{Q}, \frac{z_1|\ur|^2}{Q^2}\Big)e(-\uw\cdot (z_1\ur+z_2\ur^{\perp}))dz_1 dz_2.
	\end{align*}
	The integral over $z_1$ can be estimated using Lemma~\ref{lem:HB} with 
	\begin{align*}
	y= 
	\frac{kq}{Q},\quad
	z=z_1,\quad
	A=Q^{-1}(Q|\uw\cdot \ur|+1), \quad\quad \quad\\
	B= \frac{|\ur|^2}{ Q^2}, \quad 
	f(z_1)=\omega_0\Big(\frac{|\ur|(z_1^2+z_2^2)^{1/2}}{Q^{3/2}}\Big)e(-z_1\uw\cdot \ur),
	\end{align*}
	since the hypothesis in Lemma~\ref{lem:HB}  is guaranteed because we chose $\omega_0$ in such a way that $\omega_0(|x|^{1/2})$ is smooth, and because of our assumption in the present lemma that
	\[
	\frac{kq}{Q}
	\leq Q^{-\delta}( Q|\uw\cdot \ur|+1)^{-1}\asymp Q^{-\delta} B/A. 	\]
	Then the result follows after integrating over $z_2$.
\end{proof}

\section{Two-dimensional delta symbol: Theorem~\ref{thm:delta}} \label{sec:delta proof}\subsection{The lattice \( \Lambda(\ua,q) \)} 
To prove Theorem~\ref{thm:delta}, we show various properties of the lattice $\Lambda(\ua, q)$ defined in \eqref{eq:Lambdadef}, using the geometry of numbers. We first give some more detailed description of the lattice $\Lambda(\ua, q)$.

\subsubsection{Description of $\Lambda(\ua, q)$}  
 Let $q\in\NN, \ua \in \Z^2$ with $\gcd(\ua,q)=1$. Recall the lattice $\Lambda(\ua,q)$ defined in \eqref{eq:Lambdadef}. 
 It is easy to see that
\begin{equation*}
\mathrm{covol}(\Lambda(\ua,q))=q,
\end{equation*}
since the lattice $\Lambda(\ua,q)$ is of index $q$ in $\ZZ^2$.
 We begin by investigating relations among $\Lambda(\ua, q)$ for different $\ua$.
\begin{lemma}\label{lem:lattice}
	Given $q\in\NN, \ua \in \Z^2$ with \(\gcd(\ua,q)=1\), we have
\begin{enumerate}[$(1)$]
	\item 	$\{\ub \bmod q : \ub \in\Lambda(\ua,q),\, \gcd(\ub,q)=1 \}
	=\{k\ua \bmod q : \gcd(k,q)=1\}$;
	\item If $\ub\in \Lambda(\ua,q)$ with $\gcd(\ub,q)=1$ then
	$\Lambda(\ub,q)
	= \Lambda(\ua,q)$;
	\item $\Lambda(\ua,q)=\{\ur\in \Z^2 : q\mid \ur\cdot \ua^\perp\}.$
\end{enumerate}
		
\end{lemma}

\begin{proof}
Without loss of generality, we can assume that \(\gcd(\ua)=1\) since otherwise we can replace \(\ua\) by $\ua/\gcd(\ua)$ since $\gcd(\ua)$ and $q$ are co-prime. 
Write \(\Z^2 = \Z\ua + \Z\ua'\). Then we have  \(\Lambda(\ua,q) = \Z\ua + \Z q\ua'\).

If \(\gcd(k,q)=1 \) then \(\gcd(k\ua,q)=1\) and it follows that 
\[
\{k\ua \bmod q : \gcd(k,q)=1\}
\subseteq
\{\ub \bmod q : \ub \in  \Z\ua + \Z q\ua',\, \gcd(\ub,q)=1 \}.
\]
Conversely if \(\ub  \in \Z\ua + \Z q\ua'\) then \(\ub\equiv k\ua \bmod q \) for some \(k\), and if \(\gcd(\ub,q)=1\) then \(\gcd(k,q)=1\). This proves the first equality, and the second equality follows since for $\gcd(k, q)=1$,
\[
\{r\ua \bmod q :r\in \Z\}
=
\{rk\ua \bmod q: r\in \Z \}.
\]
For the third equality, we observe that \(\ua^\perp \cdot \ua' =- \det(\ua|\ua') = \pm 1\) and therefore
\begin{align*}
 \Lambda(\ua, q)=\Z\ua + \Z q\ua'
&=
\{\ur \in  \Z\ua + \Z \ua': q\mid\ur\cdot \ua^\perp\}.
\end{align*}
\end{proof}

\subsubsection{Geometry of numbers}

We state a standard result from the geometry of numbers, which can be seen as a special case of Lemma~4.1 of Maynard~\cite{Maynard}. 
\begin{lemma}\label{lem:minkowski_basis}
	Let $\Lambda$ be a lattice in $\RR^2$.
	Let \(M\) be an \(m \times 2\) real matrix with full rank. Further let 
	\( \ux_1 \in \Lambda \) with \(|M\ux_1| = \mu_M(\Lambda )=\min\{ |M\ux| : \ux \in \Lambda, \ux \neq\vec{0}\}\). Then there is an \(\ux_2\in \Lambda \) satisfying the following
	\begin{align*}\
	\Lambda & = \Z \ux_1 + \Z \ux_2,
	\\
	| r_1 M\ux_1 + r_2 M\ux_2 |
	&\asymp_m |r_1||M\ux_1| + |r_2| |M\ux_2|,
	\quad (\ur\in\R^2),
	\\
	|  M\ux_1 || M\ux_2 |
	&\asymp_m
	\mathrm{covol}(\Lambda)\meas\{ M [0,1]^2 \}.
	\end{align*}
\end{lemma}
The key to bounding the contributions from the $p_{2,\ur,k,q}$ terms is to apply Lemma~\ref{lem:minkowski_basis} with $\Lambda=\Lambda(\ua,q)$ and a $3\times 2$ matrix $M$ defined as below:
\begin{definition}		
Let $M=M(\uw)$ be the $3\times 2$ matrix such that for any $\ur\in \Z^2$ 
	\begin{equation}\label{def:Mmatrdef_0}
		M\ur
		=
		\left(\frac{r_1}{Q^{1/2}},
		\frac{r_2}{Q^{1/2}},
		Q\uw\cdot \ur^\perp\right)^T.
	\end{equation}
	
\end{definition}
\begin{lemma}
	\label{lem:matrix_M}
	 For $M$ in defined in \eqref{def:Mmatrdef_0}, we have 
	\begin{equation*}
		\meas \big\{ M [0,1]^2 \big\} \asymp  Q^{1/2}|\uw|+Q^{-1}.
	\end{equation*}
\end{lemma} 
\begin{proof}
Let $e_1,e_2,e_3$ denote the standard basis for $\RR^3$. The measure of the parallelogram $M[0,1]^2$ is equal to the norm of the vector \begin{align*}
&(Q^{-1/2}e_1 + Q w_2 e_3)\wedge(Q^{-1/2}e_2 - Q w_1 e_3)\\ =& Q^{-1} \, e_1 \wedge e_2 - Q^{1/2} w_1 \, e_1 \wedge e_3 + Q^{1/2} w_2 \, e_3 \wedge e_2,
\end{align*}
which is
\begin{equation*}
\sqrt{Q^{-2} + Q(w_1^2+w_2^2)}\asymp Q^{-1}+Q^{1/2}|\uw|.
\end{equation*}

\end{proof}

\subsection{The function $p_{\Lambda(\ua,q)}$.}

 We are now ready to state the formula for $p_{\Lambda(\ua, q)}$ and its properties. 

\subsubsection{Definition of $p_{\Lambda(\ua,q)}$}
After replacing the sum over \(\ur\) in \eqref{eq:p1p2} by a sum over the lattice \(\Lambda(\ua,q)\) using Lemma~\ref{lem:lattice}, we can define $p_{\Lambda(\ua, q)}$ as 
	\begin{align}
		\label{eq:p2_p_2}
		p_{\Lambda(\ua,q)}(\uw)&:=p_{1,q}(\uw)+
		\sum_{\substack{\ur\in\Lambda(\ua,q)\\ k = \gcd(\ur)/\gcd(\ur,q)}}
		\omega\bigg(
		\frac{ |\ur| }{ Q^{1/2}}
		\bigg)
		p_{2,\ur,k,q}(\uw),
		\end{align}
		which proves \eqref{eq:delta}.

\subsubsection{Asymptotic of $p_{\Lambda(\ua,q)}$}
To estimate $p_{\Lambda(\ua, q)}$, we first consider the sum over $\ur$ in \eqref{eq:p2_p_2}.
\begin{lemma}\label{lem:p1_sum_asymptotic}
	Let $\delta>0$ and $M=M(\uw)$ be as in \eqref{def:Mmatrdef_0} and let $\mu_{M}=\mu_M(\Lambda(\ua, q))$ the norm of the smallest non-zero vector in $M (\Lambda (\ua, q))$. If
	\begin{equation}\label{eq:muM assump}
	\mu_M\geq
	Q^{\delta}(qQ^{-1}+|\uw|qQ^{1/2}),
	\end{equation}
	then
	\begin{equation*}
    \begin{split}
	&\sum_{{\ur\in\Lambda(\ua,q)}}
	\omega\left(
	\frac{ |\ur| }{ Q^{1/2}}
	\right)
	p_{2,\ur,k,q}(\uw)+O_{N,\delta}(Q^{-N})
	\\=&
	\frac{cQ^{1/2}}{q}
	\int_{\R^2}\frac{1}{|\ur|}
	\omega\left(
	\frac{ |\ur| }{ Q^{1/2}}
	\right)
	\hat{\omega}_0
	\Big(
	\frac{Q^{3/2}}{|\ur|} \uw\cdot\ur^\perp
	\Big)\,d\ur
    \end{split}
	\end{equation*}
	for any $N>0$.
\end{lemma}

\begin{proof}
	Throughout, we write $\ur=dk\uc$, where $\uc$ is a primitive integer vector, \(d=\gcd(\ur,q)\) and $\gcd(\ur)=kd$ with $\gcd(q/d,k)=1$.
	First we prove that
	\begin{equation}\begin{split}
    \label{4.4}
		&\sum_{{\ur\in\Lambda(\ua,q)}}
		\omega\left(
		\frac{ |\ur| }{ Q^{1/2}}
		\right)
		p_{2,\ur,k,q}(\uw)+O_{N,\delta}(Q^{-N})
		\\=&
		\sum_{{\ur\in\Lambda(\ua,q)}}
		\frac{cQ^{1/2}}{|\ur|}
		\omega\left(
		\frac{ |\ur| }{ Q^{1/2}}
		\right)
		\hat{\omega}_0
		\Big(
		\frac{Q^{3/2}}{|\ur|} \uw\cdot\ur^\perp
		\Big).
        \end{split}
	\end{equation}
	By Lemma~\ref{lem:stationary}, the terms on each side with \(
	k\leq  \frac{Q^{-\delta/2}}{|\uw|qQ^{1/2}+qQ^{-1}}\) agree up to negligible error. We show that on each side of \eqref{4.4}, there is only a negligible contribution from values of \(\ur\) with
	\begin{equation}\label{eq:negligible_r}
	k\geq  \frac{Q^{-\delta/2}}{|\uw|qQ^{1/2}+qQ^{-1}}.
\end{equation}
	Suppose that \(\ur\) satisfies
	  \eqref{eq:negligible_r}. It is easy to see that \(	\frac{\ur} {k}=d\uc \in{\Lambda(\ua, q)}\), and therefore
	\[
	\left\lvert M(d\uc)\right\rvert \geq \mu_M.
	\]
	This yields the penultimate inequality in
	\[
	\frac{ |\ur| }{ Q^{1/2}}
+
	Q |\uw\cdot\ur^\perp|
	\asymp |M\ur|\geq \mu_Mk
	\geq
	{Q^{\delta/2}},
	\]
    where the first equivalence holds by definition of $M$ and the last inequality holds by
	 \eqref{eq:negligible_r} and the assumption on \(\mu_M\) in the lemma.
	Now using Lemma~\ref{lem:nonstationary}, one can check that for any such \(\ur\), we have
	\begin{align*}
		\omega\left(
		\frac{ |\ur| }{ Q^{1/2}}
		\right)
		p_{2,\ur,k,q}(\uw)
		&\ll_{N,\delta}Q^{-N},
		\\\
		\omega\left(
		\frac{ |\ur| }{ Q^{1/2}}
		\right)
		\hat{\omega}_0
		\Big(
		\frac{Q^{3/2}}{|\ur|} \uw\cdot\ur^\perp
		\Big)
		&\ll_{N,\delta}Q^{-N},
	\end{align*}
as required, verifying \eqref{4.4}.

We next use the Poisson summation formula to prove that  the sum on the right hand side of \eqref{4.4} may be replaced by an integral up to an admissible error. Let \(X = (\ux_1|\ux_2)\) be the matrix with columns given by the vectors  \(\ux_i\) from Lemma~\ref{lem:minkowski_basis} for the lattice $\Lambda=\Lambda(\ua,q)$ and $M$. By \eqref{4.4} it suffices to prove
\begin{align*}
&\sum_{{\ur\in\Lambda(\ua,q)}}
\frac{cQ^{1/2}}{|\ur|}
\omega\left(
\frac{ |\ur| }{ Q^{1/2}}
\right)
\hat{\omega}_0
\Big(
\frac{Q^{3/2}}{|\ur|} \uw\cdot\ur^\perp
\Big)
\\=&
\frac{cQ^{1/2}}{q}
\int_{\R^2}\frac{1}{|\ur|}
\omega\left(
\frac{ |\ur| }{ Q^{1/2}}
\right)
\hat{\omega}_0
\Big(
\frac{Q^{3/2}}{|\ur|} \uw\cdot\ur^\perp
\Big)\,d\ur
+O_{N,\delta}(Q^{-N}).
\end{align*}
 We can re-write the left-hand side above as
\begin{align*}
	\sum_{\uy\in \Z^2}
	W(M X\uy ),  \text{ where }
	W((x,y,z)^T)=\frac{c}{|({}^x_y)|}\omega\big( \lvert({}^x_y)\rvert\big) \hat{\omega}_0\big(|({}^x_y)|^{-1}z\big).
\end{align*}
By Lemma~\ref{lem:minkowski_basis} and the fact that $\mathrm{covol}(\Lambda(\ua, q))=q$, we have
\begin{align*}
	|M\ux_1||\uy|
	\ll
	|M\ux_1||y_1|
	+
	|M\ux_2||y_2|
	\asymp
	|MX\uy|\\\ll |M\ux_2||\uy|\ll
	\frac{
		q\meas\{ M [0,1]^2 \}}
	{
		|M\ux_1|}|\uy|.
\end{align*}
From the assumption \(\mu_M = |M\ux_1| \geq  Q^{\delta} ( qQ^{1/2}|\uw|  + qQ^{-1}) \) together with Lemma~\ref{lem:matrix_M}, we obtain
\[
|MX\uy|\ll
Q^{-\delta}|\uy|.
\]
Using this bound, a routine application of Poisson summation and integration by parts (or ``non-stationary phase'') shows that
\begin{equation*}
	\sum_{\uy\in \Z^2}
W(MX\uy )
=
\int_{\R^2}
W(MX\uy )\,d\uy
+O_{N,\delta}(Q^{-N}),
\end{equation*}
which concludes the result. 
\end{proof}
Next we apply Lemma~\ref{lem:p1_sum_asymptotic} to prove that under the assumption \eqref{eq:muM assump}, the function $p_{\Lambda(\ua,q)}$ is equal to $1$, up to a very small error, which implies \eqref{eq:p1bound1}.
\begin{lemma}\label{lem:1}
Let $M=M(\uw)$ and $\mu_{M}=\mu_M (\Lambda(\ua, q))$ be as in Lemma~\ref{lem:p1_sum_asymptotic}. Let $1\leq q\leq Q$ and  $\delta>0$. If
	\[
	\mu_M\geq
	Q^{\delta}(qQ^{-1}+|\uw|qQ^{1/2}),
	\]
	then
	\[
	p_{\Lambda(\ua,q)}(\uw)
	=1+O_{N,\delta}(Q^{-N}),
	\]
	for any $N>0$.
\end{lemma}

\begin{proof}
By Lemma~\ref{lem:p1_sum_asymptotic}, and making a change of variables using \(\ur=z'\uv_{\theta}, \ \uv_{\theta}=(\cos \theta, \sin \theta),\  z'=|\ur|\), we obtain
	\begin{align*}
		&\quad p_{\Lambda(\ua, q)}(\uw)-
		p_{1,q}(\uw)\\&=\frac{cQ^{1/2}}{q}
	\int_{\R^2}\frac{1}{|\ur|}
	\omega\left(
	\frac{ |\ur| }{ Q^{1/2}}
	\right)
	\hat{\omega}_0
	\Big(
	\frac{Q^{3/2}}{|\ur|} \uw\cdot\ur^\perp
	\Big)\,d\ur
	+O_{N,\delta}(Q^{-N})\\
    &=\frac{cQ^{1/2}}{q}\int_{0}^\infty \omega\Big(\frac{z'}{Q^{1/2}}\Big) dz' \int_{0}^{2\pi}\hat{\omega}_0\Big(
     Q^{3/2}\uw \cdot \uv_{\theta}^\perp \Big)d\theta+O_{N, \delta}(Q^{-N})\\
     &=\frac{cQ}{q}\int_{0}^{2\pi}\int_{\RR}\omega_0(z)e\Big(-zQ^{3/2}\uw 
     \cdot \uv_{\theta}^\perp\Big)dz d\theta +O_{N, \delta}(Q^{-N}).
    \end{align*}
By Lemmas~\ref{lem:minkowski_basis} and~\ref{lem:matrix_M}, we have 
\begin{align}\label{muM}
\mu_M \ll (qQ^{-1}+|\uw|qQ^{1/2})^{1/2}.
\end{align} Therefore, by the assumption on \(\mu_M\) we have
	 \(q\ll Q^{1-\delta}\), and therefore upon substituting  \(z=Q^{3/2}z\) we get 
	\begin{align*}
		&\quad p_{\Lambda(\ua, q)}(\uw)-
		p_{1,q}(\uw)+O_{N,\delta}(Q^{-N})
		\\&=
		\frac{c}{Q^{3/2}}
		\left(\sum_{j\in\N} \omega\left(\frac{qj}{Q}\right)\right)
		\int_{0}^{2\pi}\int_{\RR}
		{\omega}_0\left(\frac{z}{Q^{3/2}}\right)
		e\left(-zQ^{3/2}\uw\cdot \uv_\theta^\perp\right)dzd\theta
		.
	\end{align*}
    
	Note that by making a change of variable $\ux \mapsto \ux^\perp$ in $p_{1, q}(\uw)$ defined in \eqref{defp2}, and further changing to polar coordinates, we can write
    \begin{equation}
    \begin{split}
    &p_{1,q}(\uw)
			=
			-\frac{2c}{Q}
			\int_{\R^2} 
			\omega_0\left(\frac{|\ux|}{Q^{3/2}}\right)
			\sum_{j\in\NN}
			\frac{1}{q^2j^2}
			\omega
			\left(\frac{|\ux|}{jqQ^{1/2}}
			\right)
			e(-\uw\cdot \ux^\perp)
			\,d\ux\\
            &=-\frac{2c}{Q}\int_{0}^{2\pi}\int_{0}^\infty z\omega_0\Big(\frac{z}{Q^{3/2}}\Big) \sum_{j\in\NN}
			\frac{1}{q^2j^2}
			\omega
			\left(\frac{z}{jqQ^{1/2}}
			\right)
			e\Big(-z\uw\cdot \uv_{\theta}^\perp\Big)
			\,dz d\theta.\label{eq:p1qmanipulation}
            \end{split}
            \end{equation}

It is straight-forward to see, by replacing $\theta$ with $\theta+\pi$, that
\begin{align*}
    &\int_{0}^{2\pi}\int_{0}^\infty z\omega_0\Big(\frac{z}{Q^{3/2}}\Big) \sum_{j\in\NN}
			\frac{1}{q^2j^2}
			\omega
			\left(\frac{z}{jqQ^{1/2}}
			\right)
			e\Big(-z\uw\cdot \uv_{\theta}^\perp\Big)
			\,dz d\theta\\&=\int_{0}^{2\pi}\int_{0}^\infty z\omega_0\Big(\frac{z}{Q^{3/2}}\Big) \sum_{j\in\NN}
			\frac{1}{q^2j^2}
			\omega
			\left(\frac{z}{jqQ^{1/2}}
			\right)
			e\Big(z\uw\cdot \uv_{\theta}^\perp\Big)
			\,dz d\theta.
\end{align*}
As $\omega_0$ is even, we may therefore re-write \eqref{eq:p1qmanipulation} as
\begin{align*}p_{1,q}(\uw)
			=-\frac{c}{Q}\int_{0}^{2\pi}\int_{\mathbb R} |z|\omega_0\Big(\frac{z}{Q^{3/2}}\Big) \sum_{j\in\NN}
			\frac{1}{q^2j^2}
			\omega
			\left(\frac{|z|}{jqQ^{1/2}}
			\right)
			e\big(-z\uw\cdot \uv_{\theta}^\perp\big)
			\,dz d\theta.
            \end{align*}

 Therefore, we have 
	\begin{align*}
		&\quad \quad p_{\Lambda(\ua, q)}(\uw)+O_{N, \delta}(Q^{-N})\\&=\frac{c}{Q^{\frac{3}{2}}}
		\int_{0}^{2\pi}\int_{\RR}\Big(\sum_{j\in\N} \omega\big(\frac{qj}{Q}\big)-\frac{Q^{\frac{1}{2}}|z|}{j^2q^2} \omega\big(\frac{|z|}{jqQ^{1/2}}\big)\Big)
		\omega_0\big(\frac{z}{Q^{3/2}}\big)
		e(-z\uw\cdot \uv_{\theta}^\perp)
		dzd\theta,	
	\\	&=
		\frac{c}{Q^{\frac{3}{2}}}
		\int_{0}^{2\pi}\int_{\RR} 
	\omega_0\big(\frac{z}{Q^{3/2}}\big) 
		h_2\big(
		\frac{q}{Q}
		,
		\frac{z}{Q^{3/2}}
		\big)
		e\big(-z\uw\cdot \uv_{\theta}^\perp\big)
		dzd\theta,
	\end{align*}
	where
	\begin{align*}
		h_2(y,z)
		&=
		\sum_{j\in\N}
		\frac{1}{yj}
		\left(
		{yj}
		\omega\left(yj\right)
		-
		\frac{|z|}{yj}
		\omega\left(\frac{|z|}{yj}\right)
		\right).
	\end{align*}
	Using the assumption in the lemma and \eqref{muM},  we have 
	\(qQ^{-1}+|\uw|qQ^{1/2} \ll Q^{-2\delta}\), so that we can apply Lemma~\ref{lem:HB} to the $z$-integral with
	\begin{align*}
		\omega_1(x)={x\omega(x)},\quad
		y=\frac{q}{Q},\quad
		A=Q^{-3/2}+|\uw|,\quad
		\\B=Q^{-3/2},\quad
		f(z)=\omega_0\Big(\frac{z}{Q^{3/2}}\Big)e\Big(-z\uw\cdot \uv_{\theta}^\perp\Big),
	\end{align*}
	and conclude that
	\begin{equation*}\begin{split}
		p_{\Lambda(\ua, q)}(\uw)
		&=c \cdot \omega_0(0)e(0)
		\int_{0}^{2\pi} 
		\int_0^\infty x
		\omega(x)\,dx
		\,d\theta+O_{N,\delta}(Q^{-N})
		\\&=1+O_{N,\delta}(Q^{-N}),
        \end{split}
	\end{equation*}
 which proves the lemma.
\end{proof}

\subsection{Proof of Theorem~\ref{thm:delta}}\label{sec:proof-delta}
 Equation \eqref{eq:delta} follows from \eqref{eq:p2_p_2}. Applying Lemma~\ref{lem:1} and noting that $\mu_M\geq Q^{-1/2}$, we obtain \eqref{eq:p1bound1}. The bound
\eqref{eq:p1bound2} follows from Lemma~\ref{lem:p2bound} and \eqref{eq:p1bound3} follows from Lemma~\ref{lem:nonstationary}. \eqref{eq:delta0} follows from splitting $p_{\Lambda(\ua.q)}(\uw)$ into two sums using \eqref{eq:p2_p_2}, switching the order of summation, separating the two sums and in the second sum, upon further writing $\ur=kd\uc$ with $\uc\in \mathbb Z^2$ primitive, $\gcd(\ur, q)=d$, so that $\gcd(q/d, k)=1$.

\section{Rational points: preparing for Theorem~\ref{thm:application}} 
\label{sec:quad-forms-basic-notation}

All the notation introduced in the next three subsections will be used throughout the remainder of this paper.
In particular the reader will frequently encounter a distinction between four kinds of primes: bad primes, good primes for bad $\uc$, type I primes, and type II primes. These terms are defined here, in Section~\ref{sec:geometry}. Other conventions and definitions are found in sections~\ref{sec:notation} and~\ref{sec:Qdelta-and-Sigma}.

\subsection{The major and minor arcs and a version of Theorem~\ref{thm:delta}}\label{sec:delta2}
  Let $\delta>0$ be fixed but sufficiently small, to be chosen later. 
 In view of the properties \eqref{eq:p1bound1} and \eqref{eq:p1bound3}, we define the major arcs as 
 \begin{align*}
 \mathfrak M_q&=\mathfrak M_{q}(\delta) =\Big\{\uw\in \mathbb R^2: |\uw|< q^{-1}Q^{-1-\delta}\Big\}
 &\text{ if }1\leq q\leq Q^{1/2-\delta},
 \end{align*}
 {and the minor arcs as}
  \begin{align}
 \label{minorarc}
 &\quad \mathfrak m_q=\mathfrak m_{q}(\delta)\\&=\begin{cases}
 \big\{\uw \in \mathbb R^2: q^{-1}Q^{-1-\delta}\leq |\uw| <q^{-1}Q^{-1/2+\delta}\big\} & \text{ if } 1\leq q\leq Q^{1/2-\delta},\\
 \big\{\uw\in \mathbb R^2: |\uw|< q^{-1}Q^{-1/2+\delta}\big\} & \text{ if }Q^{1/2-\delta}< q\leq Q.
 \end{cases}
 \nonumber
 \end{align}
Most authors do not take major and minor arcs in the \(\delta\)-method, and indeed they are a little different to the classical case.
Note that for us the arcs \(\mathfrak M_q,\mathfrak m_q\) are all centred at $\vec{0}$. The traditional (shifted) major and minor arcs would be \(\ua/q+\mathfrak M_q\) and \(\ua/q+\mathfrak m_q\) for \((\ua,q)=1\). Contrary to the classical circle method, here the union of the shifted major arcs $\mathfrak M = \bigcup_{q=1}^{Q^{1/2-\delta}} \bigcup_{\ua}^*(\ua/q+{\mathfrak M}_q) $ is not disjoint from the union of the shifted minor arcs $\mathfrak m = \bigcup_{q=1}^{Q} \bigcup_{\ua}^*(\ua/q+{\mathfrak m}_q) $. This is acceptable because, on each arc, we integrate  with respect to a different kernel function \(p_{\Lambda(\ua,q)}\).

We begin by noting that using \eqref{eq:p1bound1}, on major arcs the function \(p_{\Lambda(\ua,q)}\) is equal to $1$ with a very small error. On the minor arcs, it would be preferable for us to not work with the kernel function \(p_{\Lambda(\ua,q)}\) directly. We will instead apply the following proposition which is an easy corollary of Theorem \ref{thm:delta}. In particular, it follows from a combination of \eqref{eq:p1bound1} and \eqref{eq:delta0}. This would allow us to work with known exponential sums and also draw parallels with~\cite[Theorem 1.1]{V19} and further allow us to re-use some of the bounds obtained in~\cite{V19} over minor arcs. More precisely, the sum $E_2$ appearing in Proposition~\ref{prop:delta2} will lead to the term $N_2(P,\delta)$ appearing in \eqref{eq:N2def}, which corresponds closely with~\cite[Lemma 5.1]{V19}.  
\begin{proposition}\label{prop:delta2}
	Let $\nn \in \mathbb Z^2 $ and let $Q\geq 1$ be a large parameter. 	Then for any $\delta, N>0$ with \(\delta\) sufficiently small, we have
	\begin{align*}
	\delta_{\nn}=M+E_1+E_2+O_{N, \delta}(Q^{-\delta N}),
	\end{align*}
where 
	\begin{equation*}
    \begin{split}
	M&=\sum_{1\leq q\leq Q^{1/2-\delta}}\,\,\starsum_{\ua \bmod q}\int_{\uw\in \mathfrak M_q }e((\ua/q+\uw)\cdot \nn)d\uw,\\
	E_1&=\sum_{1\leq q\leq Q}\,\,\starsum_{\ua \bmod q}\int_{\uw\in \mathfrak m_q}p_{1,q}(\uw)e((\ua/q+\uw)\cdot \nn)d\uw,\\
	E_2&=\sum_{\substack{d,k\in \NN\\ \uc \in \ZZ^2 \mathrm{ primitive }\\ \ur=dk\uc}}\omega\big(\frac{\ur}{Q^{1/2}}\big)\sum_{\substack{1\leq q\leq Q/k\\d\mid q\\ \gcd(q/d, k)=1}}\,\,\starsum_{\substack{\ua \bmod q\\q\mid d\uc \cdot\ua^\perp}}\\
    & \quad \quad  \times \int_{\uw \in \mathfrak m_q }p_{2,\ur,k,q}(\uw)e((\ua/q+\uw)\cdot \nn)d\uw,
    \end{split}
	\end{equation*}
with $\omega$ as in Theorem \ref{thm:delta} and the functions $p_{1, q}$ and $p_{2,\ur,k,q}$ as defined as in Lemma~\ref{lem:def_of_p_i}, satisfying decay properties \eqref{eq:p1bound2} and \eqref{eq:p1bound3}.
\end{proposition}

\subsection{Reduction of Theorem~\ref{thm:application} to the sums \(N_i\)}\label{sec:the-Ni}
For the rest of this paper we adopt the following conventions. We
write $\nF=(F_1, F_2)$ for a pair of forms of degree $d\geq 2$. We let $P\geq 1$ be a large parameter and we let $w\in C^\infty_c(\RR^s) $ be a smooth function with a compact support. We work with the counting function
\begin{equation*}
N(P):=N_{\nF, w}(P):=\sum_{\substack{\x\in \ZZ^s\\ \nF(\x)=\vec{0}}}w(\x/P)=\sum_{\x\in \ZZ^s}w(\x/P)\delta_{\uF(\x)}.
\end{equation*}
We choose $Q$ such that $Q^{1+1/2}=P^{d}$. Starting with Lemma~\ref{lem:Ni-bounds} we only consider the case $d=2$, and so we set $Q=P^{4/3}$ from then on.

\begin{lemma}\label{lem:N1P}
	With the notation above, we have, for any sufficiently small $\delta>0$, that 
    \begin{equation}\label{eq: N split}
N(P)=N_0(P, \delta)+N_1(P, \delta)+N_2(P, \delta)+O_{N,\delta}(Q^{-N\delta}),
\end{equation}
where
\begin{align} \label{eq:N0def}
N_0(P, \delta)&=\sum_{1\leq q< Q^{1/2-\delta}}q^{-s}
			\sum_{\substack{\uu\in \Z^s}}D_q(\uu)			
			\int_{\uw\in \mathfrak M_q(\delta)} I_q(\uw,\uu)
			\,d\uw,\\\label{eq:N1def}
		N_1(P, \delta)&=\sum_{1\leq q< Q}q^{-s}
			\sum_{\uu\in \Z^s}D_q(\uu)			
			\int_{\scrm_q(\delta)} p_{1,q}(\uw) I_q(\uw,\uu)
			\,d\uw,\\\label{eq:N2def}
			N_2(P, \delta)&=\sum_{\substack{d,k\in\NN\\ \uc\in\ZZ^2 \mathrm{primitive} \\ \ur=dk\uc}}\omega\big(\frac{\ur}{Q^{1/2}}\big)\sum_{\substack{1\leq q\leq Q/k\\ d\mid q\\ \gcd(q/d,k)=1}} q^{-s}				\\ &\quad \quad \times \sum_{\uu\in \Z^s}S_{q,d\uc}(\uu)
			\int_{\uw\in\scrm_q(\delta)}
			p_{2,\ur,k,q}(\uw)I_q(\uw, \uu)
			\,d\uw,			
            \nonumber
\end{align}
where the exponential sums $D_q, S_{q,d\uc}$ and the exponential integral $I_q$ are defined as
\begin{align}
D_q(\uu)&=\starsum_{\ua \bmod q}\sum_{\bb \bmod q}e_{q}(\ua\cdot \uF(\bb)+\bb\cdot \uu),\label{eq:Dq def}\\
S_{q,d\uc}(\uu)&=\starsum_{\substack{\ua \bmod q\\ q\mid d\uc\cdot\ua^\perp}}\sum_{\bb \bmod q}e_{q}(\ua\cdot \uF(\bb)+\bb\cdot \uu),\label{eq:sq def}\\
I_q(\uw, \uu)&=\int_{\R^s}w\left(\frac{\x}{P}\right)e(\uw\cdot \nF(\x))e_q(-\x\cdot\uu)d\x.\label{eq:Izu def}
\end{align}
\end{lemma}

\begin{proof}
Applying Proposition~\ref{prop:delta2}, we obtain \eqref{eq: N split}
where 
\begin{align*}
\quad N_0(P, \delta)&:=\sum_{1\leq q< Q^{1/2-\delta}}
			\,\,{\starsum_{\ua\bmod{q}}}\sum_{\x\in \Z^s}e_q(\ua \cdot \nF(\x))			
			\\\ &\quad \quad \times\int_{\uw\in \mathfrak M_q(\delta)}
			w(\x/P)e(\uw\cdot \nF(\x))\,d\uw,\\
	\quad N_1(P, \delta)&:=\sum_{1\leq q\leq Q}
			\underset{\hphantom{\ua\bmod{q}}}{\starsum_{\ua\bmod{q}}}	\sum_{\x\in \Z^s}e_q(\ua \cdot \nF(\x))					
			\\ &\quad \quad \times \int_{\uw\in\scrm_q(\delta)}
			p_{1,q}(\uw)
			w(\x/P)e(\uw\cdot \nF(\x))\,d\uw,\\
			\quad N_2(P, \delta)&:=\sum_{\substack{d,k\in \NN\\ \uc \text{ primitive }\\ \ur=dk\uc}}\omega\big(\frac{\ur}{Q^{1/2}}\big)\sum_{\substack{1\leq q\leq Q/k\\ d\mid q\\ \gcd(q/d,k)=1}}
			\underset{\hphantom{\ua\bmod{q}}}{\starsum_{\substack{\ua\bmod{q}\\ q\mid d\uc\cdot \ua^\perp}}}	\sum_{\x\in \Z^s}e_q(\ua \cdot \nF(\x))\\
   & \quad \quad  \times
			\int_{\uw\in\scrm_q(\delta)}w(\x/P)
			p_{2,\ur,k,q}(\uw)
			e(\uw\cdot \nF(\x))\,d\uw.
\end{align*}

The lemma follows by applying Poisson summation in the $\x$ variable with modulus $q$. 
\end{proof}

The process of proving the estimates for $N_1(P, \delta)$ and $N_2(P, \delta)$ from the minor arcs in Lemma~\ref{lem:Ni-bounds}
 is similar to that in~\cite{V19}. However, the current situation is more complicated due to the different weight functions $ p_{1,q},p_{2,\ur,k,q}$ appearing in $N_1(P, \delta), N_2(P, \delta)$ in \eqref{eq:N1def}-\eqref{eq:N2def}. Our bounds for the $p$-functions get increasingly worse as the sizes of $q$ and $|\uw|$ decrease. However, in the extreme case  \[q\asymp Q,\  |\uw|\asymp q^{-1}Q^{-1/2+\delta},\]
our bounds match those in~\cite{V19}.

In the rest of the paper our principal task will be to 
 prove the following lemma by investigating the exponential sum $D_q, S_{q, d\uc}$ and the exponential integral $I_q$. This lemma immediately proves Theorem~\ref{thm:application}. 
\begin{lemma} \label{lem:Ni-bounds} Let $\varepsilon>0$ be any constant. With the notation from Lemma~\ref{lem:N1P}, for all sufficiently small  \(\delta\) (depending on $s$ and $\varepsilon$)
we have
\begin{enumerate}[$(a)$]
    \item \label{HB-delta-part}
$
N_0(P, \delta)=P^{s-4}\mathfrak S \mathfrak J  +O\Big( P^{s-4-1/3+\ve}\Big)$;
\item \label{HB-delta-part2} $
N_1(P, \delta)\ll P^{s-4-(s-8)/3+\ve};$
    \item \label{new-delta-part}
    $
N_2(P, \delta)\ll \begin{cases} P^{s-4-(s-10)/3-1/6+\ve} &\text{ for even $s$ and $s\geq 10$},\\ P^{s-4-(s-9)/3-1/15+\ve} & \text{ for odd $s$ and $s\geq 9$ under GLH,}\\
 P^{s-4-(s-11)/3-1/15+\ve} & \text{ for odd $s$ and $s\geq 11$.}\end{cases}$
\end{enumerate}
Here all the implicit constants may depend on \(s,w,F_1,F_2\) and \(\ve\).
\end{lemma}

The proof of parts \eqref{HB-delta-part} and \eqref{HB-delta-part2} resembles the circle method with Kloosterman refinement and arcs of the usual ``Dirichlet'' size~\cite{Heath-Brown96,HBP}, whereas part \eqref{new-delta-part} has more in common with the nested delta-method~\cite{BM} or the function field two-dimensional delta-method~\cite[Lemma 5.1, section~6]{V19}. The structure of the proof of Lemma \ref{lem:Ni-bounds} is summarized in Figure \ref{fig:Proof structure}.  

\begin{proof}[Proof of Theorem~\ref{thm:application}]
	Combining \eqref{eq: N split} with Lemma~\ref{lem:Ni-bounds}, we obtain the result by choosing $\delta$ sufficiently small. 
\end{proof}

\subsection{Geometry of pairs of quadratic forms, and four sorts of primes}\label{sec:geometry}

To estimate the exponential sums and exponential integrals in Lemma~\ref{lem:N1P}, we need some properties of the geometry of non-singular intersection of two quadrics. The reader may consult the summary of Browning--Munshi~\cite[\S2.2]{BM} as well as Heath-Brown--Pierce~\cite[Section 2]{HBP}. Suppose $\nF=(F_1,F_2)$ is a pair of quadratic forms in $s$ variables such that the projective variety \footnote{By a variety over \(K\), we mean a \emph{reduced} separated scheme of finite type over \(\Spec K\). In fact, all varieties occurring here will be quasiprojective; thus the reader is free to interpret \emph{variety} as an \emph{open subset of a reduced, not necessarily irreducible, projective variety}.} defined by $F_1(\x)=F_2(\x)=0$ is non-singular over $\bar\QQ$. Let $M_1$ and $M_2$ denote integer matrices defining the quadratic forms $F_1$ and $F_2$. This projective variety has a dual variety, which is an absolutely irreducible hypersurface of degree $4(s-2)$ when $s\geq 4$. (See ~\cite[\S2.2]{BM} for more details;  the proof implicitly requires the absolute irreducibility of the smooth complete intersection \(F_1=F_2=0\), which holds for \(s\geq 4\), and the result of Aznar referred to there is a composite of~\cite[Theorems~2-3]{Aznar}.)

In what follows we will therefore take the dual variety to be defined by $\scrF^*(\uu)=0$, where  \(\scrF^*\) is a homogeneous polynomial with integer coefficients of degree $4(s-2)$ which is irreducible over \(\bar{\QQ}\).
Given a primitive integer vector $\uc$, let $F_\uc$ be the quadratic form
\begin{equation*}
F_\uc:=c_1F_1+c_2F_2,
\end{equation*}
defined by the integer matrix \[M_\uc:=c_1M_1+c_2M_2.\] Let $F(x,y)=\det (xF_1+yF_2)$ and
$D_F=2 \text{Disc}(F)$
where $\text{Disc}(F)$ denotes the discriminant of $F(x,y)$ as a binary form.

Over the splitting field $K$ we have 
\begin{align*}
\det (xF_1+yF_2)=h^{-1}\prod_{i=1}^s (\lambda_i x+\mu_i y),
\end{align*}
for some $h\in \NN, \lambda_i, \mu_i\in \mathcal O_K$.

Let $S,T\in \GL(s,\ZZ)$ be invertible integral matrices appearing in a Smith normal form of the matrix $M_\uc=TDS$ where $D$ is a diagonal integer matrix with entries $\rho_1|\cdots|\rho_s$.  For \(\det(M_{\uc})\neq 0\), let $F^*_\uc$ denote the dual form for the quadratic form $F_\uc$ defined by 
\begin{equation}\label{F^*cdef}
F^*_\uc(\uu):=\det(M_\uc)\uu^T M_\uc^{-1}\uu.
\end{equation}
For a fixed $\uc$, this is a quadratic polynomial in the variable $\uu$.  For a fixed $\uu$, this is a polynomial of degree $s-1$ in the projective variable $\uc$ whose discriminant is given by the equation of the dual variety \(\scrF^*(\uu)\).
 
From~\cite[Section 2]{HBP}, we know that $\operatorname{rank}M_\uc\geq s-1$ by the non-singular condition on $\nF$. A primitive integer pair $\uc$ is called bad if the matrix $M_\uc$ is singular, otherwise it is called good. Note that there are at most $s$ pairs of bad $\uc$.  We divide the primes into the following categories:
\begin{enumerate}
	\item \textbf{bad primes:} $p\mid D_F$,
	\item \textbf{good primes for bad $\uc$:} $p\nmid D_F$,
	\item \textbf{good primes of Type I for good $\uc$:} $p\nmid D_F, p\nmid \det(M_{\uc})$,
	\item \textbf{good primes of Type II for good $\uc$:} $p\nmid D_F, p\mid \det(M_{\uc})$.
\end{enumerate}
When $\uc$ is bad, the form $F_\uc$ is singular and in particular $\rho_s=0$. Similarly when $p$ is a good prime of type II for good $\uc$, we have $p\mid \det (M_\uc)$ and thus $p\mid \rho_s$. Let $\y_j=S^{-1}\e_j$ where $\{\e_j\}_{j=1}^s$ denotes the standard basis of $\ZZ^s$ and $S$ is the invertible integer matrix appearing in the Smith normal form for $M_\uc$. Note that $\{\y_i\}_{j=1}^s$ also forms a basis of $\mathbb Z^s$. Let $Q_\uc$ denote the restriction of $F_\uc$ to the $s-1$ dimensional subspace spanned by $\y_1,...,\y_{s-1}$. Namely, 
\begin{equation*}
Q_\uc(y_1,...,y_{s-1}):=F_\uc(y_1\y_1+...+y_{s-1}\y_{s-1}).
\end{equation*}
Similarly, $Q^*_\uc$ denotes the dual form to the quadratic form $Q_\uc$. Given $\uu\in \RR^s$, let $\uu'$ denote the $s-1$ dimensional vector by deleting the $s$-th entry of $(S^{-1})^T\uu$. 
Since $\{\y_j\}_{j=1}^s$ forms a basis for $\ZZ^s$, if $\uu\in\ZZ^s$, so is the vector $\uu'$. 

We remark that for bad $\uc$ and good $p$, if we write $\Delta(Q_\uc)$  for the determinant of \(Q_\uc\), then we have $p\nmid \Delta(Q_\uc)$. Indeed if $p\mid  \Delta(Q_\uc)$ then the rank of \(M_\uc\) over \(\overline{\FF_p}\) is \(\leq s-2\), and so \(F_1=F_2=0\) is singular over \(\overline{\FF_p}\) and $D_F=0$ over $\overline \FF_p$ by~\cite[Proposition 2.1]{HBP}, hence \(p \mid D_F\). 

 For a non-zero vector $\uw$, define
 \begin{equation}
 \label{eq:lambdaucdef}
 \lambda_\uw:=\min_{1\leq j\leq s}\{|(\lambda_j,\mu_j)\cdot\uw|/|\uw|\},
 \end{equation}
 which measures the distance of the unit vector $\uw/|\uw|$ from the lines $\lambda_j x+\mu_j y=0$. 
Note $\lambda_\uc=0$ if and only if $\uc$ is bad.

Before going into the details, we summarize in Table \ref{tab:limits} the conditions on the number of variables $s$ required for our method. Although $s$ must be a natural number of the indicated parity, the lower bound formally required for the last step in the argument can serve as a measure of the difficulty in making further improvements.
\begin{table}[htbp]
\renewcommand{\arraystretch}{1.5}
\centering
\begin{tabular}{ llllll }
Section         &$\uc$  &$\uu$                      &$s$    &GLH    &Condition
\\
\hline
\ref{sec:V<1}   &any      &$\uu=\vecnull$             &any      &no     &$s>8$
\\
\ref{sec:bad-c-Q-S-nonzero}
                &bad    &$Q_\uc^*(\uu')\neq 0 \text{ or }((S^{-1})^T\uu)_s\neq 0$
                                                    &any     &no     &$s\geq 9$
\\
\ref{sec:bad-c-Q-S-zero}
                &bad    &$Q_\uc^*(\uu')=(S^{-1})^T\uu)_s=0$
                                                    &any      &no     &$s\geq 9$
                                                    \\
\ref{sec:good-c-even-s-F-nonzero}
                &good   &$F_\uc^*(\uu)\neq 0$       &even   &no     &$s>8\frac 45$
\\
\ref{sec:good-c-even-s-F-zero}
                &good   &$F_\uc^*(\uu)=0$           &even   &no     &$s>9\frac 12$
\\
\ref{sec:odd_s_nonsquare_Fstar_N211}
                &good   &$F_\uc^*(\uu)\neq\square$  &odd    &yes    &$s>8\frac 45$
\\                
\ref{sec:odd_s_nonsquare_Fstar_N211}
                &good   &$F_\uc^*(\uu)\neq\square$  &odd    &no    &$s>10\frac 45$
\\
\ref{sec:Fcusq} &good   &$F_\uc^*(\uu)=\square$     &odd    &no   &$s>8\frac 12$.

\end{tabular}
\vspace{\baselineskip}
\caption{Limits of the argument. }
\label{tab:limits}
\end{table}

\section{Bounds for quadratic exponential integrals}\label{sec:exp integral}

Recall the exponential integral defined in \eqref{eq:Izu def}. After a change of variables, we see that
\[I_q(\uw, \uu)=P^s\int_{\R^s}w\left(\x\right)e(P^2\uw\cdot \nF(\x))e(-P\x\cdot\uu/q)d\x.\]
We have the following pointwise bound for $I_{q}(\uw, \uu)$. This will be used to study  $N_0(P,\delta)$ and also to prove some sharper estimates later in this section.

 \begin{lemma}\label{lem:IQ}
 If we let 
  $\delta>0$, then for $|\uu|\geq \frac{q}{P}(1+P^2|\uw|)P^\delta$, we have 
 \[I_q(\uw,\uu)\ll_N P^s(P|\uu|/q)^{-N}\ll P^{s-N\delta}\]
 for any $N>0$. 
 Moreover for  any $\uu\in\R^s$,
\begin{align*}|I_q(\uw,\uu)|&\ll P^s\prod_{j=1}^s(1+P^2|\lambda_j w_1+\mu_j w_2|)^{-1/2}\\ &\ll P^s(1+P^2|\uw| )^{-(s-1)/2}(1+P^2|\uw|\lambda_\uw)^{-1/2},\end{align*}
where $\lambda_\uw$ is defined in \eqref{eq:lambdaucdef}.
 \end{lemma}
 \begin{proof}
The lemma follows from~\cite[Lemma 3.1, Lemma 2.3]{HBP}.
 \end{proof}

 We also need some estimates of averaged bounds for $I_q(\uw, \uu)$ integrated against $p_{1,q}$ and $p_{2,\ur,k,q}$. The following bound will be used to bound $N_1(P,\delta)$.
	\begin{lemma}\label{p2integral}
		For any $\uu\in \R^s$ we have 
			\begin{align*}
				\int _{|\uw |\asymp W} 
			|	p_{1, q}(\uw)||I_q(\uw, \uu)|d\uw\ll \frac{Q}{q}W^2P^{s}(1+P^2W)^{-s/2-1}(1+|\uw| qQ^{1/2})^{-N}
			\end{align*}
			for any $N>0$. 	\end{lemma}
	\begin{proof}
	The lemma follows from~\cite[Lemma 3.3]{HBP} together with \eqref{eq:p1bound2} on the decay properties of $p_{1,q}$ and the relation $P^2=Q^{3/2}$.  
		\end{proof}

 Given $k\in \mathbb N$ and $\ur\in \mathbb Z^2 $, we write for short
	\begin{align*}
	p_2(q, \uw)=:p_{2,\ur,k,q}(\uw).
	\end{align*}
    The rest of the section will be dedicated to bounds which will be useful towards bounding $N_2(P,\delta)$.
		\begin{lemma}\label{p1integral} 
		 Let $k\in \mathbb N$ and $\ur \in \mathbb Z^2$ with $|\ur|\asymp Q^{1/2}$ and $\uc=\ur/\gcd(\ur)$.  Let \(\mu_{S^1}\) be the Lebesgue measure (arc-length measure) on the unit circle \(S^1\subset \R^2\).
		For any $W>0$, we have for $j=0,1$ 
			\begin{align*}
			&
			\int _{|\uw |\asymp W} 
			|q^j\frac{\partial^j}{\partial q^j}p_{2}(q, \uw)||I_q(\uw, \uu)|d\uw\\&\ll \mathds 1_{\frac{kq}{Q}<1}
			WQ^{-3/2} P^{s}
			( 1+kqQ^{1/2}W)^{-N}
			(1+P^2W)^{-\frac{(s-1)}{2}}(1+P^2W\lambda_\uc)^{-1/2},
			\end{align*}
   and, where $\lambda_\uc$ is as in \eqref{eq:lambdaucdef}, 
			\begin{align*}
			&
			W^2\int _{S^1} 
			|W^j\frac{\partial^j}{\partial W^j}p_{2}(q, W\uv)||I_q(W\uv, \uu)|d\mu_{S^1}(\uv)\\&\ll \mathds 1_{\frac{kq}{Q}<1}
			WQ^{-3/2} P^{s}
			( 1+kqQ^{1/2}W)^{-N}
			(1+P^2W)^{-\frac{(s-1)}{2}}(1+P^2W\lambda_\uc)^{-1/2},
			\end{align*}
			the right hand side being the same in both of these two estimates.
		\end{lemma}
\begin{proof}
 For \(\uw=(w_1,w_2)\in \mathbb R^2\), we define its radial coordinates by \(\uw=W\uv\) with \(W\geq 0\) and \(\uv\in S^1\).

We begin by considering  the derivative $\frac{\partial}{\partial W}p_{2}(q, W\uv)$ for $\uv \in S^1$.
For fixed \(\uv\), we have the identity \(\frac{d}{dW}f(x(W),y(W)) = \frac{dx}{dW} \frac{\partial f}{\partial x}+ \frac{dy}{dW} \frac{\partial f}{\partial y}\) with \(x(W) = \frac{\ur\cdot \uw}{|\ur|}\), \(y(W) = \frac{\ur^\perp\cdot \uw}{|\ur|}\) and \(f(x,y)=p_{2,\ur,k,q}(\uw)\), which gives
\begin{align*}
		\frac{\partial}{\partial W}
		p_2(q,\uw)
&=
\left(\frac{\ur}{|\ur|} \cdot \frac{d x(W)}{d W}\right)
 \partial_{\ur}p_2(q,\uw)
 +
 \left(\frac{\ur^\perp}{|\ur|}  \cdot\frac{d y(W)}{d W}\right)
 \partial_{\ur^\perp}p_2(q,\uw)
 \\
&=
\frac{\uv\cdot \ur}{|\ur|}	
 \partial_{\ur}p_2(q,\uw)
 +
\frac{\uv\cdot \ur^\perp}{|\ur|}	
 \partial_{\ur^\perp}p_2(q,\uw),
\end{align*}
 upon writing \(\partial_{\vec{\xi}} = \frac{\xi_1}{|\vec\xi|}\frac{\partial}{\partial w_1}+\frac{ \xi_2}{|\vec\xi|}\frac{\partial}{\partial w_2}\). 
 This gives $$
		W\frac{\partial}{\partial W}
		p_2(q,\uw) 
		=
		O(W)
		\partial_{\ur }
		p_2(q,\uw)
		+
		O\left(\frac{|W\uv \cdot \ur^\perp|}{Q^{1/2}}\right)
		\partial_{\ur ^\perp}
		p_2(q,\uw).
$$
 By Lemma~\ref{lem:nonstationary}, it follows that
 \begin{align*}
			&W\frac{\partial}{\partial W}
		p_2(q,\uw) \\&\ll (kqQ^{1/2}W+ {|W\uv \cdot \ur^\perp|} Q)(1+kqQ^{1/2}W)^{-N} (1+|W\uv \cdot \ur^\perp|Q)^{-N} \\
  &\ll (1+kqQ^{1/2}W+{|W\uv \cdot \ur^\perp|} Q)(1+kqQ^{1/2}W)^{-N} (1+|W\uv \cdot \ur^\perp|Q)^{-N} \\
	\end{align*}
 for any $N>0$. Thus for any $N>0$,
 \begin{equation}
     	\label{eq:nonstationary_radial}
		W\frac{\partial}{\partial W}
		p_2(q,\uw)\ll
		( 1+kqQ^{1/2}W )^{-N+1}(1+|W\uw\cdot \ur^\perp| Q)^{-N+1}.
 \end{equation}

Now let the differential operator $D$ be defined either by \(D=q\frac{\partial}{\partial q}\), or by \(D=W\frac{\partial}{\partial W}\).
From  Lemma~\ref{lem:IQ} and from Lemma~\ref{lem:nonstationary} for $D=q\frac{\partial}{\partial q}$ and \eqref{eq:nonstationary_radial} for \(D=W\frac{\partial}{\partial W}\), we obtain the following estimate for $j=0, 1$, $kq<Q$ and $\tau \in (1,2]$:
\begin{align*}
&\int _{W\leq |\uw |\leq  \tau W}|D^j p_{2}(q,\uw)||I_q(\uw, \uu)|d\uw\\&\ll_N P^s\int _{W\leq |\uw |\leq  \tau W}(1+kqQ^{1/2}|\uw|)^{-N}(1+Q\gcd(\ur)|\uc^\perp\cdot\uw|)^{-N}\\ & \quad \quad \quad \quad \quad \quad \times\prod_{j=1}^s(1+P^2|\lambda_i w_1+\mu_i w_2|)^{-1/2}d\uw.
\end{align*}
We change variables to write $\uw=x\frac{\uc}{|\uc|}+y\frac{\uc^\perp}{|\uc|}$ so that 
\begin{align*}
\MoveEqLeft \int _{W\leq |\uw |\leq  \tau W}|D^j p_{2}(q,\uw)||I_q(\uw, \uu)|d\uw\\&\ll_N P^s\int _{W\leq |\uw |\leq  \tau W}(1+kqQ^{1/2}x)^{-N}(1+Q^{3/2}y)^{-N}\\ & \quad \quad \times \prod_{j=1}^s\Big(1+P^2|(\lambda_j,\mu_j)\cdot (x\frac{\uc}{|\uc|}+y\frac{\uc^\perp}{|\uc|})|\Big)^{-1/2}dx dy,
\intertext{which is}
&\ll_N P^sQ^{-3/2}\int _{W\leq x\leq \tau W}(1+kqQ^{1/2}x)^{-N}\\ &\quad \quad \times \prod_{j=1}^s(1+P^2|(\lambda_j,\mu_j)\cdot (x\uc/|\uc|)|)^{-1/2}dx,
\end{align*} 
{and using \cite[Lemma 2.3]{HBP} this is}$$
\ll (\tau -1)WQ^{-3/2}P^{s}
( 1+kqQ^{1/2}W )^{-N}
(1+P^2W )^{-\frac{(s-1)}{2}}(1+P^2W\lambda_\uc)^{-1/2}.
$$
Taking 
\(D=q\frac{\partial}{\partial q}\), 
this concludes the first part of the lemma. Alternatively, dividing both sides by $(\tau -1)$, letting $\tau$ go to $1$, and taking \(D=W\frac{\partial}{\partial W}\), we obtain the second part of the lemma.
\end{proof}
When $|\uu|\leq \frac{q}{P}(1+P^2W)P^\varepsilon,$ we will also need a bound on average value of the derivative $\frac{\partial}{\partial q}I_q(\uw,\uu) $ integrated against $p_2(q, \uw)$ when we consider a double Kloosterman refinement to make use of cancellations from the $q$-sum. 
\begin{lemma}\label{p1derivative}
	Under the same assumptions as in Lemma~\ref{p1integral}, we have the following bound for any $N\geq 0$ and any interval $I\subseteq[1,2]$
	\begin{align*}
		&
		\int _{|\uw |/ W\in I} 
		p_2(q,\uw)\frac{\partial}{\partial q}I_q(\uw, \uu)d\uw\\&\ll
		\mathds 1_{\frac{kq}{Q}<1}
		q^{-1}WQ^{-\frac{3}{2}} P^{s}( 1+kqQ^{1/2}W)^{-N}
		(1+P^2W)^{-\frac{(s-1)}{2}}(1+P^2W\lambda_\uc)^{-\frac{1}{2}}.
	\end{align*}
\end{lemma}

\begin{proof}
	We note that
	\begin{equation*}
		\frac{\partial}{\partial q}I_q(\uw,\uu)=q^{-1}P^s\int_{\R^s}\frac{2\pi i P\x\cdot\uu}{q}w\left(\x\right)e(P^2\uw\cdot \nF(\x))e(-P\x\cdot\uu/q)d\x.
	\end{equation*}
	Treating the factor $\frac{P\x\cdot\uu}{q}$ trivially will result in an extra factor of $1+P^2 |\uw|$ in case $
		\frac{P|\uu\cdot\x|}{q}\ll P^\ve(1+P^2|\uw|)$, as Lemma~\ref{lem:IQ} suffices otherwise. 
	The goal is to remove the factor $P\x \cdot \uu/q$ and we achieve this by integration by parts. 
 We first write 
	\begin{align*}
		q \frac{\partial}{\partial q}I_q(\uw, \uu)
		\nonumber
		&=-P^s  \int_{\R^s}
		w\left(\x\right)e(P^2\uw\cdot \nF(\x)) \x \cdot \nabla e(-P\x\cdot\uu/q)
		 \,
		d\x\,\\
		\nonumber
		&=
		P^s  \int_{\R^s}
		\text{div}\left(
		w\left(\x\right)e(P^2\uw\cdot \nF(\x))\x
		\right)
		e(-P\x\cdot\uu/q)\,
		d\x,
	\end{align*}
	where we applied the divergence theorem in the second equality. Now we use the fact that $\nF$ is homogeneous, so that $\sum_{j=1}^s x_j w(\x)\frac{\partial}{\partial x_j}  e(P^2 \uw\cdot \nF(\x))= 4\pi i P^2 w(\x) \uw \cdot \nF(\x).$ It follows that
\begin{align}\nonumber
q \frac{\partial}{\partial q}I_q(\uw, \uu)={}&
-P^s  \int_{\R^s}
		\bigg(\sum_{j=1}^s \frac{\partial}{\partial x_j} x_j w(\x)\bigg)
		e(P^2\uw\cdot \nF(\x))
		e(-P\x\cdot\uu/q)\,
		d\x 
		\\    \label{eq:penultimate}		
		&
		+P^s  \int_{\R^s} 4\pi i P^2
		w(\x)  (\uw \cdot \nF(\x)) e(P^2\uw \cdot \nF(\x)) e(-P\x\cdot\uu/q)\,
		d\x,
\end{align}
	The first term in \eqref{eq:penultimate} has the same form as \(I_q\) with the weight \(w\) replaced by $\sum_{j=1}^s \frac{\partial}{\partial x_j} x_j w(\x)$, whose contribution in the integral against $p_2(q, \uw)$ can be dealt with using first part of Lemma~\ref{p1integral}.
	To handle the second term in \eqref{eq:penultimate}, we change to radial coordinates and then remove the factor $P^2 \uw \cdot \nF(\x)$ by integration by parts. 
	
	As in Lemma~\ref{p1integral}, let \(\mu_{S^1}\) be the Lebesgue measure (arc-length measure) on the unit circle \(S^1\subset \R^2\). Then, multiplying \eqref{eq:penultimate} by \( p_2(q,\uw)\)  and integrating over \(|\uw|\asymp W\), we must bound
	\begin{align*}
		&
        \int_{C_1W \leq |\uw|\leq C_2W} p_2(q,\uw)
		P^s  \int_{\R^s}
		4\pi i  P^2 \uw \cdot \nF(\x) w(\x) 
		e(-P\x\cdot\uu/q)\,
		d\x\,
		d\uw&
		\\
		&
        =
		\int_{S^1}
		\int_{C_1W}^{C_2W}
		tP^s p_2(q,t\uw) \int_{\R^s}
		4\pi i  t P^2 \uw \cdot \nF(\x) w(\x) &
		\\ & \hspace{4cm}
        \times e(P^2t\uw\cdot \nF(\x))
		e(-P\x\cdot\uu/q)\,
		d\x\, dt \,d\mu_{S^1}(\uw).&
	\end{align*}
        Using integration by parts with respect to \(t\) (and the product rule on the resulting derivative), this is equal to
$2P^st^2(B-I_1-I_2)$, where
		$$
		B=
		\int_{S^1}
		\int_{\RR^s} 
		p_2(q,t\uw)
		 w(\x) 
		e(P^2t\uw\cdot \nF(\x))
		e(-P\x\cdot\uu/q)\,
		d\x\
		\Big \vert_{C_1W}^{C_2W}d\mu_{S^1}(\uw),$$
        and
		\begin{align*}
		    I_1&=  2
		\int_{S^1}
		\int_{C_1W}^{C_2W}\int_{\RR^s}
		t^{-1}
		p_2(q,t\uw)
		w(\x) 
		\\ & \quad \quad \quad \quad \quad \quad \times e(P^2t\uw\cdot \nF(\x))
		e(-P\x\cdot\uu/q)\,
		d\x\, dt \,d\mu_{S^1}(\uw),
		\end{align*}
		and 
		$I_2$ is given by $$
		\int_{S^1}
		\int_{C_1W}^{C_2W}
		\int_{\RR^s}
		\left(
		\frac{\partial}{\partial t}
		p_2(q,t\uw)
		\right)
		w(\x)  e(P^2t\uw\cdot \nF(\x))
		e(-P\x\cdot\uu/q)\,
		d\x\, dt \,d\mu_{S^1}(\uw)
		\label{eq:Iq-change-vars}.
	$$
	The second term is satisfactory by the first part of Lemma~\ref{p1integral}, and the other terms are satisfactory by the second part of Lemma~\ref{p1integral}.
\end{proof}

\section{Bounds for quadratic exponential sums}\label{sec:exp sum}	In this section, we obtain bounds for two types of quadratic exponential sums: the sums $D_q(\uu)$ in \eqref{eq:Dq def} and $S_{q,d\uc}(\uu)$ in \eqref{eq:sq def}.  The exponential sums $D_q(\uu)$ are closely related to $S_{d,1}(\uu)$, defined in~\cite[eq. (1.2)]{BM}. The exponential sums $S_{q,d\uc}(\uu)$ are closely related to the exponential sums studied in~\cite{V19} in the function field setting. We will adapt the methods in~\cite{BM} and~\cite{V19} to study these sums in our setting.  

Using the Chinese remainder theorem, we have the following multiplicativity properties for $D_q$ and $S_{q, d\uc}$. 
\begin{lemma}\label{lem:multiplicative}
Let $s, q_1,q_2, d_1, d_2\geq 1$ be integers satisfying $\gcd(q_1,q_2)=1$ and $d_1\mid q_1, d_2\mid q_2$. Then for any primitive integer vector $\uc\in\ZZ^2$, and any integer vector $\uu\in \ZZ^s$, we have the following multiplicativity relations:
\begin{align*}
D_{q_1q_2}(\uu)&=D_{q_1}(\uu)D_{q_2}(\uu),
\\
S_{q_1q_2,d_1d_2\uc}&=S_{q_1,d_1\uc}(\uu)S_{q_2,d_2\uc}(\uu)
\end{align*}
\end{lemma}

\subsection{Bounds for $D_q(\uu)$} In this subsection, our aim is to prove Lemma~\ref{lem:Dqfinal}. This is essentially a well-known estimate for for $D_q$. In addition to handling $N_0(P,\delta)$ and $N_1(P,\delta)$, a special case of this result is gives Lemma~\ref{badpairgoodp} in the next subsection, which plays an important role for the contributions from bad paris of $\uc$ to $N_2(P,\delta)$.

Using the multiplicativity in Lemma \ref{lem:multiplicative}, it is enough to consider the case when $q=p^k$ where $p$ is prime and $k\in \NN$.
  
\begin{lemma}\label{lem:Dp bound}
For $\uu \in \ZZ^s$ we have
\begin{align*}
D_p(\uu)\ll \begin{cases}
p^{(s+2)/2} & p\nmid \scrF^*(\uu),\\
p^{(s+3)/2} & p\mid \scrF^*(\uu), p\nmid \uu.
\end{cases}
\end{align*}
Here, $\scrF^*(\uu)$ denotes the dual variety for the intersection of quadrics $F_1$ and $F_2$ defined in Section~\ref{sec:geometry}.
\end{lemma}
\begin{proof}
	We write 
    \begin{align*}D_p(\uu)=\sum_{\ua\bmod p}\sum_{\bb\bmod p}e_p(\ua \cdot \nF(\bb)+\bb\cdot \uu)-\sum_{\bb\bmod p}e_p(\bb\cdot \uu).\end{align*} 
    The lemma follows from~\cite[Lemma 19]{BM} since the second sum vanishes unless $p\mid \uu$. 
\end{proof}
\begin{lemma}\label{lem:Dpk bound}
	Let $k\geq 2, \uu \in \ZZ^s$, we have $D_{p^k}(\uu)=0$ unless $p\mid D_F\mathcal F^*(\uu)$.
\end{lemma}
\begin{proof}
	We write 
    \begin{align*}D_{p^k}(\uu)&=\sum_{\ua \bmod {p^{k}}}\sum_{\bb \bmod p}e_{p^k}(\ua \cdot \nF(\bb)+\bb\cdot \uu)\\
    & \quad -p^{s}\mathds 1_{p\mid \uu}\Big(\sum_{\ua \bmod {p^{k-1}}}\sum_{\bb \bmod p^{k-1}}e_{p^{k-1}}(\ua \cdot \nF(\bb)+\bb \cdot \frac{\uu}{p})\Big).\end{align*} 
    If $p\nmid D_F\mathcal F^*(\uu)$, then $p\nmid \uu$ and the lemma then follows from an application of Hensel's lemma akin to~\cite[Lemma 20]{BM}. 
\end{proof}
\begin{lemma}\label{lem:Dq-crude}

For any $q\geq 1$ and  $\uu \in \ZZ^s$, we have
	$D_q(\uu)\ll q^{s/2+2+\varepsilon}$ for any $\varepsilon >0$.
\end{lemma}
\begin{proof}
By Lemma \ref{lem:multiplicative}, it is enough to check the bound for $q=p^k$ for some integer $k$. After a standard squaring argument such as~\cite[Lemma 2.5]{V19},
\begin{equation*}
|\sum_{\bb \bmod q}e_{q}(\ua\cdot \uF(\bb)+\bb\cdot \uu)|\ll q^{s/2}Z(\ua;q)^{1/2},
\end{equation*}
where $Z(\ua;q)=\#\{\z\bmod q: q\mid \z^T (a_1M_1+a_2M_2),\}$
where $M_1$ and $M_2$ are matrices defining the quadratic forms. As a result,
\begin{align*}
|D_{p^r}(\uu)|&\ll p^{ks/2}\starsum_{\ua\bmod p^k}Z(\ua;p^k)^{1/2}\ll p^{ks/2+k}(\starsum_{\ua\bmod p^k}Z(\ua;p^k))^{1/2}\\ &\ll p^{ks/2+2k+\varepsilon}=q^{s/2+2+\varepsilon},
\end{align*}
for any $\varepsilon>0$
upon using~\cite[Lemma 5.4]{HBP}.
\end{proof}

Combining Lemmas~\ref{lem:Dp bound}--\ref{lem:Dq-crude}, we obtain the following upper bound for $D_q(\uu)$.
\begin{lemma}\label{lem:Dqfinal}
	Suppose \(q=q_1q_2\) where $q_1$ is the square-free part of $q$ and $q_2$ is the square-full part of $q$.  Then for any $\varepsilon>0$
	\[
	D_q(\uu)\ll q^{s/2+1+\varepsilon}
	\gcd(q_1,\scrF^*(\uu))^{1/2}
	\gcd(q_1, \uu)^{1/2}\gcd(q_2,(D_F\scrF^*(\uu))^\infty)
	.
	\]
\end{lemma}

\subsection{Bounds for $S_{q,d\uc}(\uu)$}

Recall  that the sum $S_{q, d\uc}(\uu)$ was defined in  \eqref{eq:sq def}.
The function field analogue of similar exponential sums defined in~\cite[eq. (6.1)]{V19} have been extensively studied there.
When $(d,q)=1$, the exponential sum $S_{q, \uc}(\uu)$ coincides with that considered in~\cite[Section 6]{V19} in the function field setting and analogous estimates can be proved similarly. By Lemma \ref{lem:multiplicative}, it is enough to consider sums $S_{p^k,\uc}(\uu)$. In the next four lemmas, we will obtain estimates for $S_{p^k, \uc}$ when $p$ is a good prime.
\begin{lemma}[Type I primes for good $\uc$ ]\label{lem:type 1} Let $\uc$ be good. 
For $p\nmid 2\det M_{\uc}$, we have
	\begin{align*}
	|S_{p^k, \uc}(\uu)|\leq\begin{cases}
	p^{sk/2}\gcd(F_\uc^*(\uu), p^{k}) & 2\mid s \text{ or } 2\nmid s, 2\mid k,\\ 
	p^{k(s+1)/2}\gcd (F_\uc^*(\uu), p^{k})^{1/2}\mathds {1}_{F_{\uc}^*(\uu)\not=0} & 2\nmid s, 2\nmid k.
	\end{cases}
	\end{align*}
\end{lemma}
\begin{proof}
Since $\uc$ is primitive, we see that $p^k\mid \uc \cdot \ua^{\perp}$ in \eqref{eq:sq def} implies that $\ua=a\uc$ for some $(a,p)=1$ so that \eqref{eq:sq def} can be written as 
\begin{align*}
S_{p^k, \uc}(\uu)&=\starsum_{a\bmod p^k}\sum_{\bb\bmod {p^k}}e_{p^k}(a\uc\cdot \uF  (\bb)+\bb\cdot \uu)\\&=\starsum_{a\bmod p^k}\sum_{\bb\bmod {p^k}}e_{p^k}(a F_{\uc}(\bb)+\bb\cdot \uu).
\end{align*}
From~\cite[Lemma 15]{BM}, we have 
for $p\nmid 2\det M_{\uc}$ 
	\begin{align}\label{Spkucuu}S_{p^k, \uc}(\uu)=\begin{cases}
p^{sk/2}\varepsilon(p)^{sk}\chi_p(\det M_\uc)^k c_{p^k}(F_\uc^*(\uu)) & 2\mid s,\\
p^{sk/2} c_{p^k}(F_\uc^*(\uu)) & 2\nmid s, 2\mid k,\\
p^{sk/2} \varepsilon(p)^{s}\chi_p(-1)g_{p^k}(F_\uc^*(\uu)) & 2\nmid s, 2\nmid k,
\end{cases}
\end{align}
	where $$\varepsilon(p)=\begin{cases}
	1   &p\equiv 1 \bmod 4\\ 
	i & p \equiv 3\bmod 4,
	\end{cases}$$  $\chi_p(\cdot)=(\frac{\cdot}{p})$ is the Legendre Symbol, $F_{\uc}^*$ defined in \eqref{F^*cdef} is the adjoint quadratic form of $F_{\uc}$, $c_{p^k}(a)=\starsum_{x\bmod{p^k}}e_{p^k}(ax) $ is the Ramanujan sum, and
	$g_{p^k}(a)=\sum_{x\bmod {p^k}}\chi_p(x)e_{p^k}(ax)$ is the Gauss sum.
The lemma follows from the bounds $|c_{p^k}(a)|\leq (a, p^k)$ and $|g_{p^k}(a)|\leq p^{k/2}(a, p^k)^{1/2}\mathds 1_{a\not=0}$.
\end{proof}
While dealing with a pair of quadrics over $s=9$ variables, we will need to make use of cancellations over averages of the exponential sums $S_{q,\uc}(\uu)$ over $q$,  which is the content of the following lemma.
\begin{lemma}\label{hypo:GRH}
Let $s$ be an odd integer, $\uc\in\ZZ^2$ be a good pair, and let $\uu\in\ZZ^s$ be such that $F_\uc^*(\uu)$ is not a perfect square. Assume the generalized Lindel\"of Hypothesis for Dirichlet $L$-functions. Then for any integer $\mathfrak B$ such that $2\det M_{\uc}\mid \mathfrak B$ and any $\varepsilon>0$, we have that 
\begin{equation*}
|\sum_{\substack{1\leq q\leq x\\ \gcd(q,\mathfrak B)=1 }}S_{q,\uc}(\uu)|\ll (\mathfrak B|\uc| |\uu|)^\varepsilon x^{s/2+1+\varepsilon}.
\end{equation*}
\end{lemma}
\begin{proof}
This can be viewed as a conditional version of~\cite[Lemma 18]{BM} and can be proved in the same way. Let $$\xi_{\mathfrak B}(z, \uu)=\prod_{p\nmid \mathfrak B}\Big(\varepsilon(p)^s\chi_p(-1)\sum_{2\nmid k}\frac{g_{p^k}(F_\uc^*(\uu))}{p^{k(z-s/2)}}+\sum_{2\mid k}\frac{c_{p^k}(F_{\uc}^*(\uu)}{p^{k(z-s/2)}}\Big).$$
Since $g_p(a)=\chi_p(a)\varepsilon(p)p^{1/2}$, we see that 
$$\xi_{\mathfrak B}(z, \uu)=L(z-\frac{s+1}{2}, \psi_{\uu})E_{\mathfrak B}(z),$$
where $\psi_{\uu}(\cdot)=\Big(\frac{(-1)^{\frac{s+1}{2}}F_{\uc}^*(\uu)}{\cdot}\Big)$ is a Dirichlet character with conductor of size $O(|\uc|^{s-1}|\uu|^2)$ and $E_{\mathfrak B}(z)$ is an Euler product, which is absolutely convergent when $\Re(z)>s/2+1$ and satisfies the bound $E_{\mathfrak B}(z)\ll {\mathfrak B}^\varepsilon$, for any $\varepsilon>0$. Thus $\xi_{\mathfrak B}(z, \uu)$ can be analytic continued to the region $\Re(z)> s/2+1$. Using Perron's formula, and the identity \eqref{Spkucuu} which we quoted from~\cite[Lemma 15]{BM}, we have for $c>s/2+2$
\begin{align*}
\sum_{\substack{1\leq q\leq x\\ (q, \mathfrak B)=1}}S_{q,\uc}(\uu)=\frac{1}{2\pi i} \int_{c-iT}^{c+iT} \xi_{\mathfrak B}(z, \uu)\frac{x^z}{z}dz+O(\frac{x^c}{T}).
\end{align*}
Moving the line of integration to $s/2+1+\varepsilon$, using the bound $\xi_{\mathfrak B}(z, \uu)\ll (|\uc|^{s-1}|\uu|^2\mathfrak B T)^\varepsilon$, and taking $T$ to be a sufficiently large power of $x$, we obtain the result. 
\end{proof}

\begin{lemma}[Type II primes for good $\uc$]\label{lem:type II}Let $\uc$ be good. Let $M_{\uc}=TDS$ where $T, S\in GL_s(\ZZ)$ and $D=\diag(\rho_i)$ with $\rho_i\mid \rho_{i+1}$.
	Suppose $p\nmid D_F$ and $p\mid \det(M_{\uc})$ and let $k_1=\min(k, \nu_p(\rho_s))$, where for any integer $n$, $\nu_p(n)$ is the largest integer $k$ such that $p^k\mid n$. Then
	$S_{p^k}(\uc, \uu)$
	is non-zero unless $p^{k_1}\mid ((S^{-1})^T\uu)_s$, and we also have
	\begin{align*}
	|S_{p^k,\uc}( \uu)|\ll p^{k(s/2+1)}\gcd(p^{k_1},Q_{\uc}^*(\uu') ,((S^{-1})^T\uu)_s)^{1/2},
	\end{align*}
	where $\uu'$ denotes the $s-1$ dimensional vector by deleting the $s$-th entry of $(S^{-1})^T\uu$.
\end{lemma}
\begin{proof}
This can be proved in the same way as the function field analogue~\cite[Lemma 6.3]{V19}.
\end{proof}

We now move on to bad pairs $\uc$. In this case, the matrix $M_\uc$ is singular and therefore the last diagonal entry $\rho_s=0$ in the Smith normal decomposition of $M_\uc$ as stated before. We also have $$F_\uc(\x)=Q_\uc(\x'),$$
where as before $\x'$ denotes the $s-1$ dimensional vector by deleting the $s$-th entry of $(S^{-1})^T\uu$.
\begin{lemma}[Good primes for bad $\uc$]\label{badc}
	Let $\uc$ be bad. Then for a good prime $p$, we have $S_{p^k,\uc}( \uu)$ vanishes unless $p^k\mid ((S^{-1})^T\uu)_s$, and we have
	\begin{align*}
	&|S_{p^k,\uc}(\uu)|\\ &\ll p^k\mathds{1}_{p^k\mid ((S^{-1})^T\uu)_s}\begin{cases}
	p^{k(s/2-1/2)}\gcd(p^{k},Q_{\uc}^*(\uu'))&\textrm{ if }2\mid k \text{ or } 2\nmid k, 2\nmid s,\\
	p^{ks/2}&\textrm{ if }2\nmid k, 2\mid s.
	\end{cases}
	\end{align*}
\end{lemma}
\begin{proof}
This can be proved in the same way as the function field analogue~\cite[Lemma~6.4]{V19}.
\end{proof}
Next, we consider general bounds for $S_{q, d\uc}$ with $(q, d)>1$. When $(d, q)>1$, the definition of $S_{q, d\uc}(\uu)$ differ slightly from that in~\cite[Section 6]{V19} due to the co-primality condition in $L(d\uc)$ defined in~\cite{V19}, but the same bounds apply by following the proofs in~\cite[Lemma 6.5, Lemma 6.7]{V19}. By Lemma \ref{lem:multiplicative}, it is again enough to consider $S_{p^k, p^m \uc}$ with $1\leq m \leq k.$ We start with the following lemma corresponding to~\cite[Lemma 6.7]{V19}, which will be used to obtain a general bound for $S_{q, d\uc}$.
\begin{lemma}\label{lem:Spkpm}
	Let $\uc\in \ZZ^2$ be primitive. For integers $1\leq m\leq k$, denote $k_1=\min(k-m, \nu_p(\det M_{\uc}))$, where $\min(k-m,\infty)=k-m$ by convention. Then we have 
	\begin{align*}
	|S_{p^k, p^m\uc}(\uu)|\ll_{D_F} p^{k(n/2+1)+m+k_1/2}\mathds{1}_{p^{k_1}\mid ((S^{-1})^T\uu)_s}.
	\end{align*}
\end{lemma}
\begin{proof}
	Akin to~\cite[Lemma 3.3]{V19}, we see that 
	\begin{align*}
	&\{\ua \bmod {p^k}:(\ua, p)=1, p^{k}\mid p^{m}\uc \cdot \ua \}\\ =&\{a\uc^\perp +p^{k-m}\ud\bmod {p^k}: (a, p)=1, a\bmod p^{k-m}, \ud \bmod{p^m} \},
	\end{align*}
	so that we can write \begin{align*}
	&\sum_{\substack{\ua \bmod {p^k}\\ (\ua, p)=1\\ p^{k-m}\mid \ua \cdot \uc }}\sum_{\bb \bmod {p^k}}e_{p^k}(\ua \cdot F(\bb)+\bb \cdot \uu)\\=&\sum_{\ud \bmod {p^{m}}} \sum_{\substack{a\bmod {p^{k-m}}\\ (a,p)=1}}e_{p^k}((a\uc^\perp +p^{k-m}\ud )\cdot F(\bb)+\bb\cdot \uu).
	\end{align*}
	Then the proof follows the same way as~\cite[Lemma 6.7]{V19}, using that uniformly for $(a, p)=1$, the following holds:
	\begin{align*}
	\sum_{\ud \bmod {p^{m}}}\operatorname{gcd}(F(a\uc^\perp+p^{k-m}\ud ), p^k)^{1/2}\ll p^{2m+k_1/2},
	\end{align*}
since $\uc$ is primitive so that  $(a\uc^\perp+p^{k-m}\ud, p)=1$ for all $(a, p)=1$ and $\ud\in \ZZ^2$.
\end{proof}
As a corollary of Lemma~\ref{lem:Spkpm}, we obtain the following analogue of~\cite[Lemma 6.8]{V19}.
\begin{lemma}[general bound]\label{weakbound} For $d\mid q$, we have
	\begin{align*}
	|S_{q, d\uc}( \uu)|&\ll_{D_F}dq^{s/2+1}\gcd(q/d, ((S^{-1})^T\uu)_s, \det M_{\uc})^{1/2}\\&\ll_{\nF} d^{1/2}q^{s/2+3/2}.
		\end{align*}
\end{lemma}
To estimate the contribution from bad pairs $\uc$, we need the following refined estimate, that saves a factor of order $O(d^{1/2})$ as compared to the bound in Lemma~\ref{weakbound} for square-free moduli, analogous to~\cite[Lemma 6.5]{V19}.
\begin{lemma}\label{badpairgoodp}
	Let $\uc$ be a bad pair. Then
	\begin{align*}
	|S_{p, p\uc}(\uu)|\ll p^{s/2+1}(p,\uu)^{1/2}\gcd(p, \scrF^*(\uu))^{1/2}.
	\end{align*}
\end{lemma}
\begin{proof}
Since
$S_{p,p\uc}(\uu)=D_p(\uu)$, this lemma follows from Lemma~\ref{lem:Dqfinal}.
\end{proof}

\section{Major arcs contribution: the main term \texorpdfstring{$N_0(P, \delta)$}{N0(P, δ)}}\label{sec:major arc}
In this section we prove the first part of Lemma~\ref{lem:Ni-bounds}\eqref{HB-delta-part}, the asymptotic formula for the main contribution $N_0(P, \delta)$ which was defined in \eqref{eq:N0def}. 

First we show that only $\uu=\vecnull$ contribute to the main term in $N_0(P, \delta)$.
\begin{lemma}\label{lem:major1}
	For any $\delta, N>0$, we have
\[N_0(P, \delta)=\sum_{1\leq q< Q^{1/2-\delta}}q^{-s}
			D_q(\vecnull)			
			\int_{\uw\in \mathfrak M_q(\delta) } I_q(\uw,\vecnull)
			\,d\uw+O_{N}(P^{-N}).\]
\end{lemma}
\begin{proof}
Note that for $Q=P^{4/3}$, $q\leq Q^{1/2-\delta}$, and any $\uw\in \mathfrak M_q(\delta)$, we have \[\frac{q}{P}(1+P^2|\uw|)\ll \frac{q}{P}(1+P^2\frac{1}{qQ^{1+\delta}})\ll \frac{q}{P}+\frac{P}{Q}\ll P^{-1/3}.\]
Thus from the first part of Lemma~\ref{lem:IQ}, we conclude that the contributions from all non-zero $\uu$ are negligible in $N_0(P, \delta)$. \end{proof}
To further evaluate $N_0(P, \delta)$, we define the singular series
\begin{equation}\label{eq:SS}
\mathfrak{S}:=\sum_{q=1}^\infty q^{-s}D_q(\vecnull)
\end{equation}
and the singular integral $\mathfrak{J}$ 
\begin{equation}\label{eq:SI}
\mathfrak{J}:=\int J(\uw)d\uw, 	\text{ where } J(\uw)=\int w(\x)e(\uw\cdot \uF(\x))d\x.
\end{equation}
\begin{lemma}\label{N_0asymp}
The singular integral $\mathfrak{J}$ is absolutely convergent as soon as $s\geq 6$ and the singular series $\mathfrak{S}$ is absolutely convergent for any $s\geq 7$. Moreover, for any $s\geq 7$ and any $\ve>0$, we have
\begin{equation}\label{eq:2222}
N_0(P, \delta)=P^{s-4}\left(\mathfrak{SJ}+O(P^{-(s-6)/3+\ve})\right)
\end{equation}
for sufficiently small \(\delta\) (depending on $\varepsilon$ and $s$).
\end{lemma}
\begin{proof}
First we show the convergence of $\mathfrak J$ and $\mathfrak S$.
Using Lemma~\ref{lem:IQ}, we see that 
\begin{align*}
J(\uw)\ll (1+|\uw|)^{-(s-1)/2},
\end{align*}
and it follows that 
\begin{align}\nonumber
\int_{|\uw| \gg W}J(\uw)d\uw &\ll \int_{|\uw|\gg W}(1+|\uw|)^{-(s-1)/2}d\uw \\ &\ll \int_{r\gg W}r^{-(s-1)/2+1}dr\ll W^{-s/2+5/2}.
\label{singluarintegraltail}
\end{align}
The singular integral is therefore absolutely convergent for $s\geq 6$.
From Lemma~\ref{lem:Dqfinal},
we have the bound \begin{equation}
\label{Dq0}
|D_q(\vecnull)|\ll q^{s/2+2+\varepsilon},
\end{equation}
for any $\varepsilon>0$, which implies that 
\begin{align*}
\sum_{q\geq 1}q^{-s}|D_q(\vecnull)|\ll \sum_{q\geq 1}q^{-s/2+2+\varepsilon}
\end{align*}
for any $\varepsilon>0$.
Therefore, the singular series $\mathfrak{S}$ converges absolutely for all $s\geq 7$.

Next we evaluate $N_0(P, \delta)$. 
Using Lemma~\ref{lem:major1}
and a change of variables,
\begin{align*}
\int_{|\uw|<q^{-1}Q^{-1-\ve}} I_q(\uw,\vecnull)=P^{s-4}\int_{|\uw|<q^{-1}Q^{1/2-\delta}}J(\uw)d\uw,
\end{align*}	
we see that 
\begin{equation}
\label{eq:222}
N_0(P, \delta)=P^{s-4}\sum_{1\leq q< Q^{1/2-\delta}}q^{-s}
			D_q(\vecnull)			
			\int_{|\uw|<q^{-1}Q^{1/2-\delta}} J(\uw)
			\,d\uw+O_{N}(P^{-N}).
\end{equation}
Using \eqref{singluarintegraltail}, we can
we can replace the integral over $\uw$ by $\mathfrak J$ with an error $$O\big((q^{-1}Q^{1/2-\delta})^{-s/2+5/2}\big)=O(q^{(s-5)/2}Q^{-1+ \varepsilon}),$$ for any $s\geq 7$, provided \(\delta\) is sufficiently small depending on $\varepsilon$ and $s$.

Using \eqref{Dq0}, it then follows that for $s\geq 7$ and for sufficiently small \(\delta\) we have
\begin{align*}
&\sum_{1\leq q\leq Q^{1/2-\delta}}q^{-s} D_q(\vecnull) \int_{|\uw|<q^{-1}Q^{1/2-\delta}} J(\uw)\,d\uw\\
&=\mathfrak{J}\sum_{1\leq q< Q^{1/2-\delta}}q^{-s}
			D_q(\vecnull)			
			+O\Big(\sum_{1\leq q< Q^{1/2-\delta}}q^{-s/2+2}q^{(s-5)/2}Q^{-1+\varepsilon}\Big)\\
&=\mathfrak{JS}+O\Big(\sum_{q\geq Q^{1/2-\delta}} q^{-s/2+2+\varepsilon}+Q^{-1+\varepsilon}\sum_{q\leq Q^{1/2-\delta}}q^{-1/2} \Big)\\
&=\mathfrak{JS}+O(Q^{-(s-6)/4+\varepsilon}+Q^{-3/4+ \varepsilon}))=\mathfrak{JS}+O(P^{-(s-6)/3+\ve}),
\end{align*}
 which together with \eqref{eq:222} yields \eqref{eq:2222}.
\end{proof} 

\section{Counting estimates: dimension growth and congruences}\label{sec:counting}
In this section, we prepare some lemmas to estimate the contributions from the $\uu$-sum and $\uc$-sum. Our first estimate is essentially the uniform version of the dimension growth theorem of Salberger~\cite[Theorem~0.3]{Sal}, however we will use an affine version. 
\begin{lemma}[Dimension growth]\label{dimgrowth}
	Let $X\subset \PP^{s-1}$ be an irreducible variety over $\QQ$ satisfying $\dim X\geq 1$, $\deg X \geq 2$ and $s\geq 3$. Let
 \[
\eta(d) =
\begin{cases}
    0&(d \neq 3),\\
    \frac{2}{\sqrt{3}}-1&(d=3).
\end{cases}
 \]
 
 Let \(\mathcal{Y}\subset \AA_\ZZ^{s}\) be an integral model of the affine cone over \(X\). Then for $V\geq 1$ and any $\varepsilon>0$
	\begin{align*}
	\{\x \in \mathcal{Y}(\QQ)\cap\ZZ^{s}: |\x|\leq V \}\ll_{\deg X,s,\varepsilon} V^{\operatorname{dim}X+\eta(\deg X)+\varepsilon}.
	\end{align*}
\end{lemma}

We remark that at one point, namely \eqref{eq:more-counting}, we need very slightly more than the last proposition in the degree 2 case, and draw on~\cite[Lemma 3.7]{V19}.

\begin{proof}
Let \(d=\deg X\).
When \(d\geq 3\), Salberger proved~\cite[Theorem~0.3]{Sal} that
\begin{align}\label{eqnnn}
	\{\x \in \mathcal{Y}(\QQ)\cap\ZZ^{s}:  \gcd(\x)=1, |\x|\leq V \}\ll_{d,s,\varepsilon} V^{\operatorname{dim}X+\eta (d)+\varepsilon}.
	\end{align}
 Salberger states this for integral varieties; recall that for us all varieties are reduced, and thus integral and irreducible are interchangeable terms.
 
 We prove \eqref{eqnnn} in the case  \(d=2\). Suppose first that \(X\) is geometrically reducible. In that case \(X\) is a union of two hyperplanes conjugate over a quadratic extension of \(\QQ\). All the points in \(X(\QQ)\) lie on the intersection of these hyperplanes, which is a projective linear space of dimension \(\dim X-1\); thus \(\mathcal Y(\QQ)\) is contained in an affine linear space of dimension \(\dim X\) and  therefore makes an acceptable contribution to the counting function in the lemma. In the remaining case when \(X\) is geometrically irreducible, \eqref{eqnnn} is proved in Heath-Brown~\cite[Theorem~2]{HeathBrown02}.

  We now need to remove the $\gcd$-condition in the above estimate. Clearly, using \eqref{eqnnn},
 \begin{align*}
  \{\x \in \mathcal{Y}(\QQ)\cap\ZZ^{s}: |\x|\leq V \}\ll \sum_{1\leq m\leq V} (V/m)^{\dim X+\varepsilon}\ll V^{\dim X+\eta (\deg X)+\varepsilon},
 \end{align*}
 as soon as $\dim X\geq 1$. 
\end{proof}
We shall apply the following lemma to control the dimensions of a type of varieties appearing in our analysis. 
 \begin{lemma}\label{lem:intersection-of-dual-varieties}
  	Let \(\uc\in\mathbb{Z}^2\) be a good, primitive integer vector, suppose \(s\geq 4\) and let \(X\subseteq \PP^{s-1}\) be the variety defined by $\scrF^*(\uu)= F^*_\uc(\uu)=0$ . Then
  	\begin{enumerate}[(i)]
  		\item \(X\) is a complete intersection of dimension \(s-3\) defined over \(\QQ\).
  		\item \(X\) does not contain any linear space of projective dimension \(\geq \lfloor\frac{s}{2}\rfloor\).
    \item If \(s \geq 7\) then \(X\) does not have any component of degree~$3$.
  	\end{enumerate}
  \end{lemma}
  
  \begin{proof}
  	For part (i), it suffices to show that  \(F^*_\uc=0\) and \(\scrF^*=0\) do not have a common irreducible component over \({\QQ}\). But every component of  \(F^*_\uc=0\) has degree \(\leq 2\), while \(\scrF^*=0\) is irreducible with degree \(4(s-2)\geq4\), so we are done.
  	
  	For part (ii), observe that \(\uc\) is good, so \(F^*_\uc\) is a quadratic form of rank \(s\). By Witt's Theorem, the maximum possible index of isotropy of  \(F^*_\uc\) is \(\lfloor\frac{s}{2}\rfloor \), that is
   \(F^*_\uc=0\) cannot contain a  linear space of \emph{projective} dimension  \(>\lfloor\frac{s}{2}\rfloor -1\) and \textit{a fortiori} the same is true of \(X\).

    For (iii), suppose \(V\subset \mathbb P^{s-1}\) is such a component.
    Then \(\deg V = 1+\operatorname{codim} V\), making \(V\) a \emph{variety of minimal degree}. As \(V\) has degree 3 it is a cone over a rational normal scroll~\cite[Theorem~1]{EH}. In particular there are linear spaces \(L\subseteq V\) with codimension 1~\cite[p5]{EH}. Now by the previous part (ii), we have \(\dim(V)-1 <\lfloor\frac{s}{2}\rfloor\), while by part (i) we have \(\dim(V) = s-3\). This implies \(s <7\).
  \end{proof}

  \begin{remark}
  The proof of (iii) is inspired by the mathoverflow answer of Sasha~\cite{sasha}, which would actually permit us to weaken the hypothesis in (iii) to \(s\geq 6\) by a closer inspection of the possibilities for \(V\) at the end of our argument. Presumably (ii) is not optimal either, but it is enough for our purposes.
  \end{remark}

We will also need the following counting lemma generalizing~\cite[Lemma 8.1]{V19}.
\begin{lemma}[Counting lemma]\label{counting}
	For any $s\geq 2$, $q,d\in \NN$ with $d\mid q$, any primitive $\uc\in \ZZ^2$ and any $V\geq 1$, we have that 
	\begin{align}\nonumber
	&\#\{|\uu|\leq V: q\mid ((S^{-1})^T \uu)_s, d\mid Q_{\uc}^*(\uu')\}\\
    \label{eq:Vsum1}
    &\ll V^{s-2}\min\Big(V(1+\frac{V}{q}), (1+\frac{V}{\prod_{p\mid d}p})(1+\frac{V}{\prod_{p\mid q}p})\Big),
	\end{align}
	and for $N\geq 0$, $x\in\NN$ and any $\varepsilon>0$, 
	\begin{align}\label{eq:B1}
	\sum_{\substack{\uc\textrm{ good }\\|\uc|\asymp C\\ x\mid \det M_{\uc}}}(1+N\lambda_\uc)^{-1/2}\ll_{F, \varepsilon} x^\varepsilon C \Big(1+ \frac{C}{x^{1/2}(1+N)^{1/2}}\Big),
	\end{align}
	where $\lambda_\uc$ is as defined in \eqref{eq:lambdaucdef}.
\end{lemma}
\begin{proof}
Under the assumptions of the lemma,	equation \eqref{eq:Vsum1} can be proved the same way as its function field analogue in~\cite[eq. (8.6)]{V19} and we omit the details here. 

Equation \eqref{eq:B1} is a generalization of~\cite[eq. (8.7)]{V19}. We show that for any $x_1\mid x$ satisfying $x_1\geq x^{1/2}$
\begin{align}\label{eq:B1'}
	\sum_{\substack{\uc\textrm{ good }\\|\uc|\asymp C\\ x\mid \det M_{\uc}\\\gcd(x_1,c_1)=1}}(1+N\lambda_\uc)^{-1/2}\ll_{F, \varepsilon} x^\varepsilon C \Big(1+ \frac{C}{x_1(1+N)^{1/2}}\Big),
	\end{align}
	which implies \eqref{eq:B1}. 

	   Fix a value of $c_1$ co-prime to $x_1$. Following the proof of~\cite[eq. (8.7)]{V19}, one can show that there exist integers $0\leq b_1,...,b_K<x_1$ where $ K=O_{D_F}(x_1^\varepsilon)$ such that every $|c_2|\ll C$ satisfying $\gcd(c_1,c_2)=1$ and $x_1\mid \det(M_\uc)$ must be of the form $b_i+kx_1 $ for some $|k|\ll 1+C/x_1$. Therefore
	\begin{align*}
	\sum_{\substack{\uc\textrm{ good }\\|\uc|\asymp C\\ x\mid \det M_{\uc}\\ \gcd(c_1,x_1)=1}}(1+N\lambda_\uc)^{-1/2}\ll \sum_{\substack{|c_1|\ll C\\ \gcd(c_1,x)=1}}\sum_{j=1}^K \sum_{|k|\ll 1+C/x_1}(1+N\lambda_{(c_1,b_j+kx_1)})^{-1/2}.
	\end{align*}
	Now, we have uniformly for any $0\leq b<x_1$ and $c_1 $
	\begin{align*}
	&\sum_{|k|\ll C/x_1}(1+N\lambda_{(c_1,b+kx_1)})^{-1/2}
	\\&\ll 1+\sum_{0\not=|k|\ll C/x_1}(1+Nkx_1/C)^{-1/2}\\
	&\ll 1+\sum_{0\not=|k|\leq \frac{C}{x_1(1+N)}}1+\sum_{\frac{C}{x_1(1+N)}\ll |k|\ll \frac{C}{x_1}} C^{1/2}N^{-1/2}x_1^{-1/2}k^{-1/2}
	\\&\ll1+ \frac{C}{x_1(1+N)^{1/2}},
	\end{align*}
 which is enough to establish \eqref{eq:B1'}, upon recalling that $K=O_{D_F}(x^\varepsilon)$.
\end{proof}

	\section{Minor arcs contribution: the term \texorpdfstring{$N_1(P, \delta)$}{N1(P,δ)}} \label{sec:N1}

    The objective of this section is to prove the second bound claimed in Lemma~\ref{lem:Ni-bounds}\eqref{HB-delta-part2}, concerning the quantity
     $N_1(P, \delta)$ defined in \eqref{eq:N1def},
    when $\delta$ is sufficiently small depending on $\varepsilon$ and $s$.

 \subsection{First reductions.}
	From Lemmas~\ref{lem:N1P} and~\ref{lem:IQ}, we see that 
	\begin{align}\nonumber
	\MoveEqLeft N_1(P, \delta)+O_{N,\varepsilon}(P^{- N})\\
    & =\sum_{1\leq q\leq  Q}q^{-s}\sum_{\substack{\uu \in \ZZ^s\\ |\uu|\leq \frac{q}{P}(1+P^2|\uw|)P^\varepsilon}}D_q(\uu)\int_{\scrm_q(\delta)}p_{1, q}(\uw)I_q(\uw, \uu)d\uw,\label{eq:N1_prepare_to_split}
	\end{align}
 for any \(\varepsilon,N>0\). Although we write $O_{N,\varepsilon}$ in the line above for clarity, we recall before taking the next step that our implicit constants can always depend on \(s,w,F_1,F_2\) and \(\varepsilon\), and that \(\varepsilon\) can vary between occurrences.
		 Moreover we split $q, \uw$ into dyadic ranges $q\asymp Y,|\uw| \asymp W$, we see that $\uu$ can be truncated up to $V$, 
	 where 
	   \begin{equation}\label{defV}
	   V=\frac{Y}{P}(1+P^2W)P^\varepsilon.
	   \end{equation}
    In view of the minor arcs $\mathfrak m_q(\delta)$ defined in \eqref{minorarc}, we introduce the set $\mathcal C_\delta$ as
	     \begin{align}
  \label{YWcondition}
  \mathcal C_\delta=\Big\{ (Y, W):
  1\leq Y\leq Q^{1/2-\delta}, \frac{Q^{-\delta}}{YQ}\ll W\ll \frac{Q^\delta}{YQ^{1/2}} \\ \text{or } Q^{1/2-\delta}\leq  Y \leq Q, W\ll \frac{Q^\delta}{YQ^{1/2}}
  \Big\}.\nonumber
    \end{align}
    Every pair $(q,\uw)$ occurring in \eqref{eq:N1_prepare_to_split} satisfies $Y\leq q\leq 2 Y, W\leq |\uw|\leq 2 W$  for some $(Y,W)\in \mathcal C_\delta$. Therefore  we have
    \begin{align*}
        &N_1(P, \delta)\\ &\ll   P^{\varepsilon} \max_{(Y, W)\in \mathcal C_\delta}\bigg\lvert \sum_{q\asymp Y}q^{-s} \sum_{\substack{\uu\in \mathbb Z^s\\ |\uu|\leq V}}D_q(\uu) \int_{|\uw| \asymp W} p_{1, q}(\uw) I_q(\uw, \uu)d\uw\bigg\rvert+P^{-N},
	 \end{align*}
  where, by a slight abuse of notation, we embrace both of the conditions $W\leq |\uw|\leq 2 W, \uw\in \mathfrak m_q(\delta)$ in the notation $|\uw|\asymp W$, so that the implicit constants are always in $[1,2] $, and  are usually 1 and 2, but may differ when $W$ is almost  as large or as small as possible. 
  
 	   	   Writing $q=q_1q_2$ where $q_1$ is square free, $q_2$ is square-full with $q_1\asymp Y_1, q_2\asymp Y_2$ and $Y_1Y_2\asymp Y$, we can apply  
	  Lemma~\ref{p2integral} to the $\uw$-integral and Lemma~\ref{lem:Dqfinal} to $D_q(\uu)$ to obtain the bound
  \begin{align*}
  N_1(P, \delta)
  &\ll P^\varepsilon \max_{\substack{Y_1Y_2\asymp Y\\ (Y, W)\in \mathcal C_\delta}}Y^{-s/2}QW^2P^s(1+P^2W)^{-s/2-1}\\
  &\times\sum_{q_i\asymp Y_i}\sum_{|\uu|\leq V}\gcd(\scrF^*(\uu),q_1)^{1/2}\gcd(\uu,q_1)^{1/2}\gcd(q_2,(D_F\scrF^*(\uu))^\infty).
  \end{align*}
 To estimate the $\uu$-sum, we consider the contribution from $\scrF^*(\uu)\not=0$ and $\scrF^*(\uu)= 0$ separately. 
  \subsection{Case I: $\mathcal F^*(\uu)\not=0$}The contributions from $\mathcal F^*(\uu)\not=0$ can be bounded by
\begin{align}\nonumber
\MoveEqLeft P^\varepsilon \max_{\substack{Y_1Y_2\asymp Y\\ (Y, W)\in \mathcal C_\delta}} Y^{-s/2}QW^2P^s(1+P^2W)^{-s/2-1} \sum_{\substack{|\uu|\leq V\\ \scrF^*(\uu)\neq 0}}\sum_{q_i\asymp Y_i}\\\nonumber
\MoveEqLeft[1]\gcd(\scrF^*(\uu),q_1)^{1/2} \gcd(\uu,q_1)^{1/2}\gcd(q_2,(D_F\scrF^*(\uu))^\infty)\\ \nonumber
&\ll P^\varepsilon \max_{\substack{Y_1Y_2\asymp Y\\ (Y, W)\in \mathcal C_\delta}} Y^{-s/2}QW^2P^s(1+P^2W)^{-s/2-1} \sum_{\substack{|\uu|\leq V\\ \scrF^*(\uu)\neq 0}}Y_1Y_2\\ \nonumber
&\ll  P^\varepsilon \max_{\substack{ (Y, W)\in \mathcal C_\delta}} Y^{1-s/2}QW^2P^s(1+P^2W)^{-s/2-1}V^s\\
&\ll  P^\varepsilon \max_{\substack{(Y, W)\in \mathcal C_\delta}}  Y^{1+s/2}QW^2 (1+P^2W)^{s/2-1}.\label{eq:121}
\end{align}
The right hand side of \eqref{eq:121} is maximized when $W\asymp Y^{-1}Q^{-1/2+\delta}$ and $Y\asymp Q$, which is
\begin{align*}
    &\ll  P^{\varepsilon}   Y^{1+s/2}QW^2 (Q/Y)^{s/2-1}Q^{(s/2-1)\delta}\ll  P^\varepsilon  Q^{s/2-1+(s/2+1)\delta}.
\end{align*}
By choosing $\delta$ sufficiently small (depending on $\varepsilon$ and $s$),
this contribution is bounded by 
\begin{align*}
Q^{2+s/2+\ve}P^{-4}\ll P^{s-4-(s-8)/3+\ve},
\end{align*}
handing us the bound in Lemma~\ref{lem:Ni-bounds}\eqref{HB-delta-part2}. 
  \subsection{Case II: $\mathcal F^*(\uu)=0$}
To deal with the contribution from $\scrF^*(\uu)=0$ terms, we use Lemma~\ref{dimgrowth} to obtain
$$\#\{|\uu|\leq V:\scrF^*(\uu)=0,\uu\neq \vecnull\} \ll V^{s-2+\varepsilon}.$$
The contribution from $\mathcal F^*(\uu)=0$ terms is therefore bounded by 
\begin{align}\nonumber
&\max_{\substack{Y_1Y_2\asymp Y\\ (Y, W)\in \mathcal C_\delta}}Y^{-s/2}QW^2P^s(1+P^2W)^{-s/2-1}\\\nonumber
&\times\sum_{\substack{|\uu|\leq V\\ \scrF^*(\uu)= 0}}\sum_{q_i\asymp Y_i}\gcd(\scrF^*(\uu),q_1)^{1/2}\gcd(\uu,q_1)^{1/2}\gcd(q_2,(D_F\scrF^*(\uu))^\infty)\\
\nonumber
&\ll \max_{\substack{Y_1Y_2\asymp Y\\ (Y, W)\in \mathcal C_\delta}} 
Y^{-s/2+\ve}QW^2P^{s}(1+P^2W)^{-s/2-1}\\\nonumber
&\quad\quad \quad \quad\quad \quad \times\sum_{q_i\asymp Y_i}\Big(Y_1Y_2+\sum_{\substack{|\uu|\leq V\\ \scrF^*(\uu)= 0\\ \uu\neq\vecnull}}(\uu,q_1)^{1/2}Y_1^{1/2} Y_2\Big)\\
\nonumber
&\ll \max_{\substack{Y_1Y_2\asymp Y\\ (Y, W)\in \mathcal C_\delta}} Y^{-s/2+\ve}Y_1^{1/2}Y_2 QW^2 P^{s}(1+P^2W)^{-s/2-1}\\ \nonumber &\quad\quad \quad \quad\quad \quad \times  \Big(Y_1^{1/2}\sum_{q_i \asymp Y_i}1+\sum_{\substack{|\uu|\leq V\\ \scrF^*(\uu)= 0\\ \uu\neq\vecnull}}\sum_{q_i\asymp Y_i}(\uu,q_1)^{1/2}\Big)\\
\nonumber
&\ll \max_{\substack{Y_1Y_2\asymp Y\\ (Y, W)\in \mathcal C_\delta}} P^\ve Y^{-s/2}Y_1^{1/2}Y_2QW^2 P^s(1+P^2W)^{-s/2-1}\\ \nonumber &\quad\quad \quad \quad\quad \quad \times\Big(Y_1^{3/2}Y_2^{1/2}+V^{s-2}Y_1Y_2^{1/2}\Big)\\
&\ll \nonumber \max_{\substack{(Y, W)\in \mathcal C_\delta}}Y^{2-\frac{s}{2}}Q W^2 P^{s+\ve}(1+P^2W)^{-\frac{s}{2}-1}\\ &\quad\quad \quad \quad\quad \quad + Y^{\frac{s}{2}-1/2}QW^2P^{2+\ve}(1+P^2W)^{\frac{s}{2}-3}\label{E12bound}.
\end{align}
Since $s\geq 5$, the first term in \eqref{E12bound} is maximal when $Y\asymp 1, W\asymp  Q^{-1-\delta}$ or $Y\asymp Q^{1/2-\delta},W\asymp Q^{-3/2+\delta}\asymp P^{-2+\delta}$, which is bounded by
\begin{align*}
&\ll Q^{-1}P^{s+\ve} (P^2/Q)^{-s/2-1}Q^{(s/2+1)\delta}+Q^{2-s/4}P^{s-4+\ve}Q^{(s/2+1)\delta}\\&\ll P^{s-4-(s-8)/3+\ve}Q^{(s/2+1)\delta}
\end{align*}
The second term in \eqref{E12bound} is maximal when $W\asymp Y^{-1}Q^{-1/2+\delta}$ and further when $Y\asymp Q$, as the resulting power of $Y$ in this expression is positive. Therefore, we have, $Y\asymp Q,W\asymp Q^{-3/2+\delta}\asymp P^{-2+\delta}$, which gives a bound
\begin{align*}
\ll Q^{s/2+1/2+\ve }P^{-4+2+\ve}Q^{(s/2-1)\delta}
\ll P^{s-4-(s-8)/3+\ve}Q^{(s/2-1)\delta}.
\end{align*}
By choosing $\delta$ sufficiently small, the contributions from $\mathcal F^*(\uu)=0$ also satisfy the bound in Lemma~\ref{lem:Ni-bounds}\eqref{HB-delta-part2}. 

\section{Minor arcs contribution: reduction to \texorpdfstring{$N_{2,1}(P, \delta)$ and $N_{2,2}(P, \delta)$}{N21(P, δ) and N22(P, δ)}}\label{sec:preparing-N2}
In this section, we divide up the sum $N_2(P,\delta)$ (see \eqref{N2split} below) into several pieces in order to establish the bounds in Lemma~\ref{lem:Ni-bounds}\eqref{new-delta-part} for the term $N_2(P, \delta)$ from \eqref{eq:N2def}
when $\delta$ is sufficiently small depending on $\ve$ and $s$. 

\subsection{Splitting up $N_2(P,\delta)$}\label{sec:N2split}\label{sec:Qdelta-and-Sigma}
From Lemmas~\ref{lem:N1P} and~\ref{lem:IQ}, we see that 
	\begin{multline*}
N_2(P, \delta)
=\sum_{\substack{d,k\in\NN\\ \uc\in\ZZ^2 \mathrm{primitive}\\ \ur=dk|\uc|\asymp Q^{1/2} }}\omega(\frac{\ur}{Q^{1/2}})\sum_{\substack{1\leq q\leq Q/k\\ d\mid q\\ \gcd(q/d,k)=1}} q^{-s}				\sum_{\substack{\uu\in \Z^s\\|\uu|\leq \frac{q}{P}(1+P^2|\uw|)P^\ve}}S_{q,d\uc}(\uu)\\
\int_{\scrm_q}		
			p_{2,\ur,k,q}(\uw) I_q(\uw,\uu)
			\,d\uw+O_{\ve,s,N}(P^{-N}).
	\end{multline*}
After dividing $d, k, \uc, q, \uw$ into dyadic ranges, we obtain 
\begin{equation}\begin{split}\label{N2dyadic}
   N_2(P, \delta)+O_N( P^{-N})\ll{}& P^\varepsilon \max_{\substack{KDC\asymp Q^{1/2}\\ D\ll Y\ll Q/K\\ (Y, W)\in \mathcal C_\delta}}\bigg\lvert\sum_{\substack{d,k\in \NN \\ \uc\in\ZZ^2 \mathrm{primitive}\\ d\asymp D, k\asymp K, |\uc| \asymp C }}\sum_{\substack{q\in \mathbb N\\ d\mid q, q\asymp Y\\\gcd(q/d,k)=1}} q^{-s}	\\ &\times 			\sum_{\substack{\uu\in \Z^s\\|\uu|\leq V}}S_{q,d\uc}(\uu)
   \int_{|\uw| \asymp W}		
			p_{2,\ur,k,q}(\uw) I_q(\uw,\uu)d\uw\bigg\lvert,\end{split}
\end{equation}
where $\mathcal C_\delta$ is defined in \eqref{YWcondition} and $V$ is defined as \eqref{defV}. Here we again include the condition $\uw\in \mathfrak m_q(\delta)$ in the implicit constants in the notation $|\uw|\asymp W$.

\begin{remark}\label{rem:square}
    One could hope to take advantage of cancellations over the \(\uc\)-sum in the general style of Health-Brown--Pierce~\cite{HBP} and Northey-Vishe~\cite{NV} by using Cauchy's inequality  to write 
    \begin{align*}
    &\sum_{\substack{\uc\textrm{ primitive}\\ \ur = dk\uc}}
    \omega\big(\frac{\ur}{Q^{1/2}}\big)
    \sum_{\uu\in \Z^s}
    S_{q,d\uc}(\uu)
    \int_{\uw\in\scrm_q(\delta)}
			p_{2,\ur,k,q}(\uw)I_q(\uw, \uu)\,d\uw
        \\&\ll
 \bigg(       (\frac{Q^{\delta}}{dqk})^2
    \sum_{\substack{\uc\textrm{ primitive}\\ \ur = dk\uc \\ |\ur|\asymp Q^{1/2}}}
    \int_{\uw\in\scrm_q(\delta)}
    \bigg \lvert
\sum_{\uu\in \Z^s}
 S_{q,d\uc}(\uu)
			p_{2,\ur,k,q}(\uw)I_q(\uw, \uu)
    \bigg\rvert^2
    \,d\uw
    \bigg)^{1/2}.    
    \end{align*}
    After bounding the integral, one will have sums of type \(\sum_{\uc}  S_{q,d\uc}(\uu) S_{q,d\uc}(\v)\), in which some extra cancellations might be obtained. We do not explore this here since the saving does not seem to be enough for $s=9$.
\end{remark}
We write 
\begin{align}\label{N2split}
    N_2(P, \delta)\ll N_{2,0}(P, \delta)+N_{2,1}(P, \delta)+N_{2,2}(P, \delta)+O_N(P^{-N}),
\end{align}
where $N_{2,0}(P, \delta)$ denotes the contribution on the right hand side in \eqref{N2dyadic} from $V<1$, while $N_{2,1}(P, \delta)$ denotes the contribution on the right hand side in \eqref{N2dyadic} from $V\geq 1$ and bad pairs $\uc$, and $N_{2,2}(P, \delta)$ denotes the contributions on the right hand side in \eqref{N2dyadic} from $V\geq 1$ and good pairs $\uc$.

From the definitions of $V$ and $\mathcal C_\delta$ in \eqref{defV}, \eqref{YWcondition}, we see that  
the condition $V<1$ implies that the pairs $(Y, W)$ belong to the set $\mathcal C_\delta^-$ defined as follows:
\begin{align}\label{eq:V=0cond}
\mathcal C_\delta^-=\Big\{(Y, W): Y\ll Q^{1/2-\delta}, 
Y^{-1}Q^{-1-\delta}\ll W\ll Y^{-1}P^{-1-\delta} &\\\text{ or }Q^{1/2-\delta}\ll Y\ll P^{1-\delta}, W\ll Y^{-1}P^{-1-\delta}
\Big\}.\nonumber
\end{align}
We will deal with this contribution in the next subsection.

The situation for $V\geq 1$ requires more work. To handle this case, we introduce the notation $\mathcal Q_\delta$ to estimate the contributions from $\mathcal{C}_\delta\setminus\mathcal{C}_\delta^-$:
\begin{align}\label{eq:Qvedef}
\mathcal Q_\delta
=
\Big\{(K, D, C,R,Y, W)
: KDC\asymp Q^{1/2},\, D\ll R\ll  
Y \ll \frac{Q}{K}, \\ (Y, W)\in \mathcal C_\delta, K, D, C, V\geq 1\Big\}.\nonumber
\end{align}
We first relate $N_{2,1}(P,\delta), N_{2,2}(P, \delta)$ to certain averages involving $S_{r,d\uc}$ and $\Sigma$, depending on parameters drawn from the set $\mathcal Q_\delta$. 
\begin{lemma}\label{lem:N2maxbound}
Let $\ve>0$ and $\mathcal Q_\delta$ be as in \eqref{eq:Qvedef}. Then for $\delta$ sufficiently small (depending on $\ve$ and $s$) and for $i=1,2$, we have
\begin{multline}\label{N2maxbound}
N_{2,i}(P, \delta)
\ll P^{\ve}\max_{\substack{ \mathcal Q_\delta}}\sum_{k\asymp K}\sum_{d\asymp D}\sum_{\substack{|\uc| \asymp C\\ \uc \textup{ bad if }i=1\\ \uc \textup{ good if }i=2
}}\sum_{\substack{\uu \in \ZZ^s\\ |\uu|\leq V}}
\sum_{\substack{r\asymp R\\ \gcd(k,r/d)=1\\ d\mid r\mid \mathfrak P^\infty}}\\
r^{-s}|S_{r, d\uc}(\uu)\Sigma(r,Y/R, W, \mathfrak P;k,d,\uc, \uu)|, 
\end{multline}
where $\mathfrak P=\mathfrak P(d, \uc, \uF)$ is any non-zero integer satisfying $d\mid \mathfrak P$, and 
\begin{align*}
\MoveEqLeft\Sigma(r, Y_1, W, \mathfrak P;k,d, \uc, \uu)\\=&
 \int_{|\uw|\asymp W}\sum_{\substack{q_1\asymp Y_1\\ (q_1, k\mathfrak P)=1}}q_1^{-s} S_{q_1,\uc}(\uu)p_{2, dk\uc,k, q_1r}(\uw)I_{q_1r}(\uw, \uu) d\uw.
 \end{align*}
\end{lemma}
\begin{proof}  

Using \eqref{N2dyadic} and \eqref{N2split}, we see that $N_{2,1}(P, \delta)$ (from contributions with $\uc $ bad and $V\geq 1$) can be bounded by 
 \begin{align*}
\ll P^{\ve}\max_{\mathcal{Q}_\delta}\sum_{k\asymp K}\sum_{d\asymp D}\sum_{\substack{|\uc| \asymp C\\ \uc \text{ bad}}}\sum_{\substack{\uu \in \ZZ^s\\ |\uu|\leq V}} N_2(Y, W; k,d, \uc, \uu),
 \end{align*}
 where 
 \begin{align*}
 N_2(Y, W; k, d, \uc, \uu):=\sum_{\substack{d\mid q\asymp Y\\  \gcd(q/d, k)=1}} q^{-s}S_{q, d\uc}(\uu) \int_{|\uw|\asymp W}p_{2, dk \uc, k, q}(\uw) I_{q}(\uw, \uu)d\uw.
 \end{align*}
 Here we again include the condition $\uw\in \mathfrak m_q(\delta)$ in the implicit constants in $|\uw|\asymp W$. 
We split $q=q_1r$, where $\gcd(q_1,\mathfrak{P}k)=1$ and $r\mid \mathfrak{P}^\infty$. Using $d\mid \mathfrak P$ and multiplicativity of the exponential sums in Lemma \ref{lem:multiplicative}, we can write $S_{q, d\uc}(\uu)=S_{r, d\uc}(\uu)S_{q_1, \uc}(\uu)$ with $(q_1,\mathfrak P)=1$ and $r\mid \mathfrak P^\infty$. Since $d\mid q$, we must have $d\mid r$ and the result follows by splitting $r$ into dyadic ranges $r\asymp R$ with $D\ll R\ll Y/K$. This completes the proof of \eqref{N2maxbound} when $i=1$. The proof for $i=2$ follows the same way by adjusting the conditions on $\uc$. 
 \end{proof}
 
In the case of $N_{2,1}$ we further split the sum over $\uu$ in on the right hand of \eqref{N2maxbound} into two cases:
\begin{equation}\label{E-one}
\sum_{\substack{|\uu|\leq V}} (\cdots)
=\sum_{\substack{|\uu|\leq V\\ Q_\uc^*(\uu')\neq 0\\ \textrm{ or }((S^{-1})^T\uu)_s\neq 0}}(\cdots)+\sum_{\substack{|\uu|\leq V\\ Q_\uc^*(\uu')=((S^{-1})^T\uu)_s= 0}}(\cdots).
\end{equation}
We define \(E_{1,1}(P, \delta)\) and \(E_{1,2}(P, \delta)\), corresponding to contributions to the right hand side in \eqref{N2maxbound} for bad $\uc$, from each of the two terms on the right hand side above. Note that when $\uc$ is bad, we have $\lambda_{\uc}=0$.

When  we are considering $N_{2,2}$, we split the sum over $\uu$ on the right hand side of \eqref{N2maxbound} into two cases in a way that depends on the parity of $s$. For odd $s$, we put
\begin{equation}\label{E-two}
\sum_{\substack{|\uu|\leq V}}(\cdots)=\sum_{\substack{|\uu|\leq V\\ F_\uc^*(\uu)\neq 0}}(\cdots)+\sum_{\substack{|\uu|\leq V\\ F_\uc^*(\uu)= 0}}(\cdots).
\end{equation}
We then define sums \(E_{2,1}(P, \delta)\) and \(E_{2,2}(P, \delta)\), corresponding to the contribution from the two terms on the right hand side above to the right hand side of \eqref{N2maxbound} for good $ \uc$, that is for $N_{2,2}$.
For even $s$,
we  instead split the sum over $\uu$ on the right hand side of \eqref{N2maxbound} into 
\begin{equation}\label{eq:123}
\sum_{\substack{\uu\in\ZZ^s\\ |\uu|\leq V}}(\cdots)=\sum_{\substack{\uu\in\ZZ^s\\ |\uu|\leq V\\ F_\uc^*(\uu)\neq\square }}(\cdots)+\sum_{\substack{\uu\in\ZZ^s\\ |\uu|\leq V\\ F_\uc^*(\uu)=\square }}(\cdots).
\end{equation}
 When $s$ is even, then we define the sums \(E_{2,1}(P, \delta)\) and \(E_{2,2}(P, \delta)\) to be the contribution to the bound \eqref{N2maxbound} for $N_{2,2}$, from the first and second terms on the right hand side of \eqref{eq:123} respectively.

We shall obtain the following estimates for $N_{2,i}(P, \delta)$ for $i=0,1,2$. 
\begin{lemma}\label{N2splitbounds}
    Let $\ve>0$ and adopt the notation defined above, so that 
$
N_{2,i}\ll E_{i,1}+E_{i,2}$ for $i=1,2$. Then, for sufficiently small $\delta$ depending on $\ve$ and $s$, we have
    \begin{enumerate}
        \item \label{N20bounds}
       $ N_{2,0}(P, \delta)\ll P^{s-4-(s-8)/3+\varepsilon}$,

    \item \label{N21bounds} $\max_{i=1,2}E_{1,i}(P, \delta)\ll\begin{cases}
      P^{(s-4)-(s-9)/3+\ve} & 2\mid s,\\
      P^{s-4-(s-8)/3+\ve} & 2\nmid s,
    \end{cases}$
    \item \label{N22bounds} $ \max_{i=1,2}E_{2,i}(P, \delta)\ll \begin{cases}
        P^{s-4-(s-10)/3-1/6+\ve} & 2\mid s,\\
        P^{s-4-(s-9)/3-1/15+\ve} & 2\nmid s, \text{ under GLH},\\
        P^{s-4-(s-11)/3-1/15+\ve} & 2\nmid s.
    \end{cases}$
    \end{enumerate}
\end{lemma}

Combining \eqref{N2split} and Lemma \ref{N2splitbounds}, we establish Lemma~\ref{lem:Ni-bounds}\eqref{new-delta-part} for the term $N_2(P, \delta)$ defined in \eqref{eq:N2def}.

In the following sections, we shall first prove Lemma \ref{N2splitbounds}\eqref{N20bounds} in Section~\ref{sec:V<1}. Then we provide some preliminary results in Section~\ref{n2prelim} and use them to prove Lemma \ref{N2splitbounds}\eqref{N21bounds} in Section~\ref{sec:bad-c} and Lemma \ref{N2splitbounds}\eqref{N22bounds} in Section~\ref{sec:good-pairs-unconditional} for even $s$ and Section~\ref{sec:good-pairs-GLH} for odd $s$.

\subsection{Contribution to $N_{2,0}(P,\delta)$ from $\uu=\vecnull$}\label{sec:V<1}
We consider $N_{2,0}(P, \delta)$ in \eqref{N2split} from $V<1$, where only the term $\uu=\vecnull$ contributes.
Recall from the comments preceding \eqref{eq:V=0cond} that  
the condition $V<1$ restricts the pairs $(Y, W)$ appearing in \eqref{N2dyadic} to the set $\mathcal C_\delta^-$.

Therefore, the contribution from $V<1$ terms on the right side of \eqref{N2dyadic} can be bounded by 
\begin{equation}\label{V<1}
\begin{split}
 \max_{\substack{KDC\asymp Q^{1/2}\\ D\ll Y\ll Q/K\\ (Y, W)\in \mathcal C_\delta^-}} Y^{-s} \sum_{k\asymp K, d\asymp D}\sum_{|\uc| \asymp C}\sum_{q\asymp Y}		|S_{q,d\uc}(\vecnull)|\\ \times \int_{|\uw|\asymp W}		
			|p_{2,\ur,k,q}(\uw) I_q(\uw,\vecnull)|
			\,d\uw.
\end{split}
\end{equation}
The integral $|I_q(\uw,\vecnull)|$ can be estimated by applying Lemma~\ref{p1integral} together with the trivial bound $(1+P^2W\lambda_\uc)^{-1/2}\leq 1$. The exponential sum $|S_{q,d\uc}(\vecnull)|$ can be estimated using the first inequality in Lemma~\ref{weakbound}. The sum in \eqref{V<1} can therefore be bounded by
\begin{equation*}
\begin{split}
&\ll W Q^{-3/2}P^{s} (1+P^2W)^{-(s-1)/2}\sum_{\substack{q\asymp Y\\ d\asymp D, d\mid q\\ \uc\asymp C, k \asymp K}}q^{-s}|S_{q, d\uc}(\mathbf{0})|\\
&\ll W Q^{-3/2}P^{s}(1+P^2W)^{-(s-1)/2}\sum_{\substack{d\asymp D\\ k\asymp K}}\sum_{\substack{q\asymp Y\\ d\mid q}}\sum_{\uc\asymp C}q^{-s}dq^{s/2+1}\gcd \big(\frac{q}{d},\det M_{\uc}\big)^{1/2}\\
&\ll W Q^{-3/2}P^{s}(1+P^2W)^{-(s-1)/2}\sum_{\substack{d\asymp D\\ k\asymp K}}\sum_{\substack{q\asymp Y\\ d\mid q}}q^{-s}q^{s/2+1}d\sum_{x\mid \frac{q}{d}}x^{1/2}\sum_{\substack{\uc\asymp C\\ x\mid \det(M_\uc)}}1\\
& \ll Y^\varepsilon WQ^{-3/2}P^{s}(1+P^2W)^{-(s-1)/2}Y^{-s/2+1}D\sum_{\substack{d\asymp D\\ k\asymp K}}\sum_{\substack{q\asymp Y\\ d\mid q}} \sum_{\substack{x\mid \frac{q}{d}}}C(1+\frac{C}{x^{1/2}})x^{1/2},\label{1235,}
\end{split}
\end{equation*}
where we used \eqref{eq:B1} in the last inequality. It follows that \eqref{V<1} is bounded by 
\begin{align}\nonumber
& \ll P^\varepsilon \max_{\substack{KDC\asymp Q^{1/2}\\ D\ll Y\ll Q/K\\ (Y, W)\in \mathcal C_\delta^-}}  WQ^{-3/2}P^{s}(1+P^2W)^{-(s-1)/2}DKY^{-s/2+1}Y\\
\nonumber&\hspace{5.5cm}\times\left(Y^{1/2}D^{-1/2}C+C^2\right)\\
\nonumber
& \ll P^\varepsilon  \max_{\substack{DC\ll Q^{1/2}\\ D\ll Y\ll Q\\ (Y, W)\in \mathcal C_\delta^-}}   WQ^{-1}P^{s}(1+P^2W)^{-(s-1)/2}Y^{-s/2+2}\\ &\hspace{5.5cm}\times\left(Y^{1/2}D^{-1/2}+C\right).\label{1235}
\end{align}
To estimate the maximal value of this expression for $(Y, W)\in \mathcal C_\delta^-$, we consider the case $Y\ll Q^{1/2-\delta}$ and $Y\gg Q^{1/2-\delta}$ separately. We also choose $\delta$ sufficiently small (depending on $\ve$ and $s$). 
If $Y\ll Q^{1/2-\delta}$, then $\mathcal C_\delta^-$ implies that $P^2W\gg Q^{-\delta}$. As a result, the maximal in \eqref{1235} is achieved when $W\asymp Y^{-1}Q^{-1-\delta}$ is the smallest (as $s\geq 2$), which hands us the bound (when $\delta$ is sufficiently small)
\begin{align*}
&\ll Y^{-1}Q^{-1} Q^{-	1}P^{s+\ve}(1+Q^{1/2}/Y)^{-(s-1)/2}Y^{-s/2+2}(Y^{1/2}D^{-1/2}+C)\\
&\ll P^{s+\ve}Y^{1/2}Q^{-s/4-7/4}(Y^{1/2}D^{-1/2}+C).
\end{align*}
Under the conditions $Y\ll Q^{1/2-\delta}$ and $DC\ll Q^{1/2}$, the maximal is achieved when $Y\asymp Q^{1/2-\delta},C\asymp Q^{1/2}$. Therefore, this contribution to \eqref{1235} can be estimated by 
\begin{align*}
    \ll P^{s+\ve}Q^{-s/4-1}\ll P^{s-(s+4)/3+\ve}\ll P^{s-4-(s-8)/3+\ve}
\end{align*}
when $\delta$ is sufficiently small.
If $Y\gg Q^{1/2-\delta}$,  then \eqref{1235} is maximal when $W\asymp Q^{-3/2}$ and $Y\asymp Q^{1/2-\delta}$ or $W\asymp Y^{-1}P^{-1}$ and $Y\asymp P$, whose contributes to  \eqref{1235} can be bounded by 
\begin{align*}
 &Q^{-5/2}Q^{-s/4+1}P^{s+\ve}(Q^{1/4}+Q^{1/2})+P^{-1}P^{-1} Q^{-1} P^{s+\ve} P^{-s/2+2} (P^{1/2}+Q^{1/2})\\
 &\ll Q^{-s/4-1}P^{s+\ve}+ Q^{-1/2} P^{s/2+\ve}\ll  P^{s-4-(s-8)/3+\ve},
\end{align*}
for $s\geq 9$ when $\delta$ is sufficiently small. We have thus established Lemma \ref{N2splitbounds}\eqref{N20bounds}.

\section{Auxiliary estimates for averages of
\texorpdfstring{$|S_{r,d\uc}|\times \Sigma$}{|S\{r,dc\}|×Σ}.}
\label{n2prelim}
Using Lemma \ref{lem:N2maxbound}, it is enough to consider averages of $S_{r,d\uc}$ and $\Sigma$. 
In this section we give upper bounds for these averages using the estimates for exponential sums and integrals from sections~\ref{sec:exp integral} and~\ref{sec:exp sum}. We provide these upper bounds together with the counting results from Section~\ref{sec:counting} for preparation of Lemma \ref{N2splitbounds}\eqref{N22bounds}, which  will ultimately suffice to prove Lemma~\ref{lem:Ni-bounds}\eqref{new-delta-part}. 

\subsection{Estimates for $\Sigma$}

We will choose an integer parameter $\mathfrak P$ differently depending on whether $\uc$ is a good pair or a bad pair, in order to estimate $\Sigma(r, Y_1,W, \mathfrak P;k,d, \uc, \uu)$.

\begin{lemma}[Type I primes for good $\uc$ and good primes for bad $\uc$]\label{q1sum}
Under the notation of Lemma~\ref{lem:N2maxbound}, we have the following:
\begin{enumerate}
\item \label{q1sumgoodc}Let $\uc$ is good and let $\mathfrak P=dD_F\det M_\uc$. Then unconditionally we have the following bound
\begin{align}
\nonumber
&\Sigma(r, Y_1,W, \mathfrak P; k, d, \uc, \uu)\\
\nonumber
& \ll 
WQ^{-3/2} P^{s+\ve}(1+P^2W)^{-(s-1)/2}(1+P^2W\lambda_\uc)^{-1/2}
\\
\label{eq:uncond}
&\quad\times\begin{cases}
Y_1^{-s/2+3/2} & 2\nmid s,
\\
Y_1^{-s/2+1}&2\mid s, F_\uc^*(\uu)\neq 0,
\\Y_1^{-s/2+2}
& 2 \mid s,  F_\uc^*(\uu)= 0.
\end{cases} 
\end{align}
If we further assume the generalized Lindel\"of hypothesis for Dirichlet L-functions, then for $2\nmid s$ and $F_\uc^*(\uu)\not=\square$, we may obtain
\begin{align}\label{eq:condi}
&\quad\Sigma(r, Y_1,W, \mathfrak P; k, d, \uc, \uu)\\  &\ll
WQ^{-3/2} P^{s+\ve}(1+P^2W)^{-(s-1)/2}(1+P^2W\lambda_\uc)^{-1/2}
Y_1^{-s/2+1}.\nonumber
\end{align}
\item \label{q1sumbadc}Suppose $\uc$ is bad and let $\mathfrak P=dD_F$. Then
we have
\begin{align*}
\MoveEqLeft\Sigma(r, Y_1,W, \mathfrak P; k, d, \uc, \uu)\\ &\ll 
WQ^{-3/2} P^{s+\ve}(1+P^2W)^{-(s-1)/2}
Y_1^{-s/2+3/2+\mathds 1_{2\mid s}/2},
\end{align*}if $Q_\uc^*(\uu')\not =0$ or $((S^{-1})^T\uu)_s\neq 0$, 
and otherwise, 
\begin{align*}
\MoveEqLeft\Sigma(r, Y_1,W, \mathfrak P; k, d, \uc, \uu)
\\ &\ll 
WQ^{-3/2} P^{s+\ve}(1+P^2W)^{-(s-1)/2}
Y_1^{-s/2+5/2+\mathds 1_{2\mid s}/2}.
\end{align*}
\end{enumerate}
\end{lemma}

\begin{proof} We write $\Sigma(r, Y_1)=\Sigma(r, Y_1, W, \mathfrak P;k,d, \uc, \uu)$ for short.

(1) Let $\uc$ be a good pair.
 Using Lemma~\ref{p1integral} with $j=0$, we have 
\begin{align}
\nonumber
 &|\Sigma(r,Y_1)|
 \\
 \nonumber
& \ll \sum_{\substack{q_1\asymp Y_1\\ (q_1, k\mathfrak P)=1}}q_1^{-s} |S_{q_1, \uc}(\uu)| \int_{|\uw| \asymp W}|p_{2, dk\uc, q_1r}(\uw)I_{q_1r}(\uw, \uu)|d\uw\\\nonumber
&\ll  WQ^{-3/2} P^{s}(1+P^2W)^{-(s-1)/2}(1+P^2W\lambda_\uc)^{-1/2}
\\&\qquad \times \sum_{\substack{q_1\asymp Y_1\\ (q_1, k\mathfrak P)=1}}q_1^{-s} |S_{q_1, \uc}(\uu)|.
\label{sigma1ry}
\end{align}
Using Lemma~\ref{lem:type 1}, for $2\mid s$ we have
\begin{align*}
\sum_{\substack{q_1\asymp Y_1\\ (q_1, k\mathfrak P)=1}}q_1^{-s} |S_{q_1, \uc}(\uu)|
&\ll \sum_{\substack{q_1\asymp Y_1\\ (q_1, k\mathfrak P)=1}} q_1^{-s/2}(F_\uc^*(\uu), q_1)\ll Y_1^{-s/2+1+\mathds 1_{F_\uc^*(\uu)=0}+\ve}.
\end{align*}
Similarly, for $2\nmid s, F_\uc^*(\uu)\not=0$ we have
\begin{align*}
\sum_{\substack{q_1\asymp Y_1\\ (q_1, k\mathfrak P)=1}}q_1^{-s} |S_{q_1, \uc}(\uu)|\ll  \sum_{\substack{q_1\asymp Y\\ (q_1, k\mathfrak P)=1}}Y_1^{-s/2+1/2}(F_\uc^*(\uu), q_1)^{1/2}\ll Y_1^{-s/2+3/2+\ve}.
\end{align*}
If $2\nmid s$ and $F_\uc^*(\uu)=0$, we see by Lemma~\ref{lem:type 1} that 
$S_{p^k, \uc}(\uu)=0$ for $2\nmid k$  and 
$S_{p^k, \uc}(\uu)\ll p^{ks/2+k}$ for $2\mid k$. Therefore,
\begin{align*}
\sum_{\substack{q_1\asymp Y_1\\ (q_1, k\mathfrak P)=1}}q_1^{-s} |S_{q_1, \uc}(\uu)|\ll   \sum_{\substack{q_1\asymp Y_1\\ p\mid q_1\Rightarrow p^2\mid q_1 \\(q_1, k\mathfrak P)=1}}Y_1^{-s/2+1+\ve}\ll P^{\ve} Y_1^{-s/2+3/2+\ve}.
\end{align*}
This completes the proof of the unconditional bound \eqref{eq:uncond} for $\Sigma(r, Y_1)$.
 
When $s$ is odd and $F_\uc^*(\uu)\not=\square$, we can improve the estimate for $\Sigma(r, Y_1)$ by using cancellations in the sum  $\sum_{q_1}S_{q_1, \uc}(\uu)$, using Lemma~\ref{hypo:GRH}. To be precise, we use integration by parts to write 
\begin{equation}
\begin{split}
 \Sigma(r, Y_1)={}&\int_{|\uw|\asymp W} 			\Big(\sum_{\substack{Y_1\leq q'\leq q_1\\ (q', k\mathfrak P)=1}}\left(S_{q',\uc}(\uu)\right)\Big)	q_1^{-s} p_{2}(q_1r,\uw) I_{q_1r}(\uw,\uu)\Big\vert_{Y_1}^{2Y_1}
\,d\uw \\ &
-\int_{|\uw|\asymp W} 			\int_{Y_1\leq q_1\leq 2Y_1}\Big(\sum_{\substack{Y_1\leq q'\leq q_1\\ (q',k\mathfrak P)=1}}\left(S_{q',\uc}(\uu)\right)\Big)	\\ &\hspace{4em}
\times \partial_{q_1}\big(q_1^{-s} p_{2}(q_1r,\uw) I_{q_1r}(\uw,\uu)\big)
\,dq_1 d\uw.\label{eqn}
\end{split}
\end{equation}
For any $q_1\ll Y_1$, using Lemma~\ref{hypo:GRH}, we have
\begin{equation}\label{eq:hypobo}
\Big|\sum_{\substack{Y_1\leq q'\leq q_1\\ (q_1, k \mathfrak P)=1}}S_{q',\uc}(\uu)\Big|\ll Y_1^{s/2+1+\ve}.
\end{equation}
The desired bound for the first term on the right hand side of \eqref{eqn} follows upon applying Lemma~\ref{p1integral} to estimate the integral over $\uw$ and combining it with \eqref{eq:hypobo}.
For the second term in \eqref{eqn}, we change the order of integrals and use \eqref{eq:hypobo} to obtain
\begin{equation}
\begin{split}
		\qquad&\int_{Y_1\leq q_1\leq 2Y_1}\Big(\sum_{\substack{Y_1\leq q'\leq q_1\\ (q',k\mathfrak P)=1}}\left(S_{q',\uc}(\uu)\right)\Big)	\\
        &\quad\times\int_{|\uw|\asymp W} 	\partial_{q_1}\big(q_1^{-s} p_{2}(q_1r,\uw)
        I_{q_1r}(\uw,\uu)\big)
\,d\uw\, dq_1\\
&\qquad \ll  Y_1^{\frac{s}{2}+1+\ve}\int_{q_1\asymp Y_1}\Big|    \int_{|\uw|\asymp W} 	\partial_{q_1}\big(q_1^{-s} p_{2}(q_1r,\uw) I_{q_1r}(\uw,\uu)\big)
\,d\uw\Big|\, dq_1.\label{q1w}
\end{split}
\end{equation}
From the chain rule with respect to $q_1$ for a fixed $\uw$, we have
\begin{equation}
\begin{split}\label{12.7}
&\quad \partial_{q_1}\big(q_1^{-s} p_{2}(q_1r,\uw) I_{q_1r}(\uw,\uu)\big)\\&=-q_1^{-s-1} p_{2}(q_1r,\uw) I_{q_1r}(\uw,\uu)+q_1^{-s}I_{q_1r}(\uw,\uu)r\partial_q  p_{2}(q_1r,\uw)\\ & \quad +q_1^{-s} p_{2}(q_1r,\uw)r\partial_{q}I_{q_1r}(\uw,\uu).
\end{split}
\end{equation}
By substituting \eqref{12.7} into \eqref{q1w}, we apply Lemma~\ref{p1integral} to estimate the contribution from the first two terms of \eqref{12.7}, and Lemma~\ref{p1derivative} to estimate the contribution from the final term in \eqref{12.7}. This yields the desired bound for the expression in \eqref{eqn}, under the hypothesis of Lemma~\ref{hypo:GRH}, allowing us to conclude \eqref{eq:condi}.

(2)
When $\uc$ is a bad pair, Lemma~\ref{badc} hands us
\begin{align*}
\MoveEqLeft \sum_{\substack{q_1\asymp Y_1\\ (q_1, dD_F)=1}}q_1^{-s}S_{q_1, \uc}(\uu) \\ & \ll \sum_{q_1\asymp Y_1}
q_1^{-s/2-1/2+\mathds 1_{2\mid s}/2}\gcd(q_1,Q_\uc^*(\uu'))\gcd(q_1,((S^{-1})^T\uu)_s).
\end{align*}
The result follows by summing over $q_1$ depending on whether $Q_\uc^*(\uu')=0$ or $((S^{-1})^T\uu)_s=0$,  along with an application of \eqref{sigma1ry} with $\lambda_\uc=0$.
\end{proof}

\subsection{Averages of $|S_{r,d\uc}|$}

Next, we need to consider the $r$-sum appearing in Lemma~\ref{lem:N2maxbound}. 
An easy bound for the $r$-sum is given in following lemma.\begin{lemma}[Trivial Bound]\label{rsumbound}
Let $A>0$ be some large fixed number. 
For any $d\asymp D, 1\ll R\ll Q$, we have uniformly for $\mathfrak P\ll Q^A$ 
\begin{align*}
\sum_{\substack{d\mid r\asymp R\\ r\mid \mathfrak P^\infty}}|S_{r, d\uc}(\uu)|\ll Q^\ve D^{1/2} R^{s/2+3/2}
\end{align*}
\end{lemma}
\begin{proof}
The lemma follows from Lemma~\ref{weakbound} and that the number of $r\asymp R$ with $r\mid \mathfrak P^\infty$ is $O(Q^\ve)$. 
\end{proof}

\subsubsection{Improved bound for bad $\uc$}

For bad $\uc$, the contributions from good primes are larger in Lemma \ref{q1sum} and our weak bound in Lemma~\ref{rsumbound} is insufficient due the loss of the $d^{1/2}$ factor in Lemma \ref{weakbound},  and we need to use the improved averaged bound from Lemma~\ref{lem:splitting-d}, making use of Lemma \ref{badpairgoodp} for the $d$-sum.

\begin{lemma}
    \label{lem:splitting-d}
Let $\ve>0$. 
    Suppose that $\uc$ is bad. Under the notation of Lemma~\ref{lem:N2maxbound}, for $\delta$ sufficiently small (depending on $\ve$ and $s$) and any choice of $(K, D,C,R, Y, W)\in \mathcal Q_\delta$, we have  
\begin{equation}
\begin{split}
&WQ^{-3/2} P^{s} (1+P^2W)^{(1-s)/2}(Y/R)^{-s/2+3/2+\mathds 1_{2\mid s}/2}\\
&\quad\qquad\times \quad \quad \sum_{d\asymp D}\sum_{k\asymp K}\sum_{\substack{\uu\in\ZZ^s,  |\uu|\leq V\\ Q_\uc^*(\uu')\not=0 \\\text{ or } ((S^{-1})^T\uu)_s\not=0 }}\sum_{\substack{d\mid r\asymp R\\ r\mid (dD_F)^\infty}}R^{-s}|S_{r, d\uc}(\uu)|
\\&
 \ll
 P^{s-4-(s-8-\mathds 1_{2\mid s})/3+\ve}
 .\label{E21}
 \end{split}
\end{equation}
\end{lemma}

\begin{proof}
We further split the $d$ and $r$-sum into $d=d_1d_2d_3$, where we have $\gcd(d_1d_2,D_F)=1$, the $d_i$ are pairwise co-prime, $d_1d_2$ is square-free with $\gcd(d_1,r/d_1)=1$, $d_2^2\mid r$ and $d_3$ consists of numbers whose prime divisors $p$ satisfies $p\mid D_F$ or $p^2\mid d_3$. We also need to write $r=d_1r_3$, with $d_2^2d_3\mid r_3$ and $r_3\mid (d_2d_3D_F)^\infty$. 
We will re-order the sums as 
\begin{equation*}
\sum_{d_2}\sum_{d_3}\sum_{r_3}\sum_{\uu}\sum_{d_1}(\cdots),
\end{equation*}
and split the sums $d_i\asymp D_i, r_3\asymp R_3$ into dyadic ranges.
From Lemma~\ref{badpairgoodp}, we see that for $\uu\not=\vecnull$
\begin{equation}\label{eq:10001}
\sum_{d_1\asymp D_1}|S_{d_1,d_1\uc}(\uu)|\ll D_1^{s/2+3/2}\sum_{d_1\asymp D_1}\gcd(d_1,\uu)^{1/2}\ll D_1^{s/2+5/2+\ve},
\end{equation}
and from Lemma~\ref{weakbound}, we have 
\begin{equation}\label{eq:10002}
|S_{r_3,d_2d_3\uc}(\uu)|\ll D_2D_3 R_3^{s/2+1}\gcd(r_3/d_2d_3,((S^{-1})^t\uu)_s))^{1/2}.
\end{equation} 
Moreover, there are $O(R_3^\ve)$ different values of $r_3$ for fixed $d_2,d_3$ and there are $O(D_2)$ numbers of $d_2\asymp D_2$ and $O(D_3^{1/2+\varepsilon})$ numbers of $d_3\asymp D_3$ for any $\varepsilon>0$. From \eqref{eq:10001} and \eqref{eq:10002} we have 
\begin{align*}
&\sum_{d_2}\sum_{d_3}\sum_{r_3}\sum_{\uu}\sum_{d_1}|S_{d_1r_1r_3,d_1d_2d_3\uc}(\uu)|\\&\ll P^{\ve}  D_1^{s/2+5/2}D_2D_3R_3^{s/2+1}\sum_{d_2,d_3}\sum_{r_3}\sum_{x_1\mid \frac{r_3}{d_2d_3}}\sum_{\substack{0\neq |\uu|\leq V\\ x_1\mid ((S^{-1})^T \uu)_s}}x_1^{1/2}.
\end{align*}
After an application of \eqref{eq:Vsum1} to estimate the $\uu$ sum, this  is 
\begin{align*}
&\ll P^{\ve} D_1^{s/2+5/2}D_2D_3R_3^{s/2+1}\sum_{d_2,d_3}\sum_{r_3}\sum_{x_1\mid r_3/d_2d_3}(x_1^{1/2}V^{s-1}+V^s)\\
& \ll P^{\ve}D_1^{s/2+5/2}D_2^{2}D_3^{3/2}R_3^{s/2+1}V^{s-1}(R_3^{1/2}D_2^{-1/2}D_3^{-1/2}+V)\\
&\ll P^{\ve}  R^{s/2+3/2} D_1D_2^{2}D_3^{3/2}V^{s-1}(D_2^{-1/2}D_3^{-1/2}+VR_3^{-1/2}).
\end{align*}
Substituting back in \eqref{E21}, the total contribution is
\begin{align*}
&\ll P^{\ve}Y^{-s/2+3/2+\mathds{1}_{2\mid s}/2}R^{-\mathds{1}_{2\mid s}/2}WQ^{-3/2}P^{s}(1+P^2W)^{(1-s)/2}\\ & \quad \quad \quad \quad \quad \ \ \times D_1D_2^{2}D_3^{3/2}KV^{s-1}(D_2^{-1/2}D_3^{-1/2}+VR_3^{-1/2})\\
&\ll P^{\ve} Y^{-s/2+3/2+\mathds{1}_{2\mid s}/2}R^{-\mathds{1}_{2\mid s}/2}WQ^{-1}P^{s}(1+P^2W)^{(1-s)/2}\\ & \quad \quad \quad \quad \quad \ \ \times D_2D_3^{1/2}V^{s-1}(D_2^{-1/2}D_3^{-1/2}+VR_3^{-1/2}),
\end{align*}
since $D_1D_2D_3K\asymp Q^{1/2}$.
Note that as $d_2^2d_3\mid r_3$, and also $d_2^2\mid r$ and $d_2\mid d$, this contribution is
\begin{align*}
&\ll P^{\ve} \max_{ \substack{D_2 \geq 1\\ D_2\leq R^{1/2}\\D_2 \leq D}}Y^{-s/2+3/2+\mathds{1}_{2\mid s}/2}R^{-\mathds{1}_{2\mid s}/2}WQ^{-1}P^{s}\\ & \quad \quad \quad \quad \quad \ \ \times (1+P^2W)^{(1-s)/2}V^{s-1}(D_2^{1/2}+V).
\end{align*}
When $2\nmid s$, the contribution is maximized when $K=D_1=D_3=1$ and $D_2\asymp Q^{1/2}$, and it is bounded by
\begin{align*}
&\ll P^{\ve}Y^{-s/2+3/2}WQ^{-1}P^{s}(1+P^2W)^{(1-s)/2}V^{s-1}(Q^{1/4}+V).
\end{align*}
The maximum value is reached when $W=Q^{-1/2+\delta}Y^{-1} $, and therefore $V=Q^{1/4+\ve+\delta}$. Therefore, when $\delta$ is sufficiently small (depending on $\ve$ and $s$) this contribution is
\begin{align*}
&\ll P^{\ve}Y^{-s/2+1/2}Q^{-3/2}P^{s}(Q/Y)^{(1-s)/2}Q^{s/4}
\\&\ll P^{\ve}P^{s}Q^{-s/4-1}\ll P^{s-4-(s-8)/3+\ve},
\end{align*}
which completes the case for odd $s$. When $2\mid s $, our contribution is
\begin{align*}
    &\ll P^{\ve} \max_{ D_2}Y^{-s/2+2}R^{-1/2}WQ^{-1}P^{s}(1+P^2W)^{(1-s)/2}V^{s-1}(D_2^{1/2}+V)\\
    &\ll P^{\ve} Y^{-s/2+2}R^{-1/2}WQ^{-1+1/4}P^{s}(1+P^2W)^{(1-s)/2}V^{s-1},
\end{align*}
upon bounding $D_2^{1/2}+V\ll Q^{1/4+\ve+\delta}$ and using that $\delta$ is sufficiently small. We next note that $d\mid r_3$ and $Y\leq Q/K$ therefore $Y/R\leq Q/(KD)\leq Q^{1/2}$.  Upon further substituting $V=YP^{-1+\ve}(1+P^2W)$, the expression is maximum when $W=Y^{-1}Q^{-1/2+\delta}$ and therefore $V\asymp Q^{1/4+\ve+\delta}$.
By choosing $\delta$ sufficiently small in terms of $\ve$ and $s$, we see that the above is 
\begin{align*}
    &\ll P^{\ve} Y^{-s/2+3/2}WQ^{-1/2}P^{s}(1+P^2W)^{(1-s)/2}V^{s-1},\\
    &\ll P^{\ve} Y^{-s/2+1/2}Q^{-1}P^{s}(Q/Y)^{(1-s)/2}Q^{s/4-1/4}
    \\&\ll P^{s+\ve}Q^{-s/4-3/4}\ll P^{s-4-(s-9)/3+\ve}
\end{align*}
which completes the proof for even $s$.
\end{proof}

\subsubsection{Improved bound for good $\uc$}

For good $\uc$, Lemma~\ref{rsumbound} turns out to be insufficient to compensate for our weaker bounds on exponential sums modulo powers of good primes of type II, as well as powers of bad primes. To address this, we combine Lemma~\ref{lem:type II} and Lemma~\ref{weakbound} with Lemma~\ref{counting} to obtain the following improved bound, which takes advantage of the average over $\uu$.
\begin{lemma}\label{cursum}
Let $\ve>0$. Under the notation of Lemma~\ref{lem:N2maxbound}, for $\delta$ sufficiently small (depending on $\ve$ and $s$) any $(K, D,C,R, Y, W)\in \mathcal Q_\delta$, we have  \begin{equation*}
\begin{split}
&Y^{-s/2+1}R^{-s/2-1}WQ^{-3/2}P^{s} (1+P^2W)^{-(s-1)/2}\\
 &\qquad\times \sum_{\substack{d\asymp D\\ k\asymp K}}\sum_{\substack{\uc\in \ZZ^2\\ \mathrm{good, primitive} \\|\uc|\asymp C}}\sum_{\substack{\uu\in \ZZ^s, |\uu|\leq V}}\sum_{\substack{d\mid r\asymp R\\  r\mid (dD_F\det M_\uc)^\infty}}|S_{r, d\uc}(\uu)| (1+P^2W\lambda_\uc)^{-1/2} \\
&\ll P^{s-4-(s-9)/3-1/15+\ve}.
\end{split}
\end{equation*}
\end{lemma}
\begin{proof}
For any fixed $k, d, \uc, \uu$, we write $r=r_2r_3$ where $r_2\mid \det(M_{\uc})^\infty$, $(r_2, dD_F)=1$ and $r_2$ is free of fifth powers, while $r_3$ consists of the remaining factors, i.e. $r_3$ consists of numbers whose prime factors divide $dD_F$ and all $5$-full numbers. Note that the condition $d\mid r$ implies $d\mid r_3$.
We further split the $r$-sum into $O(R^\ve)$ sums over $r_2,r_3$ with with $r_i\asymp R_i$ with $R_2R_3\asymp R$, so that it is enough to bound
\begin{multline*}
\mathcal S=Y^{-s/2+1}WQ^{-3/2}P^{s}(1+P^2W)^{-(s-1)/2}\\
\times\sum_{\substack{d\asymp D\\ k\asymp K}}\sum_{\substack{\uc\in \ZZ^2, |\uc|\asymp C\\ \text{primitive}\\ \uc \text{ good} }}\sum_{\substack{\uu\in \ZZ^s, |\uu|\leq V}}\sum_{\substack{d\mid r_3\asymp R_3\\ r_3\mid (dD_F\det M_\uc)^\infty}}\\ \sum_{\substack{d\mid r\asymp R\\  r_2\mid (\det M_\uc)^\infty\\ (r_2, r_3dD_F)=1\\ r_2 \text{ fifth power free}}}(R_2R_3)^{-s/2-1}|S_{r, d\uc}(\uu)| (1+P^2W\lambda_\uc)^{-1/2}.
\end{multline*}
By Lemma \ref{lem:multiplicative}, we write
$S_{r,d\uc}(\uu)=S_{r_2,\uc}(\uu)S_{r_3,d\uc}(\uu)$ and using 
Lemma \ref{lem:type II} and Lemma~\ref{weakbound} we obtain
\begin{align}
|S_{r_2,\uc}(\uu)|&\ll r_2^{1+s/2}\gcd(r_2,((S^{-1})^T\uu)_s,Q^*_\uc(\uu'))^{1/2}\label{eq:sr2}\\
|S_{r_3,d\uc}(\uu)| &\ll dr_3^{s/2+1}\gcd(r_3/d,((S^{-1})^T\uu)_s,\det(M_\uc))^{1/2}\label{eq:sr3}.
\end{align}
 We arrange the order of summation as  
\begin{equation*}
\sum_{k}\sum_{d}\sum_{r_3}\sum_{\text{$\uc$}}\sum_{r_2}\sum_{\uu}(\cdots).\end{equation*}
Combining our bounds in \eqref{eq:sr2} and \eqref{eq:sr3}, we see that 
$$
\sum_{\substack{d\mid r_3\\ r_3\asymp R_3}}\sum_{|\uc|\asymp C}\sum_{r_2\asymp R_2}\sum_{\uu}\sum_{\substack{\uu\in \Z^s\\|\uu|\leq V }}(R_2R_3)^{-s/2-1}|S_{r_2r_3,d\uc}(\uu)|(1+P^2W\lambda_\uc)^{-1/2}
$$
has an upper bound
\begin{align*}
    &\ll D\sum_{\substack{r_3\asymp R_3\\ d\mid r_3}}\sum_{|\uc|\asymp C}(1+P^2W\lambda_\uc)^{-1/2}\sum_{r_2\asymp R_2}\sum_{\substack{\uu\in \Z^s\\|\uu|\leq V}} \gcd(r_2,((S^{-1})^T\uu)_s,Q^*_\uc(\uu'))^{1/2}\\ &\hspace{5.5cm}\times \gcd(r_3/d,((S^{-1})^T\uu)_s,\det(M_\uc))^{1/2},
\end{align*}
and hence
\begin{align}\nonumber
&\sum_{\substack{d\mid r_3\\ r_3\asymp R_3}}\sum_{|\uc|\asymp C}\sum_{r_2\asymp R_2}\sum_{\uu}\sum_{\substack{\uu\in \Z^s\\|\uu|\leq V }}(R_2R_3)^{-s/2-1}|S_{r_2r_3,d\uc}(\uu)|(1+P^2W\lambda_\uc)^{-1/2}\\ 
\nonumber
& 
\ll D\sum_{\substack{d\mid r_3\asymp R_3\\ x_1\mid \frac{r_3}{d}}}\sum_{\substack{|\uc|\asymp C\\ x_1\mid \det(M_\uc) }}(1+P^2W\lambda_\uc)^{-1/2}\sum_{\substack{r_2\asymp R_2\\ x_2\mid r_2}}\sum_{\substack{\uu\in \Z^s, |\uu|\leq V\\x_1x_2\mid ((S^{-1})^T\uu)_s\\x_2\mid Q_\uc^*(\uu') }} (x_1x_2)^{1/2}.
\\
&~\quad\ \;\text{ }\text{$ $}
\label{eq:10}
\end{align}
We apply \eqref{eq:Vsum1} with $q=x_1x_2, d=x_2$ to the inner sum of \eqref{eq:10}, 
along with the observation that $\prod_{p\mid x_2}p\gg x_2^{1/4}$, since $x_2$ is free of fifth powers, to obtain
\begin{equation*}
\begin{split}
\sum_{\substack{\uu\in \Z^s, |\uu|\leq V\\x_1x_2\mid ((S^{-1})^T\uu)_s\\x_2\mid Q_\uc^*(\uu') }} 1&\ll V^{s-2}\min\Big(V+\frac{V^2}{x_1x_2}, (1+\frac{V}{x_2^{1/4}})^2\Big)\\ &\ll V^{s-2}+\frac{V^s}{(x_1x_2)^{1/2}}+\min\Big(V^{s-1}, \frac{V^s}{x_2^{1/2}}\Big).
\end{split}
\end{equation*}
We now consider the cases $R_2\leq V^2$ and $R_2\geq V^2$ separately. If $R_2\leq V^2$, we apply \eqref{eq:B1} with $N=P^2W$ to estimate the sum over $\uc$, and obtain
\begin{align*}
&\sum_{\substack{r_3\asymp R_3\\ x_1\mid \frac{r_3}{d}}}\sum_{\substack{|\uc|\asymp C\\ x_1\mid \det(M_\uc) }}(1+P^2W\lambda_\uc)^{-1/2}\sum_{\substack{r_2\asymp R_2\\ x_2\mid r_2}} (x_1x_2)^{1/2}(V^s/(x_1x_2)^{1/2}+V^{s-1})\\
&\ll P^{\ve} \sum_{\substack{r_3\asymp R_3\\ x_1\mid \frac{r_3}{d}}}C(1+C/(x_1(1+P^2W))^{1/2})(V^s+x_1^{1/2}R_2^{1/2}V^{s-1}),
\end{align*}
which is turn has the upper bound
\begin{multline*}
\ll P^{\ve} \Big(\sum_{\substack{r_3\asymp R_3\\ x_1\mid \frac{r_3}{d}}}C^2(V^s+R_2^{1/2}V^{s-1})(1+P^2W)^{-1/2}
\\  \hspace{2cm} +\sum_{\substack{r_3\asymp R_3\\ x_1\mid \frac{r_3}{d}}}C(V^s+x_1^{1/2}R_2^{1/2}V^{s-1})\Big),
\end{multline*}
which is
\begin{multline}
    \ll P^{\ve}\Big((R_3/D)^{1/5}C^2V^s(1+P^2W)^{-1/2} \\
    +(R_3/D)^{1/5}C(V^s+(R_3R_2/D)^{1/2}V^{s-1})\Big). \label{R2<V^2}
\end{multline}
If $R_2\geq V^2$, then the contribution from divisors $x_2\leq V^2$ can again be bound analogously as before. For the contribution from $x_2\geq V^2$, we apply \eqref{eq:B1} with $N=P^2W$ to obtain
 \begin{align*}
\MoveEqLeft\sum_{\substack{r_3\asymp R_3\\ x_1\mid \frac{r_3}{d}}}\sum_{\substack{|\uc|\asymp C\\ x_1\mid \det(M_\uc) }}(1+P^2W\lambda_\uc)^{-1/2}\sum_{\substack{r_2\asymp R_2\\ x_2\mid r_2\\ x_2\gg V^2}} ((x_1x_2)^{1/2}V^{s-2}+x_1^{1/2}V^{s})\\
&\ll\sum_{\substack{r_3\asymp R_3\\ x_1\mid \frac{r_3}{d}}}C(1+C/(x_1(1+P^2W))^{1/2})((x_1R_2)^{1/2}V^{s-2}+x_1^{1/2}V^{s}),
\end{align*}
which is
\begin{multline*}
    \ll P^{\ve}(R_3/D)^{1/5}\Big(C^2(V^s+R_2^{1/2}V^{s-2})(1+P^2W)^{-1/2}\\ +C((R_3/D)^{1/2}V^s+(R_3R_2/D)^{1/2}V^{s-2})\Big),
    \end{multline*}
which is finally
\begin{multline}
\ll P^{\ve} (R_3/D)^{1/5}\Big(C^2(V^s+R_2^{1/2}V^{s-2})(1+P^2W)^{-1/2}\\
+C(R_3R_2/D)^{1/2}V^{s-1}\Big).\label{R_2>V^2}
\end{multline}
In the last inequality, we have used 
\[(R_3/D)^{1/2}V^s\ll (R_3/D)^{1/2}R_2^{1/2}V^{s-1},\]
since $R_2\geq V^2$. Combining \eqref{R2<V^2} and \eqref{R_2>V^2} together, \eqref{eq:10} is bounded by
\begin{align*}
&\ll P^{\ve} D (R_3/D)^{1/5}\Big(C^2(V^s+(R_2R_3/D)^{1/2}V^{s-2})(1+P^2W)^{-1/2}\\& \hspace{4cm} \ \ +C(V^s+(R_3R_2/D)^{1/2}V^{s-1})\Big).
\end{align*} 
After summing over $d,k$, and using the fact that $DKC\asymp Q^{1/2}$, we obtain
\begin{align*}
& DK (R_3/D)^{1/5}\Big(C^2(V^s+(R_2R_3/D)^{1/2}V^{s-2})(1+P^2W)^{-1/2}\\& \hspace{3cm}
+C(V^s+(R_3R_2/D)^{1/2}V^{s-1})\Big)\\
&\ll Q^{1/2}\Big( C(V^s (R_3/D)^{1/5}+C R_2^{1/2}(R_3/D)^{7/10}V^{s-2})(1+P^2W)^{-1/2} \\& \quad \quad \quad \quad  +V^s(R_3/D)^{1/5}+R_2^{1/2}(R_3/D)^{7/10}V^{s-1}\Big),\\
\end{align*}
and this is
\begin{align*}
&\ll Q^{1/2}\Big((Q^{1/2}R^{1/5}V^s+Q^{1/2}R^{7/10}V^{s-2})(1+P^2W)^{-1/2}\\& \quad \quad \quad \quad  +R^{1/5}V^s+R^{7/10}V^{s-1}\Big).
\end{align*}
By taking $V=YP^{-1}(1+P^2W)P^\ve$, we see that the sum $\mathcal{S}$ we need to bound is
\begin{align*}
&\ll  Y^{-s/2+1}WQ^{-1} P^{s+\ve}(1+P^2W)^{-s/2}
\Big(Q^{1/2}R^{1/5}V^s+Q^{1/2}R^{7/10}V^{s-2}\\& \quad \quad \quad \quad +R^{1/5}V^s(1+P^2W)^{1/2}+R^{7/10}V^{s-1}(1+P^2W)^{1/2}\Big)
\\&\ll Y^{s/2+1}WQ^{-1} (1+P^2W)^{s/2}P^{\ve}\Big(Q^{1/2}R^{1/5}+Q^{1/2}R^{7/10}Y^{-2}P^2(1+P^2W)^{-2}\\& \quad \quad \quad \quad +R^{1/5}(1+P^2W)^{1/2}+R^{7/10}Y^{-1}P(1+P^2W)^{-1/2}\Big)\\&\ll Y^{s/2+1}WQ^{-1} (1+P^2W)^{s/2}P^{\ve}\Big(Q^{1/2}R^{1/5}+Q^{1/2}R^{7/10}Y^{-2}P^2(1+P^2W)^{-2}\\& \quad \quad \quad \quad +R^{7/10}Y^{-1}P(1+P^2W)^{-1/2}\Big),
\end{align*}
using the fact that $(1+P^2W)\ll P^\ve (Q/Y)\ll Q$ since $(Y, W)\in \mathcal C_\delta$ and $\delta$ is sufficiently small.  
As the exponents of $W$ and $R$ in the above expression are positive, the maximum over $\mathcal Q_\delta$ is attained when $W\asymp Y^{-1}Q^{-1/2}P^\delta$ which implies $P^2W\asymp (Q/Y)P^\delta$ and $R=Y$. In that maximal case, and when $\delta$ is sufficiently small, the  upper bound above becomes 
\begin{align*}
&\ll Y^{s/2}Q^{-3/2}(Q/Y)^{s/2}P^{\ve}\Big(Q^{7/10}+Q^{-3/2}Y^{7/10}P^2\\ &\hspace{5.5cm} +Y^{7/10}Y^{-1}P(Q/Y)^{-1/2}\Big);
\end{align*}
by noting that the maximum of the latter expression is attained when $Y=Q$, because the resulting exponents of $Y$ are positive. The above is
\begin{align*}
&\ll P^{s+\ve} Q^{-s/4-3/2}\Big(Q^{7/10}+Q^{7/10-1/4}\Big)&&\ll P^{s+\ve}Q^{-s/4-4/5}\\&\ll P^{s-(s/3+16/15)+\ve}&&\ll P^{s-4-(s-9)/3-1/15+\ve},
\end{align*}
by using the relation $Q=P^{4/3}$. 
\end{proof}

\section{Minor arcs contributions: \texorpdfstring{$N_{2,1}(P, \delta)$ and $N_{2,2}(P, \delta)$}{N21(P, δ) and N22(P, δ)}}\label{sec:minor N21}
In this section, we prove Lemma \ref{N2splitbounds}\eqref{N21bounds} and Lemma \ref{N2splitbounds}\eqref{N22bounds} by considering the contributions from bad pairs of $\uc$ in Section~\ref{sec:bad-c}, the contributions from good pairs of $\uc$ when $s$ is even in Section~\ref{sec:good-pairs-unconditional} and when $s$ is odd in Section~\ref{sec:good-pairs-GLH}. 
\subsection{$N_{2,1}(P, \delta)$ contribution from bad pairs \texorpdfstring{$\uc$}{c}}\label{sec:bad-c}

In this section we prove Lemma \ref{N2splitbounds}\eqref{N21bounds} for contributions from bad pairs $\uc$. 
In this case, we know $|\uc|\ll 1$. We henceforth fix a bad pair $\uc$ and estimate the minor arc contribution from this pair.  Note that when $\uc$ is bad, we have $\lambda_{\uc}=0$.
 Recall the terms \(E_{1,1}(P, \delta)\) and \(E_{1,2}(P, \delta)\) from \eqref{E-one}.

\subsubsection{Case: $Q_\uc^*(\uu')\not=0$ or $((S^{-1})^T\uu)_s\not=0$}
\label{sec:bad-c-Q-S-nonzero}
Using Lemma~\ref{q1sum}\eqref{q1sumbadc} which requires $Q_\uc^*(\uu')\not=0$ or $((S^{-1})^T\uu)_s\not=0$, we see that  
\begin{equation*}
\begin{split}
& E_{1,1}(P, \delta)\ll P^{\ve} \max_{ \mathcal Q_\delta, C\ll 1}\sum_{d\asymp D}\sum_{k\asymp K}\sum_{\substack{\uu\in\ZZ^s,  |\uu|\leq V\\ Q_\uc^*(\uu')\not=0 \\\text{ or } ((S^{-1})^T\uu)_s\not=0 }}\sum_{\substack{d\mid r\asymp R \\ r\mid (dD_F)^\infty}}\\&\quad\quad \quad \quad\quad \quad \quad\quad \quad  r^{-s}|S_{r, d\uc}(\uu) \Sigma(r, Y/R, W, \mathfrak P;k,d, \uc, \uu)|\\
&\ll P^{\ve} \max_{\mathcal Q_\delta, C\ll 1}
\sum_{d\asymp D}\sum_{k\asymp K}\sum_{\substack{\uu\in\ZZ^s,  |\uu|\leq V\\ Q_\uc^*(\uu')\not=0 \\\text{ or } ((S^{-1})^T\uu)_s\not=0 }}\sum_{\substack{d\mid r\asymp R\\ r\mid (dD_F)^\infty}}R^{-s}|S_{r, d\uc}(\uu)|\\ 
 &\quad\quad \times WQ^{-3/2} P^{s} (1+P^2W)^{(1-s)/2}(Y/R)^{-s/2+3/2+\mathds 1_{2\mid s}/2}.
\end{split}
\end{equation*}
Using Lemma~\ref{lem:splitting-d},
 when $\delta$ is sufficiently small (depending on $\ve$ and $s$) this contribution is
\begin{align*}
\\&\ll P^{\ve} \max_{ \mathcal Q_\delta, C\ll 1}P^{s}Q^{-s/4-1}\ll P^{s-4-(s-8-\mathds 1_{2\mid s})/3+\ve},
\end{align*}
which suffices for $s\geq 9$. 

\subsubsection{Case: $Q_\uc^*(\uu')=((S^{-1})^T\uu)_s=0 $}
\label{sec:bad-c-Q-S-zero}
Using Lemma~\ref{q1sum}\eqref{q1sumbadc}, we see that  
\begin{align*}
& E_{1,2}(P,\delta)\ll P^{\ve} \max_{ \mathcal Q_\delta, C\ll 1}\sum_{d\asymp D}\sum_{k\asymp K}\sum_{\substack{\uu\in\ZZ^s,  |\uu|\leq V\\ Q_\uc^*(\uu')=((S^{-1})^T\uu)_s=0  }}\sum_{\substack{d\mid r\asymp R\\  r\mid (dD_F)^\infty}}\\ & \quad \quad \quad \quad  \quad \ \ r^{-s}|S_{r, d\uc}(\uu) \Sigma(r, Y/R, W, \mathfrak P;k,d, \uc, \uu)|\\
&\ll P^{\ve} \max_{\mathcal Q_\delta, C\ll 1} \sum_{d\asymp D}\sum_{k\asymp K}\sum_{\substack{\uu\in\ZZ^s,  |\uu|\leq V\\ Q_\uc^*(\uu')=((S^{-1})^T\uu)_s=0}}\sum_{\substack{d\mid r\asymp R\\ r\mid (dD_F)^\infty}}R^{-s}|S_{r, d\uc}(\uu)|\\ 
&\quad\quad \times WQ^{-3/2} P^{s} (1+P^2W)^{(1-s)/2} (Y/R)^{-s/2+5/2+\mathds 1_{2\mid s}/2}.
\end{align*}
Using Lemma~\ref{rsumbound} for the $r$-sum, we see that 
\begin{align*}
E_{1,2}(P, \delta)&\ll P^{\ve} \max_{\mathcal Q_\delta, C\ll 1} Y^{-s/2+5/2+\mathds 1_{2\mid s}/2}WQ^{-3/2}P^{s}(1+P^2W)^{(1-s)/2}\\ & \quad \quad \quad \quad \quad \ \ \times D^{3/2}K R^{-1-\mathds 1_{2\mid s}/2}\sum_{\substack{\uu\in\ZZ^s, |\uu|\leq V\\ Q_\uc^*(\uu')=((S^{-1})^T\uu)_s=0}}1.
\end{align*}
To estimate the sum over $\uu$, we note that the conditions $$Q_\uc^*(\uu')=((S^{-1})^T\uu)_s=0  $$ defines a smooth variety of degree 2 and projective dimension $s-3$, as $Q_\uc^*$ is nonsingular. As a result, using Lemma~\ref{dimgrowth},
\begin{equation*}
\#\{ |\uu|\leq V:Q_\uc^*(\uu')=((S^{-1})^T\uu)_s=0\} \ll V^{s-3+\ve}.
\end{equation*}
Thus 
\begin{align*}
E_{1,2}(P, \delta)&\ll P^{\ve}\max_{\mathcal Q_\delta, C\ll 1}Y^{-s/2+5/2+\mathds 1_{2\mid s}/2}WQ^{-3/2} P^{s} \\ & \quad \quad \quad \quad \quad \ \ \times (1+P^2 W)^{(1-s)/2} D^{3/2}K R^{-1-\mathds 1_{2\mid s}/2} V^{s-3}.
\end{align*}
With $V=\frac{Y}{P}(1+P^2 W)P^\ve$, the maximal is achieved at $W=Y^{-1}Q^{-1/2+\delta}$, which finally gives when $\delta$ is sufficiently small
\begin{align*}
E_{1,2}(P, \delta)&\ll P^{\ve} \max_{\mathcal Q_\delta, C\ll 1}Y^{-s/2+3/2+\mathds 1_{2\mid s}/2}Q^{-3/2}P^{s}(Q/Y)^{(1-s)/2} \\ & \quad \quad \quad \quad \quad \ \ \times D^{1/2}R^{-1-\mathds 1_{2\mid s}/2}Q^{s/4-3/4}\\
&\ll P^{\ve}\max_{\mathcal Q_\delta, C\ll 1} Q^{-s/4-3/4+\mathds 1_{2\mid s}/2}D^{1/2}K^{-1-\mathds 1_{2\mid s}/2}R^{-1-\mathds 1_{2\mid s}/2}P^{s}\\
&\ll Q^{-s/4-1+\mathds 1_{2\mid s}/4}P^{s+\ve}\ll P^{s-4-(s-8-\mathds 1_{2\mid s})/3+\ve},
\end{align*}
using $Y\ll Q/K$, $D\ll Q^{1/2}$ and $RK\gg DK\asymp Q^{1/2}$.
This is admissible as soon as $s\geq 9$ when $\delta$ is sufficiently small in terms of $\ve$ and $s$. 

\subsection{$N_{2,2}(P, \delta)$ contribution from good pairs \texorpdfstring{$\uc$}{c} for even \texorpdfstring{$s$}{s}}\label{sec:good-pairs-unconditional}

In this section we prove Lemma \ref{N2splitbounds}\eqref{N22bounds} when $s$ is even. Recall the sums \(E_{2,1}(P, \delta)\) and \(E_{2,2}(P, \delta)\) as corresponding contributions coming from an application of \eqref{E-two}.
 
\subsubsection{The case $F_\uc^*(\uu)\not=0$}\label{sec:good-c-even-s-F-nonzero}
Using Lemma~\ref{q1sum}\eqref{q1sumgoodc}, we see that \emph{unconditionally} for even $s$ and $\delta$ sufficiently small (depending on $\ve$ and $s$)
\begin{align*}
&E_{2,1}(P, \delta)\ll P^{\ve} \max_{ \mathcal Q_\delta}\sum_{d\asymp D}\sum_{k\asymp K}\sum_{\substack{|\uc|\asymp C\\ \uc \text{ good}}}\sum_{\substack{\uu\in\ZZ^s,  |\uu|\leq V\\ F_{\uc}^*(\uu)\not=0 }}\sum_{\substack{d\mid r\asymp R\\ r\mid (2dD_F\det M_\uc)^\infty}}\\& \hspace{4cm} r^{-s}|S_{r, d\uc}(\uu) \Sigma(r, Y/R, W, \mathfrak P;k,d, \uc, \uu)|\\
\ll & P^{\ve} \max_{\mathcal Q_\delta} \sum_{d\asymp D}\sum_{k\asymp K}\sum_{|\uc|\asymp C}\, \, \,\sum_{\substack{\uu\in\ZZ^s,  |\uu|\leq V\\ F_{\uc}^*(\uu)\not=0}}\sum_{\substack{d\mid r\asymp R\\ r\mid (2dD_F\det M_\uc)^\infty}}R^{-s}|S_{r, d\uc}(\uu)| \\
&\times WQ^{-3/2} P^{s} (1+P^2W)^{(1-s)/2}(1+P^2W\lambda_\uc)^{-1/2} (Y/R)^{-s/2+1},
\end{align*}
which gives 
\begin{align}\label{N221bounds}
    E_{2,1}(P, \delta)\ll P^{s-4-(s-9)/3-1/15+\ve}
\end{align} by
 using Lemma~\ref{cursum}. 

\subsubsection{The case $F_\uc^*(\uu)=0$}\label{sec:good-c-even-s-F-zero}
In this case, Lemma~\ref{q1sum} gives 
\begin{align*}
& E_{2,2}(P, \delta)\ll P^{\ve} \max_{ \mathcal Q_\delta}\sum_{d\asymp D}\sum_{k\asymp K}\sum_{\substack{|\uc|\asymp C\\ \uc \text{ good}}}\sum_{\substack{\uu\in\ZZ^s,  |\uu|\leq V\\F_{\uc}^*(\uu)=0 }}\sum_{\substack{d\mid r\asymp R\\ r\mid (2dD_F\det M_\uc)^\infty}}\\& \hspace{4cm} r^{-s}|S_{r, d\uc}(\uu) \Sigma(r, Y/R, W, \mathfrak P;k,d, \uc, \uu)|\\
\ll & P^{\ve} \max_{ \mathcal Q_\delta} \sum_{d\asymp D}\sum_{k\asymp K}\sum_{\substack{|\uc|\asymp C\\ \uc \text{ good}}}\sum_{\substack{\uu\in\ZZ^s,  |\uu|\leq V\\ F_{\uc}^*(\uu)=0 }}\sum_{\substack{d\mid r\asymp R\\ r\mid (2dD_F\det M_\uc)^\infty}}R^{-s}|S_{r, d\uc}(\uu)| \\&
\times W Q^{-3/2} P^{s} (1+P^2W)^{(1-s)/2}(1+P^2W\lambda_\uc)^{-1/2} (Y/R)^{-s/2+2}.
\end{align*}
Using Lemma~\ref{rsumbound} for the $r$-sum we see that 
\begin{align}\nonumber
& E_{2,2}(P, \delta)\ll P^{\ve}\max_{\mathcal Q_\delta} \sum_{d\asymp D}\sum_{k\asymp K}\sum_{\substack{|\uc|\asymp C\\ \uc \text{ good}}}\sum_{\substack{\uu\in\ZZ^s,  |\uu|\leq V\\ F_{\uc}^*(\uu)=0 }} Y^{-s/2+2} WQ^{-3/2} \\
\nonumber& \quad \quad \quad\quad \quad \times P^{s}(1+P^2W)^{(1-s)/2}(1+P^2W\lambda_\uc)^{-1/2} D^{1/2}R^{-1/2}\\
\nonumber
&\ll P^{\ve} \max_{\mathcal Q_\delta} Y^{-s/2+2} R^{-1/2}D^{3/2} K WQ^{-3/2} P^{s} (1+P^2W)^{(1-s)/2}\\& \quad \quad \quad\quad \quad \times \sum_{\substack{|\uc|\asymp C\\ \uc \text{ good}}}\sum_{\substack{\uu\in\ZZ^s,  |\uu|\leq V\\ F_{\uc}^*(\uu)=0 }} (1+P^2W\lambda_{\uc})^{-1/2}.\label{evenn22}
\end{align}

Now we focus on the $\uc, \uu$ sum. For a fixed good $\uc$, $F^*_\uc(\uu)$ is an irreducible quadratic form in $\uu$ and therefore by~\cite[Theorem~2]{HeathBrown02} we have
\begin{equation}\label{evenusum}
\#\{|\uu|\leq V:F_\uc^*(\uu)=0\}\ll  V^{s-2+\ve}.
\end{equation}
However, this saving in the $V$ variable together with \eqref{eq:B1} will not be enough. Therefore, we further split the sum over $\uu$ into terms for which $\scrF^*(\uu)\neq 0 $ and $\scrF^*(\uu)=0$.

For fixed $\uu$ such that $\scrF^*(\uu)\neq 0$, we note that the polynomial $F^*_\uc(\uu)$ must be a non-zero polynomial of degree $s-1$ in the $\uc$ variable. $\scrF^*(\uu)$ is the discriminant of the polynomial $F^*_\uc(\uu)$, seen as a polynomial in $\uc$.   Therefore,  $F^*_\uc(\uu)$  has at most $s-1$ primitive roots. As a result, we reach an alternate bound:
\begin{equation}\label{evencsum}
\#\{|\uc|\asymp C, \uc \text{ good},  |\uu|\leq V: \scrF^*(\uu)\neq 0,F_\uc^*(\uu)=0\}\ll V^{s+\ve}. 
\end{equation}
Combining \eqref{evenusum} with \eqref{eq:B1} and comparing this bound with \eqref{evencsum}, we have 
\begin{align}\nonumber	 
&\sum_{\substack{ \uc\in\ZZ^2, |\uc| \asymp C\\ \textrm{ good primitive}}}	\sum_{\substack{\uu\in \Z^s, |\uu|\leq V\\ F_\uc^*(\uu)= 0, \scrF^*(\uu)\neq 0}}(1+P^2W\lambda_\uc)^{-1/2}\\
\nonumber
&\ll V^{s-2+\ve}C+V^{s-2+\ve}\min\{ C^2(1+P^2W)^{-1/2},V^2\}\\
\nonumber
&\ll V^{s-2+\ve}C+V^{s-3/2+\ve}C^{3/2}(1+P^2W)^{-3/8}\\
&\ll P^{\ve} C \Big(V^{s-2}+V^{s-3/2}C^{1/2}(1+P^2W)^{-3/8}\Big).\label{eq:105}
\end{align}
Now we consider the case $\mathcal F^*(\uu)=0$. 
If \(\uc\) is good and \(s\geq 8\), then Lemma~\ref{lem:intersection-of-dual-varieties} shows that the variety  $\scrF^*(\uu)=F_\uc^*(\uu)=0$ has projective dimension \(s-3\) and no components of degree 1 or 3. Hence, applying Lemma~\ref{dimgrowth} to each component of $\scrF^*(\uu)=F_\uc^*(\uu)=0$, we find
\begin{align*}
\{|\uu|\leq V:\scrF^*(\uu)=F^*_\uc(\uu)=0 )\}\ll  V^{s-3+\ve}.
\end{align*}
Therefore,  
\begin{multline}\label{eq:104}
\sum_{\substack{ \uc\in\ZZ^2 , |\uc| \asymp C\\ \textrm{ good primitive}}}		 \sum_{\substack{\uu\in \Z^s, |\uu|\leq V\\ F_\uc^*(\uu)= 0, \scrF^*(\uu)= 0}}(1+P^2W\lambda_\uc)^{-1/2}\\ \ll  V^{s-3+\ve}C(1+C(1+P^2W)^{-1/2}).
\end{multline}
Combining \eqref{eq:105} and \eqref{eq:104} we see that 
\begin{align*}
&\sum_{\substack{ \uc\in\ZZ^2, |\uc| \asymp C\\ \textrm{ good primitive}}}			 \sum_{\substack{\uu\in \Z^s, |\uu|\leq V\\ F_\uc^*(\uu)= 0}}(1+P^2W\lambda_\uc)^{-1/2}\\&\ll P^{\ve}C \Big(V^{s-3/2}C^{1/2}(1+P^2W)^{-3/8}+V^{s-2}+V^{s-3}C(1+P^2W)^{-1/2}\Big).
\end{align*}

Substituting these bounds in \eqref{evenn22}, we obtain 
\begin{align*}
&E_{2,2}(P, \delta)\ll P^{\ve}\max_{\mathcal Q_\delta} Y^{-s/2+2}R^{-1/2}WD^{3/2}KCQ^{-3/2}P^{s}(1+P^2W)^{(1-s)/2}\\
&\quad \quad \times \Big(V^{s-3/2}C^{1/2}(1+P^2W)^{-3/8}+V^{s-2}+V^{s-3}C(1+P^2W)^{-1/2}\Big).
\end{align*}
Recall that $DKC\asymp Q^{1/2}, V=YP^{-1} (1+P^2W)P^\ve$ and $W\ll Y^{-1}Q^{-1/2}P^\delta$.  We see that the maximum is reached when $W\asymp Y^{-1}Q^{-1/2}P^\delta, R\asymp 1$ and hence $1+P^2W\asymp P^\delta Q/Y, V\asymp P^{1/3+\ve+\delta}$. Thus, for $\delta$ is sufficiently small (depending on $\ve$ and $s$), we have
\begin{align*}
     E_{2,2}(P, \delta)&\ll P^{\ve}\max_{\mathcal Q_\delta} Y^{-s/2+1}D^{1/2}Q^{-3/2}P^{s}(Q/Y)^{(1-s)/2}C^{1/2}\\
&\times \Big(P^{s/3-1/2}(Q/Y)^{-3/8}+P^{s/3-2/3}+P^{s/3-1}C^{1/2}(Q/Y)^{-1/2}\Big).
\end{align*}
 Further, the maximum is attained when $Y\asymp Q, R\asymp 1$
and upon further using $KDC\asymp Q^{1/2}$ and $\delta$ sufficiently small, we obtain
\begin{align}\nonumber
\MoveEqLeft[2.5] E_{2,2}(P, \delta)\\
\nonumber
&\ll P^\ve\max_{\mathcal Q_\delta} Q^{-s/2+1+1/4}Q^{-3/2}P^{s}
\Big(P^{s/3-1/2}+P^{s/3-2/3}+P^{s/3-1}Q^{1/4}\Big)\\
\nonumber
&\ll
P^{s+\ve} Q^{-s/2-1/4}\Big(P^{s/3-1/2}+P^{s/3-2/3}\Big)
\\
&\qquad\ll P^{s-s/3-2/3-1/6+\ve}\ll  P^{s-4-(s-10)/3-1/6+\ve}.\label{N222bounds}
\end{align}
Combing \eqref{N221bounds} and \eqref{N222bounds}, we obtain Lemma \ref{N2splitbounds}\eqref{N22bounds} when $s$ is even.

\subsection{$N_{2,2}(P, \delta)$ from good pairs \texorpdfstring{$\uc$}{c} for odd \texorpdfstring{$s$}{s}}\label{sec:good-pairs-GLH}

In this section, we prove Lemma \ref{N2splitbounds}\eqref{N22bounds} when $s$ is odd.
Recall that the contributions $E_{2,1}(P,\delta)$ and $E_{2,2}(P,\delta)$ arise from the splitting \eqref{eq:123}.
  We will first assume the generalized Lindel\"of hypothesis (GLH) as in Lemma~\ref{hypo:GRH} for $E_{2,1}(P, \delta)$ and the proof for the unconditional case follows similarly by simple modifications which we shall describe at the end. The estimation for $E_{2,2}(P, \delta)$ is unconditional.  
 
\subsubsection{Case: $F_\uc^*(\uu)\neq \square$}\label{sec:odd_s_nonsquare_Fstar_N211}
 When $F_\uc^*(\uu)\neq \square$, we can use Lemma~\ref{q1sum} to see that for $\delta$ sufficiently small and under GLH \begin{align*}
&E_{2,1}(P, \delta)\ll P^{\ve} \max_{ \mathcal Q_\delta}\sum_{d\asymp D}\sum_{k\asymp K}\sum_{\substack{|\uc|\asymp C\\ \uc \text{ good}}}\sum_{\substack{\uu\in\ZZ^s,  |\uu|\leq V\\ F_{\uc}^*(\uu)\not=\square }}\,\sum_{\substack{d\mid r\asymp R\\ r\mid (2dD_F\det M_\uc)^\infty}}\\ &\quad \quad \quad \quad r^{-s}|S_{r, d\uc}(\uu) \Sigma(r, Y/R, W, \mathfrak P;k,d, \uc, \uu)|\\
\ll& P^{\ve} \max_{\mathcal Q_\delta} Y^{-s/2+1}R^{-s/2-1}WQ^{-3/2} P^{s} (1+P^2W)^{-(s-1)/2}\\&\times \sum_{d\asymp D}\sum_{k\asymp K}\sum_{\substack{|\uc|\asymp C\\ \uc \text{ good}}}\sum_{\substack{\uu\in\ZZ^s,  |\uu|\leq V\\ F_{\uc}^*(\uu)\not=\square }}\sum_{\substack{d\mid r\asymp R\\  r\mid (2dD_F\det M_\uc)^\infty}}|S_{r, d\uc}(\uu)(1+P^2W\lambda_\uc)|^{-1/2},
\end{align*}
which gives
\begin{equation}\label{N221boundsodd}
    E_{2,1}(P, \delta)\ll P^{s-4-(s-9)/3-1/15+\ve}
\end{equation} by 
 using Lemma~\ref{cursum}.  Unconditionally, we see that Lemma \ref{q1sum} gives an extra factor of $Y_1^{1/2}\ll Q^{1/2}=P^{2/3}$ compared to the bound \eqref{N221boundsodd} under GLH, yielding \begin{equation}\label{N221boundsodduncond}
    E_{2,1}(P, \delta)\ll P^{s-4-(s-11)/3-1/15+\ve}
\end{equation} unconditionally.

\subsubsection{Case: $F_\uc^*(\uu)=\square$}\label{sec:Fcusq}
We now consider the sum over $\uu$ for which $F_\uc^*(\uu)=\square$. In this case, Lemma~\ref{q1sum} gives 
\begin{align*}
&E_{2,2}(P, \delta)\ll P^{\ve}  \max_{ \mathcal Q_\delta}\sum_{d\asymp D}\sum_{k\asymp K}\sum_{\substack{|\uc|\asymp C\\ \uc \text{ good}}}\sum_{\substack{\uu\in\ZZ^s,  |\uu|\leq V\\F_{\uc}^*(\uu)=\square }}\sum_{\substack{d\mid r\asymp R\\ r\mid (2dD_F\det M_\uc)^\infty}}\\ &\quad \quad \quad\quad \quad \quad \quad r^{-s}|S_{r, d\uc}(\uu) \Sigma(r, Y/R, W, \mathfrak P;k,d, \uc, \uu)|\\
& \ll P^{\ve} \max_{ \mathcal Q_\delta} \sum_{d\asymp D}\sum_{k\asymp K}\sum_{\substack{|\uc|\asymp C\\ \uc \text{ good}}}\sum_{\substack{\uu\in\ZZ^s,  |\uu|\leq V\\ F_{\uc}^*(\uu)=\square }}\sum_{\substack{d\mid r\asymp R\\  r\mid (2dD_F\det M_\uc)^\infty}}R^{-s}|S_{r, d\uc}(\uu)| \\
&\times WQ^{-3/2} P^{s} (1+P^2W)^{-(s-1)/2}(1+P^2W\lambda_\uc)^{-1/2} (Y/R)^{-s/2+3/2}.
\end{align*}
Using Lemma~\ref{rsumbound} for the $r$-sum we obtain 
\begin{align*}
&E_{2,2}(P, \delta)\ll P^{\ve} \max_{\mathcal Q_\delta} \sum_{d\asymp D}\sum_{k\asymp K}\sum_{\substack{|\uc|\asymp C\\ \uc \text{ good}}}\sum_{\substack{\uu\in\ZZ^s,  |\uu|\leq V\\ F_{\uc}^*(\uu)=\square }} Y^{-s/2+3/2} WQ^{-3/2} P^{s}\\ &\quad \quad \quad \quad  \quad \quad \times (1+P^2W)^{(1-s)/2}(1+P^2W\lambda_\uc)^{-1/2} D^{1/2}\\
&\ll P^{\ve}\max_{\mathcal Q_\delta} Y^{-s/2+3/2} D^{3/2} K WQ^{-3/2} P^{s} \\ &\quad \quad \quad \quad \quad \quad  \times(1+P^2W)^{-(s-1)/2}\sum_{\substack{|\uc|\asymp C\\ \uc \text{ good}}}\sum_{\substack{\uu\in\ZZ^s, |\uu|\leq V\\ F_{\uc}^*(\uu)=\square }} (1+P^2W\lambda_{\uc})^{-1/2}.
\end{align*}

We now estimate the sum over primitive good $\uc$ and $\uu$ such that $F_\uc^*(\uu)=\square$ by considering whether $\mathcal F^*(\uu)=0$ or not. If $\mathcal F^*(\uu)\not=0$, we can bound the $\uc, \uu$ sum in two ways. First, we fix such a $\uu$ and estimate the sum over $\uc$ using~\cite[Theorem 5]{Bro_03}. Here, for a fixed $\uu$, we see $F_\uc^*(\uu)=z^2$ as a polynomial of degree $s-1$ in the $\uc$ variable. Moreover, when $\scrF^*(\uu)\neq 0$, $F_\uc^*(\uu)$ is square-free and therefore this is an irreducible polynomial in $z$. Moreover, $F_\uc^*(\uu)=(z^{(s-1)/2})^2$ would be a homogeneous polynomial in $\uc$ and $z$. Therefore, the hypothesis of~\cite[Theorem 5]{Bro_03} is indeed satisfied and as a result we obtain:
\begin{align}\label{eq:100}
&\quad \sum_{\substack{|\uc|\asymp C, |\uu|\leq V\\ \scrF^*(\uu)\neq 0, F_\uc^*(\uu)=\square }}(1+P^2W\lambda_{\uc})^{-1/2}\\ &\ll \sum_{\substack{|\uu|\leq V\\ \scrF^*(\uu)\neq 0}}\#\{|\uc|
\asymp C, \uc\textrm{ primitive and good, } F_\uc^*(\uu)=\square\}\ll P^{\ve} V^s C.\nonumber
\end{align}
In a different approach, we first fix a good $\uc$ and estimate the sum over $\uu$ using a slight generalization of~\cite[Theorem 2]{HeathBrown02} obtained in~\cite[Lemma 3.7]{V19}:
\begin{equation}\label{eq:more-counting}
\#\{|\uu|\leq V: F_\uc^*(\uu)=z^2:|\uu|\leq V \}\ll  V^{s-1+\ve}.
\end{equation}
Combining it with the estimate for $\sum_\uc (1+P^2W\lambda_\uc)^{-1/2} $ from \eqref{eq:B1} to obtain the second bound:
\begin{equation}\label{eq:101}
\sum_{\substack{|\uc|\asymp C, \uc \text{ good},  |\uu|\leq V\\ \scrF^*(\uu)\neq 0, F_\uc^*(\uu)=\square }}(1+P^2W\lambda_{\uc})^{-1/2}\ll P^\ve\Big( V^{s-1}C+V^{s-1}C^2(1+P^2W)^{-1/2}\Big).
\end{equation}
Combining \eqref{eq:100} and \eqref{eq:101} we see that 
\begin{align}\nonumber
\MoveEqLeft \sum_{\substack{|\uc|\asymp C, \uc \text{ good},   |\uu|\leq V\\ \scrF^*(\uu)\neq 0, F_\uc^*(\uu)=\square }}(1+P^2W\lambda_{\uc})^{-1/2} 
\\ &\ll P^{\ve} V^{s-1}(1+P^2W)^{-1/2}C \min\{V(1+P^2W)^{1/2},C\}.\label{eq:1001}
\end{align}
If $\mathcal F^*(\uu)=0$, then we use Lemma~\ref{dimgrowth} to estimate
\begin{align*}
\#\{|\uu|\leq V:\scrF^*(\uu)=0, F_\uc^*(\uu)=\square\} 
&\ll \#\{|\uu|\leq V:\scrF^*(\uu)=0\}\\ &\ll  V^{s-2+\ve},
\end{align*}
so that 
\begin{equation}\label{eq:1011}
\sum_{\substack{|\uc|\asymp C, \uc \text{ good}, |\uu|\leq V\\ \scrF^*(\uu)= 0,  F_\uc^*(\uu)=\square }}(1+P^2W\lambda_{\uc})^{-1/2}\ll P^{\ve} \Big(V^{s-2}C+V^{s-2}C^2(1+P^2W)^{-1/2}\Big).
\end{equation}
Combining \eqref{eq:1001} and \eqref{eq:1011}, we obtain
\begin{align} \nonumber
E_{2,2}(P, \delta)
&\ll P^{\ve} \max_{\mathcal Q_\delta}Y^{-s/2+3/2}D^{3/2}KWQ^{-3/2} P^{s}(1+P^2W)^{-s/2}
\\ \nonumber
& \quad \quad \quad\quad \quad \times V^{s-1}C\min\{V(1+P^2W)^{1/2},C\}\\ \nonumber
&\quad  +\max_{\mathcal Q_\delta}Y^{-s/2+3/2}D^{3/2}KWQ^{-3/2} P^{s}(1+P^2W)^{-s/2}
\\ \nonumber
& \quad \quad \quad\quad \quad \times V^{s-2}C((1+P^2W)^{1/2}+C)\\ \nonumber
& \ll P^{\ve}  \max_{\mathcal Q_\delta} Y^{-s/2+3/2}D^{1/2}WQ^{-1} P^{s}(1+P^2W)^{-s/2}
\\ \nonumber
& \quad \quad \quad\quad \quad \times C^{1/2}(1+P^2W)^{1/4}V^{s-1/2}\\ \nonumber
& \quad+ \max_{\mathcal Q_\delta}Y^{-s/2+3/2}D^{1/2}WQ^{-1} P^{s}(1+P^2W)^{-s/2}
\\  \nonumber
& \quad \quad \quad\quad \quad \times((1+P^2W)^{1/2}+C)V^{s-2}\\ \nonumber
 &\ll P^{\ve}\max_{\mathcal Q_\delta}Y^{-s/2+3/2}D^{1/2}WQ^{-1} P^{s}(1+P^2W)^{-s/2}\\
& \times (C^{1/2}(1+P^2W)^{1/4}+((1+P^2W)^{1/2}+C)V^{-3/2})V^{s-1/2}.\label{eq:102}
\end{align}
Comparing the powers of $D,C$ and $K$ in the above expression, since $KDC\asymp Q^{1/2}$, the above expression is maximized when $C\asymp Q^{1/2},DK\asymp 1$ and this the above is
\begin{equation*}
\begin{split}
  &\ll P^{\ve} \max_{\mathcal Q_\delta}Y^{-s/2+3/2}WQ^{-1} P^{s}(1+P^2W)^{-s/2}\\& \quad \quad \quad\quad \quad \times (Q^{1/4}(1+P^2W)^{1/4}+((1+P^2W)^{1/2}+Q^{1/2})V^{-3/2})V^{s-1/2}.
\end{split}\end{equation*}
Recalling that $V=YP^{-1} (1+P^2W)P^\ve$ and $W\ll Y^{-1}Q^{-1/2}P^\delta$, we see that \eqref{eq:102} again takes maximum when $W\asymp Y^{-1}Q^{-1/2}P^\delta$, which means $1+P^2W\asymp (Q/Y)P^\delta$ and $V\asymp P^\ve P^\delta Q/P=P^{1/3+\ve+\delta}$. Therefore, for $\delta$ sufficiently small (depending on $\ve$ and $s$) we have 
\begin{align*}
E_{2,2}(P, \delta)&\ll P^{\ve} \max_{\mathcal Q_\delta}Y^{-s/2+1/2}Q^{-3/2} P^{s}(1+P^2W)^{-s/2}\\& \quad \quad \quad\quad \quad \times (Q^{1/4}(1+P^2W)^{1/4}+Q^{1/2}V^{-3/2})V^{s-1/2}\\&\ll P^{\ve} \max_{\mathcal Q_\delta}Y^{-s/2+1/2}Q^{-3/2} P^{s}(Q/Y)^{-s/2}\\& \quad \quad \quad\quad \quad \times (Q^{1/4}(Q/Y)^{1/4}+Q^{1/2}P^{-1/2})P^{s/3-1/6}\\
&\ll P^{\ve}\max_{\mathcal Q_\delta}Y^{-s/2+1/2} P^{s-2}(Q/Y)^{-s/2}\\& \quad \quad \quad\quad \quad \times (P^{1/3}(Q/Y)^{1/4}+P^{1/6})P^{s/3-1/6}.
\end{align*}
Since the power of $Y$ appearing in the above is positive, this contribution is maximum when $Y\asymp Q$ and therefore for $\delta$ small enough
\begin{align}
 E_{2,2}(P, \delta)   & \ll \nonumber
 Q^{-s/2+1/2} P^{s-2}(P^{1/3}+P^{1/6})P^{s/3-1/6+\ve}\\&\ll  P^{s-4-(s-9)/3-1/6+\ve }.\label{N222boundsodd}
\end{align}

Combining \eqref{N221boundsodd} (or \eqref{N221boundsodduncond}) and \eqref{N222boundsodd} we obtain Lemma \ref{N2splitbounds}\eqref{N22bounds} when $s$ is odd under GLH (or unconditionally, respectively).

\section{Heuristic comparison of delta symbol methods}
\label{sec:heuristic}
In this section, we give a heuristic with which to compare our two-dimensional delta symbol to other existing methods. 
Based on the one-dimensional delta symbol and the  two-dimensional Farey dissection over function fields, we consider an $R$-dimensional delta symbol method over \(\QQ\) to be any identity of the type
	\begin{equation}\label{eq:heuristic}
			\delta_{\nn}
			=
			\sum_{1\leq q\leq Q}
			\ {\starsum_{\ua\bmod{q}}}
    \int_{\substack{|\uw|\leq \frac{Q^\epsilon}{qQ^{\eta}}}}
			p_{q,\ua}(\uw)
			e((\ua/q+\uw)\cdot\nn)\,d\uw
			+O_{\eta,N}(Q^{-N}),
	\end{equation}
for some \(0<\eta\leq 1/R\), some explicit smooth functions $p_{q,\ua}$, and for all \(\nn\in\ZZ^R\) and all \(Q,N,\epsilon>0\).
We believe that many practitioners of the circle method have some intuition along the following lines. 
\begin{heuristic}\label{babyheu}
    If we apply \eqref{eq:heuristic} to a sequence of vectors \(\nn\) which typically have size around $M$, then we should choose $Q$ so that $M=Q^{1+\eta}$ holds. In this way, for \(q\asymp Q\), the function $e((\ua/q+\uw)\cdot\nn)$ does not oscillate very much as $\uw$ varies over the domain of integration. It is desirable for the efficacy of this approach that $Q$ should be taken as small as possible.
\end{heuristic}

According to the heuristic above, it may be desired for $Q=M^{1/(1+\eta)}$ to be as small as possible, in other words $\eta$ should be as large as possible.
However we have the restriction $\eta \leq 1/R$ in  \eqref{eq:heuristic}. The reason is that
the numbers \(\ua/q+\uw\) will have to run over essentially the entire unit box \([0,1]^R\) in order to approximate the function \(\delta_{\nn}\) accurately in $\ell^\infty$ norm. By Khinchine's theorem, this requires \(|\uw|\gg_\epsilon q^{-1}Q^{-1/R-\epsilon}\) for any \(\epsilon>0\), and hence \(\eta\leq 1/R\) must hold.
In particular, our form of the  two-dimensional delta symbol in Theorem \ref{thm:delta} with \(\eta =1/2\) has the optimal choice of $Q=M^{2/3}$. 
 
To perform a Kloosterman refinement by taking advantage of cancellations in the sum over $\ua$, one needs to arrange that the function \(p_{q,\ua}\) is the same for many different \(\ua\). We partition \((\ZZ^R/q\ZZ^R)^*\) into  equivalence classes $[\ua]$ such that \([\ua]=[\ub]\) exactly when \(p_{q,\ua}=p_{q,\ub}\) as functions. We can hope to make use of cancellations in the sums \(\sum_{\ub\in [\ua]}e((\ub/q+\uw)\cdot\nn)\) to get additional savings. 
There may be a trivial class $\{\ua:p_{q,\ua}(\uw)= 0 \ \forall \uw\}$ on which $p_{q,\ua}$ vanishes; we shall exclude this class from our analysis as it contributes nothing to \eqref{eq:heuristic}. We let $A$ be the average size of a nontrivial class $[\ua]$, that is
\begin{align*}
A&=\Big(\sum_{1\leq q\leq Q}
\# \mathcal A_q \Big)^{-1}
\Big(\sum_{1\leq q\leq Q} \sum_{C\in\mathcal A_q}
\#C\Big),
    \\ \mathcal A_q &= \{
C:
C=[\ua] \text{ for some }\ua \text{ with } (\ua,q)=1, p_{q,\ua}\neq 0
\}.
\end{align*}
Our Theorem \ref{thm:delta} allows averages over $\ua$ with $A\asymp Q$.
 We mention that the methods in~\cite{NV,HBP,PSW} effectively take the optimal values $\eta=1/R$, $A=\#((\ZZ^R/q\ZZ^R)^*)$, but  they can only be applied to situations where the exponential sum is an absolute square, due to the use of a classical major-arc/minor-arc decomposition rather than the $\delta$-method.
With this notation we give a heuristic for $\delta$-methods in Diophantine problems.
\begin{heuristic}\label{heu}
Suppose we use \eqref{eq:heuristic} as a form of the circle method to count solutions to $\uF=\vec{0}$, where $\uF(\x)=(F_1(\x), \dots, F_R(\x))$ is a system of $R$ polynomials in $\x\in \mathbb Z^s$ with $|\uF(\x)|\ll M$. Then we should take $Q=M^{1/(1+\eta)}$ and the efficacy of \eqref{eq:heuristic} is roughly captured by the quantity $$Q^{s/2}A^{-1/2},$$ where $A$ is as defined above. A smaller value of this quantity suggests a more effective $\delta$-method. In principle, a double Kloosterman refinement utilizing also the average over $q$ might save a another factor of $Q^{1/2}$, although in practice, this saving may be less due to lack of good estimates for short character sums without GLH. Additionally one may want to apply differencing methods to exponential sums over $\uF$, such as van der Corput differencing. In such cases one can instead take $M$ to be an upper bound for a system $\Delta_{\mathbf h} \uF(\x)$ of differenced polynomials.
\end{heuristic}

The reasoning behind Heuristic~\ref{heu} is as follows. Using the delta symbol in \eqref{eq:heuristic} and summing over $\x$, we see that the number of $\x$ with $|\x|\leq P,$ $\uF=\vec{0}$, is essentially given by
\begin{equation*}
    \sum_{q\leq Q}\underset{\hphantom{\ua\bmod{q}}}{\starsum_{\ua\bmod{q}}}
\int_{|\uw|\ll \frac{1}{qQ^{\eta}}}p_{q, \ua}(\uw)\sum_{|\ux|\leq P}e((\ua/q+\uw)\cdot \uF (\mathbf{x}))
			d\uw.
\end{equation*}
After any differencing steps one applies Poisson summation in the $\x$ variable to modulus $q$; the innermost sum becomes $\sum_{\uu} q^{-s}S_q(\ua,\uu)I_q(\uw,\uu)$ for some exponential sums $S_q(\ua, \uu)$ and exponential integrals $I_q(\uw, \uu)$. The integral $I_q(\uw, \uu)$ allows a truncation of the $\uu$ variables up to essentially $\frac{q}{P}(1+M|\uw|)$ (see e.g. Lemma \ref{lem:IQ} in our setting). The zero frequency $\uu=\mathbf 0$ will generally give the main term. For the non-zero frequencies, one may expect $I_q\ll P^s(1+M|\uw|)^{-s/2}$ from stationary phase analysis and the bound $S_q\ll q^{s/2}$, assuming square-root cancellation of exponential sums. If we can further make use of the average over $\ua$ to get a saving of $(\#[\ua])^{-1/2}$, we may find that the non-zero frequencies contribute to a term of size roughly $$q^{s/2}(1+M|\uw|)^{s/2}(\#[\ua])^{-1/2}\leq (Q+M/Q^\eta)^{s/2}(\#[\ua])^{-1/2}.$$ This error term is of size $\gg Q^{s/2}(\#[\ua])^{-1/2}$ as soon as $Q\geq M^{1/(1+\eta)}\gg M^{R/(R+1)}$. This predicts an error in our original problem of size
$$
\sum_{q\leq Q}\ {\starsum_{\ua\bmod{q}}}
\int_{|\uw|\ll \frac{1}{qQ^{\eta}}}p_{q, \ua}(\uw)
O(Q^{s/2}(\#[\ua])^{-1/2})\,d\uw.
$$
We generally expect $\sum_{q}\sum_{\ua}^*\int |p_{q,\ua}(\uw)|\,d\uw \ll Q^\epsilon$ to hold and that we can replace $\#[\ua]$ by its average $A$. Hence the quantity above should be around $Q^{s/2}A^{-1/2}$ as we posit above.
Furthermore, if one caries out a double Kloosterman refinement, one may save another $Q^{1/2}$ from the $q$-sum.

\begin{remark}
    It is tempting to interpret Heuristic~\ref{heu} as a genuine prediction for the size of the error term in the best possible case. 
    But we should remember that the differencing procedures referred to at the end of the heuristic might alter the actual error term. Moreover better than square-root cancellation in exponential sums may be available such as 
     Ramanujan sums, which appear for quadratic forms in an even number of variables.
\end{remark}

We now compare various versions of delta symbols in dimension two in applications to a pair of quadratic forms with the this heuristic, that is $R=2$ and $\uF=(F_1, F_2)$ are quadratic forms so that $M=P^2$.

\begin{example} \label{eg:GRH}
The nested delta symbol of Munshi~\cite{M} uses a value of $Q\asymp P^{3/2}$ with a function $p_{q, \ua}=p_q$ that allows an average over $\ua$ with $A\asymp Q^{5/3}$. The value $Q^{5/3}$ is drawn from equation (7) of Munshi's paper, where $q_1 q_2 $ different values of $\ua$ are averaged with \(q_1\leq P, q_2\leq P^{1/2}\) and the additional condition that $q_1$ divides a certain quadratic form $Q_2$. This congruence condition can be detected using additive characters modulo $q_1$, resulting in $q_1^2q_2 $ values of \(\ua\), which is typically around $P^{5/2}=Q^{5/3}$ values of \(\ua\) as claimed. Therefore, the heuristic suggests that (without introducing a double Kloosterman refinement) we have an error of size $Q^{s/2} A^{-1/2} \asymp Q^{s/2-5/6}$ which is $P^{(3s-5)/4}$.

Our Theorem~\ref{thm:delta} allows us to take $Q=P^{4/3}$ and to make use of $\ua$-averages with $A\asymp Q$. Thus the heuristic leads to an error term of size $Q^{s/2}A^{-1/2}=P^{2(s-1)/3}$, which is strictly smaller than the heuristic error term $P^{(3s-5)/4}$ from Munshi~\cite{M} as soon as $s\geq 8$.
It is possible that further savings can be obtained as indicated in Remark~\ref{rem:square}.

Both the nested $\delta$-method of Munshi and our Theorem \ref{thm:delta} allow a double Kloosterman refinement, that is to make use of averages over $q$. In 
Munshi~\cite{M}, any non-trivial bounds for character sums in the form $\sum_{q_2\leq P^{1/2}}\chi(q_2)$ for Dirichlet characters $\chi$ with conductors up to $P$ would suffice to handle the case $s=11$, which is the one he considers there. In our case, we need to have more than $Q^{-1/4}$ in savings from the sum $\sum_{q\leq Q}\chi(q)$ for Dirichlet characters $\chi$ with conductor of size as large as $Q^{s/2}$ to handle the case when $s=9$. Due to the large size of the conductors of the characters, this saving in the $q$-sum is more difficult to obtain unconditionally compared to that in Munshi~\cite{M}. In both methods, when $s$ is even, one has more than square-root cancellations due to Ramanujan sums when the moduli $q$ are generic, although typically this advantage is balanced against worse bounds for non-generic moduli.

According to Heuristic~\ref{heu}, the smaller size of $Q$ should be considerably more advantageous for equations of more variables or higher degrees, for which $M$ is larger compared to the size of the variables. These will however bring additional complications by comparison with the quadratic case.
\end{example}

\begin{example}
Another point of comparison is the one-dimensional delta symbol method over $\mathbb Q(i)$, applied to the complex number \(n=F_1+iF_2\), where $F_i$ are quadratic forms in variable of size $P$. In the Gaussian integer version of the method, the denominators $q$  would be Gaussian integers $q_1+iq_2$ of absolute value at most $P$. However, to apply this to a pair \(F_1,F_2\) with coefficients and variables in $\mathbb Z$, we must clear denominators to give rational vectors \(\ua/\operatorname{Nm}(q)\), where the denominator now has size $P^2$. Thus in \eqref{eq:heuristic} we must actually take $Q=P^2$. We have $A\asymp q^2$, and so $Q^{s/2}A^{-1/2}\asymp P^{s-2}$. However, this heuristic is not accurate because, for sums that come from problems over $\mathbb Q[i]$, one should not perform Poisson summation modulo $\operatorname{Nm}(q)$ but rather modulo $q$. This allows Browning--Pierce--Schindler~\cite{BPS} to obtain an error term smaller than $P^{s-4}$ by this technique for certain pairs of quadratic forms. 
\end{example}

\begin{ack}
The authors thank Tim Browning for helpful discussions and providing comments on an earlier version. The authors also would like to thank the anonymous reviewers for their valuable comments and suggestions.
\end{ack}

\begin{oa}
    For the purpose of open access, the
author has applied a Creative Commons Attribution (CC-BY) licence to any Author Accepted
Manuscript version arising from this submission.
\end{oa}

\begin{fund}
The second author was supported by a Leverhulme Early Career Fellowship, and also received support from ERC grant 101054336. 
\end{fund}

\appendix

\section{Dependency graph of the main results}
\begin{figure}[H]
    \centering
\begin{tikzpicture}[scale=0.7, every node/.style={transform shape},
	lemma/.style={rectangle, draw=black!80, fill=white, rounded corners, minimum width=2cm, minimum height=0.5cm, align=center},
	theorem/.style={rectangle, draw=black!80, fill=blue!20, rounded corners, minimum width=2cm, minimum height=1 cm, align=center},
	proposition/.style={rectangle, draw=black!80, fill=blue!10, rounded corners, minimum width=2cm, minimum height=.9 cm, align=center},
	application/.style={rectangle, draw=black!80, fill=white, rounded corners, minimum width=2cm, align=center},
	case/.style={rectangle, draw=black!80, fill=gray!15, rounded corners, minimum width=1.2cm, minimum height=0.7cm, align=center,inner sep=1},
	arrow/.style={-Latex, thick},
	every path/.style={rounded corners=8pt}
	]

	\node[theorem] (thm1) at (-6-3,0) {Theorem \ref{thm:application}};
	\node[theorem] (thm12) at (-10-3,0) {Theorem \ref{thm:delta}};
	\node[lemma] (prop51) at (-10-3,-1.5) {Proposition \ref{prop:delta2}};
	\node[proposition] (lem53) at (-6-3,-1.5) {Lemma \ref{lem:Ni-bounds}};
	
	\node[case] (N0) at (-10,-3) {$N_0(P,\delta)$};
	\node[case] (N1) at (-6-1,-3) {$N_1(P,\delta)$};
	\node[case] (N2) at (-1-2,-3) {$N_2(P,\delta)$: Lemma \ref{N2splitbounds}};
	
	\node[lemma] (lem82) at (-10,-4) {Lemma \ref{N_0asymp}};
	\node[lemma] (lem81) at (-13,-4) {Lemma \ref{lem:major1}};
	
	\node[lemma] (lem91) at (-13,-12+1) {Lemma \ref{dimgrowth}\\ dimension growth};
	\node[lemma] (lem94) at (-13.5,-15+0.5) {Lemma \ref{counting}\\ counting lemmas};
	\node[lemma] (lem92) at (-13.5,-13) {Lemma \ref{lem:intersection-of-dual-varieties}\\ dimension control};
	
	\node[case] (Gc) at (-4.7-1,-4) {$N_{2,2}(P, \delta)$\\Good $\uc$};
	\node[case] (Bc) at (-1.6-1,-5) {$N_{2,1}(P, \delta)$\\ Bad $\uc$};
    \node[case] (U0) at (1.1-1,-6) {$N_{2,0}(P, \delta)$\\  $\uu=\boldsymbol{0}$};
	
	\node[case] (even) at (-5.3-1,-5) {$s$ even};
	\node[case] (odd) at (-4-1,-5) {$s$ odd};
    \node[case] (nzero) at (-2.5,-6-1.5) {$Q^*_{\uc}\not=0$ \text{or} \\$(S^{-1})^T_s\not=0$};
    \node[case] (zero) at (-0.7,-6-2.5) {$Q^*_{\uc}$=0 \text{and}\\ $(S^{-1})^T_s=0$};

	\node[case] (evgood) at (-6.4-1,-6) {$F_{\uc}^*\neq 0$};
	\node[case] (evbad) at (-5.1-1,-6) {$F_{\uc}^*=0$};
	
	\node[case] (oddgood) at (-3.7-1,-6) {$F_{\uc}^*\neq \square$};
	\node[case] (oddbad) at (-2.3-1,-6) {$F_{\uc}^*=\square$};
	
	\node[proposition] (lem132) at (-6.5-2.2,-12.5) {Lemma \ref{q1sum}\\ $S_{q_1,\uc}I_{q_1r}$};
	\node[proposition] (lem133) at (-4.1-2.2,-12.5) {Lemma \ref{rsumbound}\\ $S_{r, d\uc}$ general};
    \node[proposition] (lem134) at (-1.7-2.2,-12.5) {Lemma \ref{lem:splitting-d}\\ $ S_{r,d\uc}$ bad $\uc$};
	\node[proposition] (lem135) at (0.7-2.2,-12.5) {Lemma \ref{cursum}\\  $S_{r,d\uc}$ good $\uc$};
	
	\node[lemma] (lem61) at (-13,-7.8+1) {Lemma \ref{lem:IQ} $I_{q}$};
	\node[lemma] (lem62) at (-13,-8.5+1) {Lemma \ref{p2integral} $p_{1,q}I_{q}$};
	\node[lemma] (lem63) at (-13,-9.3+1) {Lemma \ref{p1integral} $p_{2,q}I_q$};
	\node[lemma] (lem64) at (-13,-10+1) {Lemma \ref{p1derivative} $p_{2,q}{\scriptstyle\frac{d}{dq}}I_q$};
	\node[lemma] (lem75) at (-10.5-1,-16) {Lemma \ref{lem:Dqfinal}\\ $D_q$ general};
	\node[lemma] (lem76) at (-7.5,-16) {Lemmas \ref{lem:type 1}, \ref{hypo:GRH}, \ref{badc}\\ Type I $p$ good $\uc$ (GLH)\\and good $p$ for bad $\uc$};
	\node[lemma] (lem711) at (1.6,-16) {Lemma \ref{weakbound}\\ $S_{q, d\uc}$ general};
	\node[lemma] (lem78) at (-1.2,-16) {Lemma \ref{lem:type II}\\Type II $p$ good $\uc$};
	\node[lemma] (lem712) at (-4,-16) {Lemma \ref{badpairgoodp}\\ $S_{p, p\uc}$ bad $\uc$};
	
	\draw[-] (lem53)--(N0);
	\draw[-] (lem53)--(N1);
	\draw[-] (lem53)--(N2);
	\draw[-] (N2)--(Gc);
	\draw[-] (N2)--(Bc);
	\draw[-] (Gc)--(even);
	\draw[-] (Gc)--(odd);
	\draw[-] (even)--(evgood);
	\draw[-] (even)--(evbad);
	\draw[-] (odd)--(oddgood);
	\draw[-] (odd)--(oddbad);
	\draw[-] (N2)--(U0);
    \draw[-] (Bc)--(nzero);
    \draw[-] (Bc)--(zero);
	\draw[arrow] (lem82) -- (N0);
	\draw[arrow] (lem81) -- (lem82);
    \draw[arrow] (lem81) -- (N0);
	\draw[arrow] ($(lem61.north east)-(1pt, 2pt)$) -- (lem82.south);
	\draw[arrow] (lem75) -- (lem82.south);
	
	\draw[arrow] ($(lem61.north east)-(1pt,2pt)$)--($(N1.south west)+(1pt,2pt)$);
	\draw[arrow] ($(lem62.north east)-(1pt, 2pt)$) --($(N1.south west)+(1pt,2pt)$);
	\draw[arrow] (lem75) -- ($(N1.south west)+(1pt,2pt)$);
	\draw[arrow] ($(lem91.north east)-(1pt, 2pt)$) --node[pos=0.5,draw, fill=white, rounded corners=2pt, inner sep=2pt] {$\mathcal F^*=0$}  ($(N1.south west)+(1pt,2pt)$);
	
	\draw[arrow] (lem132.north) -- (evgood.south);
	\draw[arrow] (lem135.north) -- (evgood.south);
	\draw[arrow] (lem132.north) -- (evbad.south);
	\draw[arrow] (lem133.north) -- (evbad.south);
	\draw[arrow] ($(lem91.north east)-(1pt, 2pt)$) -- (evbad.south);
	\draw[arrow] ($(lem92.north east)-(1pt, 2pt)$) --node[pos=0.07,draw, fill=white, rounded corners=2pt, inner sep=2pt] {$\mathcal F^*=0$} (evbad.south);
	\draw[arrow] ($(lem94.north east)-(1pt, 2pt)$)-- (evbad.south);
	
	\draw[arrow] (lem132.north) -- (oddgood.south);
	\draw[arrow] (lem135.north) -- (oddgood.south);
	\draw[arrow] (lem132.north) -- (oddbad.south);
	\draw[arrow] (lem133.north) -- (oddbad.south);
	\draw[arrow] ($(lem91.north east)-(1pt, 2pt)$) -- (oddbad.south);
	\draw[arrow] ($(lem94.north east)-(1pt, 2pt)$) -- (oddbad.south);
	
	\draw[arrow] ($(lem63.north east)+(0pt,-2pt)$) -- ($(lem132.north)+(-10pt,0pt)$);
	\draw[arrow] ($(lem64.south east)+(-1pt, 2pt)$) -- ($(lem132.north)+(-10pt,0pt)$);
	\draw[arrow] (lem76) -- (lem132);
	
	\draw[arrow] ($(lem711.north west)+(5pt, 0pt)$) -- (lem133.south);
    \draw[arrow] ($(lem711.north west)+(5pt, 0pt)$)--(lem134.south);
    \draw[arrow] (lem712)--(lem134.south);
    \draw[arrow] ($(lem94.south east)+(-1pt, 2pt)$)--(lem134.south);

	\draw[arrow] (lem78) -- (lem135.south);
	\draw[arrow] ($(lem711.north west)+(5pt, 0pt)$) -- (lem135.south);
	\draw[arrow] ($(lem94.south east)+(-1pt, 2pt)$) -- (lem135.south);
	
	\draw[arrow] (lem132.north) --(zero.south west);		
	\draw[arrow] (lem133.north) -- (zero.south west);		
	\draw[arrow] ($(lem91.north east)-(1pt, 2pt)$) --(zero.south west);		
	\draw[arrow] (lem132.north)--(nzero.south);		
	\draw[arrow] (lem134) -- (nzero.south);		

    \draw[arrow] (lem63)--(U0);
    \draw[arrow] (lem711)-- (U0);

	\draw[arrow] (thm12)  -- (prop51);
	\draw[arrow] (prop51) -- (lem53);
	\draw[arrow] (lem53) -- (thm1);

\end{tikzpicture}
 \caption{Main results for the proof of Theorem~\ref{thm:application}.
 }
    \label{fig:Proof structure}
\end{figure}

\bibliographystyle{plain}

\end{document}